\documentclass[leqno,11pt]{amsart}
\usepackage{amsmath,amscd,amsthm,amsxtra}
\usepackage{epsfig,graphics,color,colortbl}
\usepackage{amssymb,latexsym}
\usepackage{mathabx}
\usepackage{mathrsfs,eucal,upgreek}
\usepackage{hyperref}
\usepackage{tikz-cd}
\usepackage{arydshln}
\usepackage{ytableau}
\usepackage{adjustbox}
\usepackage{listings}
\usepackage{dirtytalk}
\usepackage[draft]{todonotes}
\usepackage{color}
\usepackage[color, poly,all]{xy}
\usepackage{stackengine}
\stackMath
\usetikzlibrary{arrows}

\usepackage{dynkin-diagrams}

\newtheorem{thm}{\bf Theorem}[section]
\newtheorem{df}[thm]{\bf Definition}
\newtheorem{prop}[thm]{\bf Proposition}
\newtheorem{cor}[thm]{\bf Corollary}
\newtheorem{lem}[thm]{\bf Lemma}
\newtheorem{rem}[thm]{\bf Remark}
\newtheorem{ex}[thm]{\bf Example}

\numberwithin{equation}{section}

\makeatletter
\renewcommand{\tocsection}[3]{%
  \indentlabel{\@ifnotempty{#2}{\bfseries\ignorespaces#1 #2.\quad}}\bfseries#3}
\renewcommand{\tocsubsection}[3]{%
  \indentlabel{\@ifnotempty{#2}{\ignorespaces#1 #2.\quad}}#3}

\newcommand\@dotsep{4.5}
\def\@tocline#1#2#3#4#5#6#7{\relax
  \ifnum #1>\c@tocdepth 
  \else
    \par \addpenalty\@secpenalty\addvspace{#2}%
    \begingroup \hyphenpenalty\@M
    \@ifempty{#4}{%
      \@tempdima\csname r@tocindent\number#1\endcsname\relax
    }{%
      \@tempdima#4\relax
    }%
    \parindent\z@ \leftskip#3\relax \advance\leftskip\@tempdima\relax
    \rightskip\@pnumwidth plus1em \parfillskip-\@pnumwidth
    #5\leavevmode\hskip-\@tempdima{#6}\nobreak
    \leaders\hbox{$\m@th\mkern \@dotsep mu\hbox{.}\mkern \@dotsep mu$}\hfill
    \nobreak
    \hbox to\@pnumwidth{\@tocpagenum{\ifnum#1=1\bfseries\fi#7}}\par
    \nobreak
    \endgroup
  \fi}
\AtBeginDocument{%
\expandafter\renewcommand\csname r@tocindent0\endcsname{0pt}
}
\def\l@subsection{\@tocline{2}{0pt}{2.5pc}{5pc}{}}
\makeatother

\newcommand{\ag}{\hat{\mathfrak g}}
\newcommand{\aI}{\hat{\text{I}}}
\newcommand{\g}{{\mathfrak g}}
\newcommand{\U}{U_q'(\hat{\mathfrak g})}
\newcommand{\Esix}{\text{E}_6^{(1)}}
\newcommand{\Eseven}{\text{E}_7^{(1)}}
\newcommand{\te}{\tilde{e}}
\newcommand{\tf}{\tilde{f}}
\newcommand{\weyl}{{W}}
\newcommand{\aweyl}{\widehat{{W}}}
\newcommand{\eweyl}{\widetilde{{W}}}
\newcommand{\B}{{\bf B}}
\newcommand{\Asix}{\Delta^{\text{J}}_{6, r}}
\newcommand{\Asixa}{\Delta^{\text{J}}_{6, 1}}
\newcommand{\Asixb}{\Delta^{\text{J}}_{6, 6}}
\newcommand{\Asev}{\Delta^{\text{J}}_7}
\newcommand{\biJ}{{\bf i}^{\text{J}}}
\newcommand{\wJ}{w^{{\text{J}}}}
\newcommand{\mf}{\mathfrak}
\newcommand{\J}{\text{J}}
\newcommand{\I}{\text{I}}

\begin{document}

\title[A realization of KR crystals for type E arising from translations]
{A combinatorial realization of Kirillov-Reshetikhin crystals for type E arising from translations}

\author{IL-SEUNG JANG}

\address{Department of Mathematical Sciences, Seoul National University, Seoul 08826, Republic of Korea}
\email{is\_jang@snu.ac.kr}


\thanks{This work is supported by the National Research Foundation of Korea(NRF) grant funded by the Korea government(MSIT) (No. 2020R1A5A1016126)}

\begin{abstract}
The main purpose of this paper is to give a combinatorial realization of Kirillov-Reshetikhin (KR simply) crystals $B^{r, s}$ for type $\text{E}_n^{(1)}$ with a minuscule node $r$ and $s \ge 1$. To do this, we describe explicitly the crystal of the quantum nilpotent subalgebra associated with the translation by the negative of the $r$-th fundamental weight. Then the crystal can be extended as an affine crystal, in which a certain subcrystal characterized by the $\varepsilon_r^*$-statistic is isomorphic to $B^{r,s}$ as an affine crystal, where $\varepsilon_r^*$ is also realized precisely in terms of triple and quadruple paths.
\end{abstract}

\maketitle
\setcounter{tocdepth}{1}
\tableofcontents

\section{Introduction}

Let $\g$ be a finite-dimensional simple Lie algebra over $\mathbb{C}$.
Let $\ag$ be an affine Lie algebra corresponding to $\g$ of untwisted type and let $\U$ be the quantum affine algebra associated to $\ag$ without the derivation. The finite-dimensional irreducible $\U$-modules (of type $1$) are classified by the Drinfeld polynomials \cite{CP95} and it is well-known that a certain class of such modules known as {\em Kirillov-Reshetikhin modules} plays an important role in the category of the finite-dimensional $\U$-modules (see, for example, \cite{CH}).

Let $W^{(r)}_{s, a}$ be the KR module associated to $1 \le r \le n$, $s \in \mathbb{Z}_+$ and $a \in \mathbb{C}^{\times}$.
It was conjectured by Hatayama et al. \!\!\cite{HKOTY99} that for $1 \le r \le n$ and $s \in \mathbb{Z}_+$, there exists $a_{r, s} \in \mathbb{C}^\times$ such that $W^{(r)}_{s, a_{r, s}}$ has the crystal pseudo-base introduced in \cite{KKMMNN92b} (cf. \cite{Kas91}).
The conjecture have been proved for $r$ in the orbit of or adjacent to $0$ in all affine types \cite{KKMMNN92a,KKMMNN92b}, all non-exceptional types \cite{OS08}, types $\text{G}_2^{(1)}$ and $\text{D}_4^{(3)}$ \cite{Nao18}, and types $\text{E}_{6,7,8}^{(1)}$, $\text{F}_4^{(1)}$, $\text{E}_6^{(2)}$ with near adjoint nodes \cite{NaoScr21}.
Let $B^{r, s}$ be the crystal of the KR module, which is often called {\em KR crystal} for short.
Then it is an interesting problem to describe the structure of $B^{r, s}$ in a combinatorial way.

As the main result of this paper, we give a combinatorial realization of the KR crystal $B^{r,s}$ for type $\text{E}_n^{(1)}$ with a minuscule node $r$ and $s \ge 1$ following the approach in \cite{JK19}. 

Let us explain it in more details. 
Let $B(\infty)$ be the crystal associated to $U_q^-({\mf g})$, and let $B(\Lambda)$ be the crystal of an irreducible highest weight $U_q({\mf g})$-module with highest weight $\Lambda$, where $\Lambda$ is an integral dominant weight of ${\mf g}$.
It was known in \cite{Kas95} that there is a crystal embedding from $B(\Lambda)$ into $B(\infty)$, and the image of $B(\Lambda)$ is described by using the $\varepsilon_i^*$-statistic.
Therefore, one can describe the crystal $B(\Lambda)$ in $B(\infty)$. By realizing $B(\infty)$ explicitly (see Section \ref{sec:PBW crystals}) in terms of {\em PBW basis} \cite{Lu90, Sai}, we have a tensor product decomposition of $B(\infty)$, namely,
\begin{equation} \label{eq:intro decomp}
	B(\infty) \cong \B^\J \otimes \B_\J,
\end{equation}
where $\B^\J$ and $\B_\J$ are subcrystals of $B(\infty)$ given in \eqref{eq:B^J and B_J}. By \eqref{eq:intro decomp} and the characterization of $B(s\varpi_r)$ in $B(\infty)$, we also have
\begin{equation*}
\begin{tikzpicture}
	\node (1) at (-1, 0) {$B(s\varpi_r)$};
	\node (2) at (2, 0) {$B(\infty)$};
	\node (3) at (2, -1.5) {\,\,$\B^\J$};
	\node (4) at (2, -0.75) {$\bigcup$};
	\draw[right hook->] (1) -- (2);
	\draw[right hook->] (1) -- (3);
\end{tikzpicture}
\end{equation*}

In other words, the crystal $\B^\J$ can be viewed as a limit of $B(s\varpi_r)$ as $s \rightarrow \infty$. Moreover, it can be extended as an affine crystal by giving $0$-crystal operator explicitly (see Section \ref{sec:polytope of KR crystals}). Then we denote by $\B^{\J,s}$ the image of $B(s\varpi_r)$ in $\B^\J$, which is an affine crystal.
It was well-known in \cite{Cha01} that $B^{r,s} \cong B(s\varpi_r)$ classically, that is, by ignoring $0$-crystal operators, for a minuscule $r$. Then we will prove that
\begin{equation*}
	\B^{\J, s} \cong B^{r, s},
\end{equation*}
which is the main result of this paper. In the proof, the uniqueness of KR crystals for type E \cite{JS} (cf. Remark \ref{rem:uniqueness E7}) plays a crucial role (see also \cite{FOS} for non-exceptional types).

We should remark that there are several realizations of $B(\infty)$. For example, a 
{\em polyhedral realization} of $B(\infty)$ was known by Nakashima-Zelevinsky in \cite{NZ}. It would be interesting to interpret $\B^\J$ and $\B^{\J,s}$ in this realization as affine crystals (cf. Remark \ref{rem:polytope}).

The motivation of this work is an observation that a reduced expression ${\bf i}^{\J}$ obtained from $t_{-\varpi_r}$ is simply braided in the sense of \cite{SST} up to $2$-term braid moves, which was essentially introduced by Littelmann \cite{Li98} (cf.~\cite{SST}). This yield a nice reduced expression of the longest element $w_0$ of Weyl group by which it enables us to describe the structure of $B(\infty)$ explicitly. In particular, the crystal operators of $B(\infty)$ are given as the signature rule induced from the tensor product rule of crystals, and we also obtain an explicit formula of the $\varepsilon_r^*$-statistic on $\B^{\J} \subset B(\infty)$ in terms of a combinatorial object in Definition \ref{df:paths}, which is called path.

A combinatorial model of the KR crystals $B^{r, s}$ in terms of tableaux was already known in \cite{JS} for type $\Esix$ with $r = 1, 2, 6$ and in \cite{BS} for type $\Eseven$ with $r = 7$.
Both of them use a promotion operator to define the $0$-crystal operators as in \cite{Shi} for type $\text{A}_n^{(1)}$ (cf. \cite{FOS}).
On the other hand, there exists another model of $B^{r, s}$ without the promotion operator for type $\text{A}_n^{(1)}$ and types $\text{D}_{n+1}^{(1)}$, $\text{C}_n^{(1)}$, $\text{D}_n^{(1)}$ with exceptional nodes $r$ \cite{Kw13} by using the RSK correspondence \cite{Knu}.

It would be interesting to ask whether there exists a tableau model of $B^{r, s}$ without a promotion operator for type $\text{E}_n^{(1)}$ and then an isomorphism of affine crystals between our model and a (previous) tableau model in type $\text{E}_n^{(1)}$. 
We would like to remark that it was proved in \cite{Kw13} that the RSK correspondence can be extended to the isomorphisms of affine crystals of type $\text{A}_n^{(1)}$ between the tableau and polytope model of $B^{r, s}$.
Recently, it is also shown in \cite{JK19} that Burge correspondence \cite{Bur}, which is an analog of the RSK for type $\text{D}_n$, can be extended to an isomorphism of affine crystals of type $\text{D}_n^{(1)}$ between the tableau model of $B^{n,s}$ in \cite{Kw13} and the polytope model of $B^{n, s}$.

On the other hand, it is still an open problem to describe KR crystals for all types in a uniform way. 
There are several results by many authors related to this problem, see \cite{LNSSS, LT, LL} and references therein. We hope to find a connection with the results to understand the structure of KR crystals more deeply and extend the approach of this paper beyond minuscule cases if it is possible.

This paper is organized as follows.
In Section \ref{sec:preliminaries}, we review the necessary background on the crystals.
In Section \ref{sec:simply rex}, we give the simply braided reduced expressions for types ADE.
In Section \ref{sec:PBW crystals}, for type $\text{E}_n$ with $n = 6, 7$, we describe the crystal operators of the crystal of $U_q^-({\mf g})$ and then we obtain the crystal $\B^{\J}$ of the quantum nilpotent subalgebra associated to $t_{-\varpi_r}$.
In Section \ref{sec:polytope of KR crystals}, we give a combinatorial realization of the KR crystal $B^{r, s}$ for type $\text{E}_n^{(1)}$ with a minuscule $r$ (Theorem \ref{thm:main theorem}). 
In Section \ref{subsec:epsilon}, we provide an explicit combinatorial formula of $\varepsilon_r^*$ on $\B^{\J}$ in terms of the triple and quadruple paths (Theorem \ref{thm:formula of epsilon star}). 
\vskip 3mm 
\noindent{\bf Acknowledgement.} 
The author is very grateful to Prof.~Jae-Hoon Kwon for many stimulating, helpful discussions and encouraging him in the preparation of this paper; to Prof.~Euiyong Park for his interest in this work and valuable comments. 
He would like to thank Prof.~Cristian Lenart for letting him know the references \cite{LNSSS, LL, LT} with his kind explanation and suggesting possible directions, and Travis Scrimshaw for his interest in this work, stimulating discussions, and letting him know the references \cite{Hiro, Ste}. Finally, he also would like to thank the anonymous referee for helpful and detailed comments.


\section{Preliminaries} \label{sec:preliminaries}

\subsection{Crystals}

Let us give a brief review on crystals (see \cite{HK,Kas91,Kas95} for more details). Let $\mathbb{Z}_+$ denote the set of non-negative integers. Let $\g$ be the Kac-Moody algebra associated to a symmetrizable generalized Cartan matrix $A =(a_{ij})_{i,j\in \I}$ with an index set $\I$. Let $P^\vee$ be the dual weight lattice, $P = {\rm Hom}_\mathbb{Z}( P^\vee,\mathbb{Z})$ the weight lattice, $Q$ the root lattice, $\Pi^\vee=\{\,h_i\,|\,i\in I\,\}\subset P^\vee$ the set of simple coroots, and $\Pi=\{\,\alpha_i\,|\,i\in I\,\}\subset P$ the set of simple roots of $\g$ such that $\langle \alpha_j,h_i\rangle=a_{ij}$ for $i,j\in \I$. Let $P^+$ be the set of integral dominant weights and let $\varpi_i$ be the $i$th fundamental weight for $i \in \I$. 
For an indeterminate $q$, let $U_q(\g)$ be the quantized enveloping algebra of $\g$ generated by $e_i$, $f_i$, and $q^h$ for $i\in \I$ and $h\in P^\vee$ over $\mathbb{Q}(q)$. We denote by $U_q^{-}(\g)$ the subalgebra of $U_q(\g)$ generated by $f_i$ for all $i \in \I$.

A {\it $\g$-crystal} (or {\it crystal} if there is no confusion on $\g$) is a set $B$ together with the maps ${\rm wt} : B \rightarrow P$, $\varepsilon_i,\, \varphi_i: B \rightarrow \mathbb{Z}\cup\{-\infty\}$ and $\te_i,\, \tf_i: B \rightarrow B\cup\{{\bf 0}\}$ for $i\in \I$ satisfying certain axioms. We denote by $B(\infty)$ the crystal associated to $U^-_q(\g)$.
Let $\ast$ be the $\mathbb{Q}(q)$-linear anti-automorphism of $U_q(\g)$ such that $e_i^\ast =e_i$, $f_i^\ast=f_i$, and $(q^h)^\ast =q^{-h}$ for $i\in \I$ and $h\in P$. Then $\ast$ induces a bijection on $B(\infty)$. For $i\in \I$, we define $\te_i^\ast =\ast \circ \te_i \circ \ast$ and $\tf_i^\ast =\ast \circ \tf_i \circ \ast$ on $B(\infty)$. 
In particular, $\varepsilon_i^*(b)$ is defined by $\max \{ \, n \, | \, \te_i^{*n} b \neq {\bf 0} \, \}$.
For $\Lambda\in P^+$, we denote by $B(\Lambda)$ the crystal associated to an irreducible highest weight $U_q(\g)$-module $V(\Lambda)$ with highest weight $\Lambda$.
For $\mu\in P$, let $T_\mu=\{t_\mu\}$ be a crystal, where ${\rm wt}(t_\mu)=\mu$, and $\varphi_i(t_\mu)=-\infty$ for all $i\in \I$.

For crystals $B_1$ and $B_2$, the {\em tensor product $B_1 \otimes B_2$} is defined to be $B_1 \times B_2$ as a set with elements denoted by $b_1 \otimes b_2$, where
{\allowdisplaybreaks
\begin{gather*}
	 {\rm wt}(b_1 \otimes b_2) = {\rm wt}(b_1) + {\rm wt}(b_2), \\
	 \varepsilon_{i}(b_{1} \otimes b_{2}) = \max\{ \varepsilon_{i}(b_{1}), \varepsilon_{i}(b_{2})-\langle {\rm wt}(b_{1}), h_i \rangle \}, \\
	 \varphi_{i}(b_{1} \otimes b_{2}) = \max\{ \varphi_{i}(b_{1})+\langle {\rm wt}(b_2), h_i \rangle, \varphi_{i}(b_{2}) \}, 
\end{gather*}
}
\begin{equation} \label{eq:tensor_product_rule}
\begin{split}
	 \tilde{e}_{i}(b_{1} \otimes b_{2}) = \left\{ \begin{array}{cc} \tilde{e}_{i}b_{1} \otimes b_{2} & \textrm{if} \ \varphi_{i}(b_{1}) \ge \varepsilon_{i}(b_{2}), \\ b_{1} \otimes \tilde{e}_{i}b_{2} & \textrm{if} \ \varphi_{i}(b_{1}) < \varepsilon_{i}(b_{2}), \end{array} \right.
	 \\
	 \tilde{f}_{i}(b_{1} \otimes b_{2}) = \left\{ \begin{array}{cc} \tilde{f}_{i}b_{1} \otimes b_{2} & \textrm{if} \ \varphi_{i}(b_{1}) > \epsilon_{i}(b_{2}), \\ b_{1} \otimes \tilde{f}_{i}b_{2} & \textrm{if} \ \varphi_{i}(b_{1}) \le \epsilon_{i}(b_{2}),\end{array} \right.
\end{split}
\end{equation}
for $i \in \I$. Here, we assume that $\textbf{0} \otimes b_{2} = b_{1} \otimes \textbf{0} = \textbf{0}$. Then $B_1 \otimes B_2$ is a crystal.

For two crystals $B$ and $B'$, an {\em isomorphism of crystals} between $B$ and $B'$ is a bijection $\Psi : B \cup \{ {\bf 0} \} \rightarrow B' \cup \{ {\bf 0} \}$ such that $\Psi({\bf 0}) = {\bf 0}$, $\varepsilon_i(\Psi(b)) = \varepsilon_i(b)$, $\varphi_i(\Psi(b)) = \varphi_i(b)$, $\tilde{e}_i\Psi(b) = \Psi(\tilde{e}_i b)$ and $\tilde{f}_i \Psi(b) = \Psi(\tilde{f}_i b)$ for all $i \in \I$.

\subsection{Quantum nilpotent subalgebra $U_q^-(w)$}
Let $\weyl$ be Weyl group of $\g$ generated by the simple reflections $s_i$ $(i \in \I)$. For $w \in \weyl$, let $R(w)$ be the set of reduced expressions of $w$, that is, 
\begin{equation*}
	R(w)
	=
	\left\{\,
		{\bf i} = (i_1, \dots, i_m) \, \mid \, w = s_{i_1} \dots s_{i_m} \, \text{ and $m$ is minimal\,}
	\right\}.
\end{equation*}
Let $\Phi^+$ be the set of positive roots of $\g$.
It is well-known that for ${\bf i} = (i_1, \dots, i_m) \in R(w)$, 
\begin{equation} \label{eq:beta_k}
 \left\{
 	\beta_1 := \alpha_{i_1}\,, \,
	\beta_2 := s_{i_1}(\alpha_{i_2})\,, \cdots\,,\,
	\beta_m := s_{i_1}\dots s_{i_{m-1}}(\alpha_{i_m})\,
 \right\} \subset \Phi^+,
\end{equation} (see \cite{Pa} and reference therein).

For each $i \in \I$, let $T_i$ be the $\mathbb{Q}(q)$-algebra automorphism of $U_q(\g)$, where $T_i$ is equal to $T_{i, 1}^{''}$ in \cite{Lu10}.
For $w \in \weyl$ and ${\bf i} = (i_1, \dots, i_m) \in R(w)$,  put
\begin{equation*} 
	f_{\beta_k} = T_{i_1} \dots T_{i_{k-1}}(f_{i_k}) \quad (1 \le k \le m),
\end{equation*}
and for ${\bf c} = (c_{\beta_1}, \dots, c_{\beta_m}) \in \mathbb{Z}_+^m$, 
\begin{equation*} 
	b_{\bf i}({\bf c}) = f_{\beta_1}^{(c_{\beta_1})} f_{\beta_2}^{(c_{\beta_2})} \cdots f_{\beta_m}^{(c_{\beta_m})},
\end{equation*}
where $f_{\beta_k}^{(c_{\beta_k})}$ is the divided power of $f_{\beta_k}$ for $1 \le k \le m$.
Then we denote by $U_q^-(w)$ the $\mathbb{Q}(q)$-subspace of $U_q^-(\g)$ generated by $\{ \, b_{\bf i}({\bf c}) \, | \, {\bf c} \in \mathbb{Z}^m \, \}$.
It is known in \cite{Lu10} that $U_q^-(w)$ does not depend on ${\bf i}$. In particular, if $\g$ is of finite type and $w_0$ is the longest element of $\weyl$, then $U_q^-(w_0) = U_q^-(\g)$.

The commutation relation on the root vectors $f_{\beta_k}$ $(1 \le k \le m)$ is well-known as Levendorskii-Soibelman formula (e.g. \!\!\!see \cite{Ki12} and references therein) and this implies that the subspace $U_q^-(w)$ is the $\mathbb{Q}(q)$-subalgebra of $U_q^-(\g)$ generated by $\{ \, f_{\beta_k}  \, | \,1 \le k \le m \, \}$, which is called {\em quantum nilpotent subalgebra} associated to $w$.

\subsection{PBW crystal $\B_{\bf i}$}
Assume that $\g$ is of finite type. Let us review the formulation of $B(\infty)$ for finite types by using Poincar\'e-Birkhoff-Witt type bases \cite{Lu90, Lu90-2, Sai} (cf. \cite{Kas91}).

Let us take $w = w_0$ in \eqref{eq:beta_k}. Then the set of positive roots associated to $w_0$ is equal to $\Phi^+$, and the set
$$
	B_{\bf i}
	:=
	\left\{
		\, b_{\bf i}({\bf c}) \, \mid \, {\bf c} \in \mathbb{Z}_+^N
	\right\}
$$
is a $\mathbb{Q}(q)$-basis of $U_q^-(\g)$, which is often called a {\em PBW basis} associated to ${\bf i}$.
Let $\mathcal{A}_0$ be the subring of $\mathbb{Q}(q)$ consisting of rational functions regular at $q = 0$.
The $\mathcal{A}_0$-lattice of $U_q^-(\g)$ generated by $B_{\bf i}$ is independent of the choice of ${\bf i}$, which we denote by $L(\infty)$.
Let $\pi : L(\infty) \, \longrightarrow \, L(\infty) / qL(\infty)$ be the canonical projection. Let $B(\infty)$ be the image of $B_{\bf i}$ under $\pi$. Then it is known in \cite{Lu90-2, Sai} that the pair $(L(\infty), B(\infty))$ coincides with the Kashiwara's crystal base of $U_q^-(\mathfrak{g})$ \cite{Kas91}

We identify 
$$
\B_{\bf i} := \mathbb{Z}_+^N
$$ 
with the crystal $\pi(B_{\bf i}) (= B(\infty))$ under the map ${\bf c} \mapsto b_{\bf i}({\bf c})$, and call ${\bf c} \in \B_{\bf i}$ an {\em ${\bf i}$-Lusztig data} associated to ${\bf i}$. Then $\B_{\bf i}$ is called the crystal of ${\bf i}$-Lusztig datum. We also call it {\em PBW crystal} for short if there is no confusion for ${\bf i}$.

\section{Simply braided reduced expressions} \label{sec:simply rex}

\subsection{Simply braided words}
Let us recall that a total order $\prec$ on the set $\Phi^+$ of positive roots. Let us assume that $\g$ is of simply-laced type for simplicity. A order $\prec$ is called {\em convex} 
if either $\gamma \prec \gamma' \prec \gamma''$ or $\gamma'' \prec \gamma' \prec \gamma$ whenever $\gamma'=\gamma+\gamma''$ for $\gamma, \gamma', \gamma''\in \Phi^+$. It is well-known that there exists a one-to-one correspondence between $R(w_0)$ and the set of convex orders on $\Phi^+$, where the convex order $\prec$ 
associated with ${\bf i}=(i_1,\ldots,i_N)\in R(w_0)$ is given by 
\begin{equation} \label{eq:convex}
\beta_1\prec \beta_2\prec \ \ldots \prec \beta_N,
\end{equation} 
where $\beta_k$ is as in \eqref{eq:beta_k} \cite{Pa}.

There exists a reduced expression ${\bf i}' \in R(w_0)$ obtained from ${\bf i}$ by a 3-term braid move 
$(i_k,i_{k+1},i_{k+2}) \rightarrow (i_{k+1},i_{k},i_{k+1})$ with $i_k=i_{k+2}$ if and only if 
$\{\,\beta_{k},\beta_{k+1},\beta_{k+2}\,\}$ 
forms the positive roots of type $A_2$,
where the corresponding convex order $\prec'$ is given by replacing $\beta_k\prec \beta_{k+1}\prec \beta_{k+2}$ with $\beta_{k+2}\prec' \beta_{k+1}\prec' \beta_{k}$. 
Also there exists a reduced expression ${\bf i}' \in R(w_0)$ obtained from ${\bf i}$ by a 2-term braid move $(i_k,i_{k+1})\rightarrow (i_{k+1},i_{k})$ if and only if 
$\beta_{k}$ and $\beta_{k+1}$ are of type $A_1 \times A_1$, where the associated convex ordering $\prec'$ is given by replacing $\beta_k\prec \beta_{k+1}$ with 
$\beta_{k+1}\prec'\beta_{k}$.

\begin{df} \label{df:simply braided}
{\em 
(\cite[Definition 4.1]{SST})
Let $i \in \I$ be given. Then ${\bf i} \in R(w_0)$ is {\em simply braided for $i$} if one can obtain ${\bf i}'=(i'_1,\ldots ,i'_N) \in R(w_0)$ with $i'_1=i$ by applying to ${\bf i}$ a sequence of braid moves consisting of either a 2-term move or 3-term braid move such that
$(\gamma,\gamma',\gamma'')\rightarrow (\gamma'',\gamma',\gamma)$ with $\gamma''=\alpha_i$.
If ${\bf i}$ is simply braided for all $i \in \I$, then we say that ${\bf i}$ is {\em simply braided}.
}
\end{df}

\subsection{Reduced expressions of the longest element of types ADE}
Let us briefly review some facts related to the (extended) affine Weyl groups (see \cite{Bo, Hum, K} for more details).

Let $\ag$ be the affine Kac-Moody algebra of symmetric affine type with an index set $\aI = \{ 0, 1, \dots, n \}$, where $\g$ is the underlying finite-dimensional simple Lie algebra.
We denote by $\aweyl$ the affine Weyl group of $\ag$, that is, the group generated by simple reflections $s_i$ $(i \in \aI)$.
It is well-known that $\aweyl$ is isomorphic to the semidirect product $\weyl \ltimes Q$ under the identification $s_0 \leftrightarrow (s_{\theta}, \theta)$ and $s_i \leftrightarrow (s_i, 0)$.

The extended affine Weyl group $\eweyl$ is given by the semidirect product $\weyl \ltimes P$.
Let $\mathcal{T}$ be the set of bijections $\tau : \aI \rightarrow \aI$ such that $a_{\tau(i)\tau(j)} = a_{ij}$ for all $i, j \in \aI$. 
Each element $\tau \in \mathcal{T}$ induces a unique automorphism $\psi_{\tau}$ of $\aweyl$ so that $\psi_{\tau}(s_i) = s_{\tau(i)}$ for $i \in \aI$.
Then it is known that  $\aweyl$ is a normal subgroup of $\eweyl$ such that $\mathcal{T} \simeq \eweyl \big/ \aweyl$
and
$\eweyl \simeq \mathcal{T} \ltimes \aweyl$, where the action of $\tau \in \mathcal{T}$ in $\eweyl$ is given by $\psi_{\tau}$.
A reduced expression of $w \in \eweyl$ is defined by $w = \tau s_{i_1} \cdots s_{i_m}$, where $\tau \in \mathcal{T}$ and $s_{i_1} \cdots s_{i_m}$ is a reduced expression of $\tau^{-1}w \in \aweyl$.

Let us fix $r \in \I$ and let $t_{\varpi_r}$ be the translation by the $r$-th fundamental weight $\varpi_r$ of $\g$ \cite{K}. 
One can take a diagram automorphism $\tau_r \in \mathcal{T}$ so that $\tau_r^{-1} t_{\varpi_r} \in \aweyl$. 
Let $\wJ$ be the minimal length coset representative associated to $\J = \I \setminus \{ r \}$.
Then it is well-known (e.g. see \cite{Bo}) that $t_{-\varpi_r} \tau_r = \wJ$.

Now, we assume that $\varpi_r$ is minuscule, that is, the weyl group $W$ acts transitively on the weights of $V(\varpi_r)$.
Note that the complete list of $r \in I$ such that $\varpi_r$ is minuscule is as follows (e.g. see \cite{Bo, Bo2}):
\begin{figure*}[h]
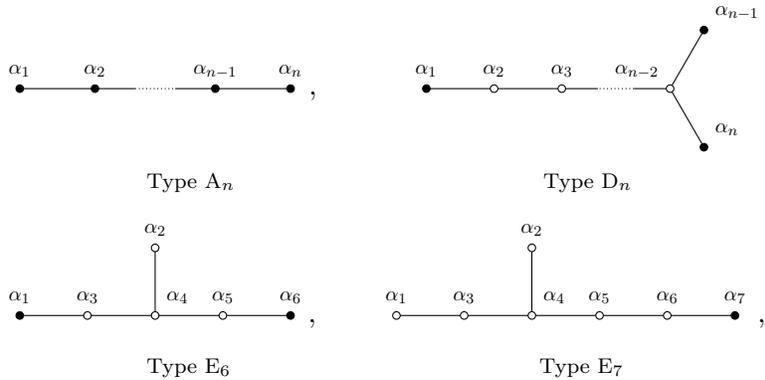

\begin{equation*} 
\begin{split}
  & \dynkin[edge length=1cm, o/.append style={ﬁll=black}, */.append style={ﬁll=white}, labels*={\alpha_1, \alpha_2, \alpha_{n-1}, \alpha_n}] A{oo.oo}, 
  \qquad \quad
  	\dynkin[edge length=0.9cm, o/.append style={ﬁll=black}, */.append style={ﬁll=white}, labels*={\alpha_1, \alpha_2, \alpha_3, \alpha_{n-2}, \alpha_{n-1}, \alpha_n}] D{o**.*oo}, 
  \\
  & \,\,\, \qquad \qquad\,\,\, \text{\scriptsize Type $\text{A}_n$ \,\,\,} \qquad \qquad \qquad \qquad \quad \quad \text{\scriptsize Type $\text{D}_n$ \,\,\,\,\,} \\
  & \dynkin[edge length=0.9cm, o/.append style={ﬁll=black}, */.append style={ﬁll=white}, labels*={\alpha_1, \alpha_2, \alpha_3, \alpha_4, \alpha_5, \alpha_6}] E{o****o},\qquad
  \dynkin[edge length=0.9cm, o/.append style={ﬁll=black}, */.append style={ﬁll=white}, labels*={\alpha_1, \alpha_2, \alpha_3, \alpha_4, \alpha_5, \alpha_6, \alpha_7}] E{******o}\,, \\
  & \,\,\, \qquad \qquad\,\,\, \text{\scriptsize Type $\text{E}_6$ \,\,\,} \qquad \qquad \qquad \qquad \quad \quad \text{\scriptsize Type $\text{E}_7$ \,\,\,\,\,}
\end{split}
\end{equation*}
\caption{The Dynkin diagrams for types $\text{A}_n$, $\text{D}_n$, $\text{E}_6$ and $\text{E}_7$ with minuscule nodes marked as $\bullet$.}
\label{fig:dynkins}
\end{figure*}

The following reduced expressions of $\wJ$ for types $\text{ADE}$ can be essentially found in \cite[Sections 5, 7 and 8]{Li98}, respectively, where the enumeration of simple roots for types $\text{D}$ and $\text{E}$ in \cite{Li98} are slightly different from the setting in this paper. We remark that the reduced expression in type $\text{A}$ is slightly generalized form of the one in \cite[Section 5]{Li98}.

\begin{prop} \label{prop:i^J}
{\em 
The reduced expression $\biJ$ of $\wJ$ given as follows.
\begin{itemize}
  \item[(1)] for type $\text{A}_n$, we have
    \begin{equation*}
        \biJ = {\bf i}_1 \cdot \, {\bf i}_2 \cdot \, \cdots\, \cdot\, {\bf i}_{n-r+1},
    \end{equation*} where ${\bf i}_s = (r+s-1, \,r+s-2, \dots, \,s+1, \,s)$ for $1 \le s \le n-r+1$. 
  \vskip 1mm
  
  \item[(2)] for type $\text{D}_n$, we have
    \begin{equation*}
      \begin{cases}
        \biJ = (1, 2, \cdots, n-1,\, n,\, n-2,\, n-3, \cdots, 2, 1) & \text{if $r = 1$}, \\
        \biJ = {\bf i}_1 \cdot \, {\bf i}_2 \, \cdots \, {\bf i}_{n-1} & \text{if $r = n-1$ or $n$},
      \end{cases}
    \end{equation*}
    where ${\bf i}_s$ $(1 \le s \le n-1)$ is given by
    \begin{equation*}
      {\bf i}_s = 
      \begin{cases}
        (r,n-2,\ldots,k+1,k), & \text{if $s$ is odd},\\
        (r-1,n-2,\ldots,k+1,k), & \text{if $s$ is even and $r = n$},\\
        (r+1,n-2,\ldots,k+1,k), & \text{if $s$ is even and $r = n-1$},\\
        (r), & \text{if $n$ is even and $s=n-1$}.
      \end{cases}
    \end{equation*}
  \vskip 1mm

  \item[(3)] for type $\text{E}_6$, we have
    \begin{equation*}
        \biJ = 
          \begin{cases}
            (1, 3, 4, 5, 6, \,\, 2, 4, 5, \,\, 3, 4, 2, \,\, 1, 3, 4, 5, 6) & \text{if $r = 1$}, \\
            (6, 5, 4, 3, 1, \,\, 2, 4, 3, \,\, 5, 4, 2, \,\, 6, 5, 4, 3, 1) & \text{if $r = 6$}.
          \end{cases}
    \end{equation*}
  \vskip 1mm

  \item[(4)] for type $\text{E}_7$, we have
    \begin{equation*}
      \biJ = (7, 6, 5, 4, 3, 1, \,\, 2, 4, 3, \,\, 5, 4, 2, \,\, 6, 5, 4, 3, 1, \,\, 7, 6, 5, 4, 3, \,\, 2, 4, 5, 6, 7). 
    \end{equation*}
\end{itemize}
}
\qed
\end{prop}

Let us take $w_0 = \wJ w_{\text{J}^{\star}}$ such that $\wJ w_{\text{J}^{\star}} = w_{\text{J}} \wJ$, where $w_{\text{J}}$ is the longest element of the subgroup $\weyl_{\text{J}} \subset \weyl$ generated by $s_j$ for $j \in \J$.
Then we define the reduced expression ${\bf i}_0 \in R(w_0)$ by
\begin{equation} \label{eq:i_0}
  {\bf i}_0 = \biJ \cdot {\bf i}_{\text{J}^{\star}},
\end{equation}
where $\biJ$ is given in Proposition \ref{prop:i^J} and ${\bf i}_{\text{J}^{\star}}$ is a reduced expression of $w_{\text{J}^{\star}}$. 
In particular, if $\g$ is of types $\text{E}_6$ or $\text{E}_7$, then we may choose the reduced expression ${\bf i}_{\text{J}^{\star}}$ by 
\begin{equation} \label{eq:i_J}
{\footnotesize 
  {\bf i}_{\text{J}^{\star}} = 
  \begin{cases}
    (5, 4, 3, 1, \,\, 2, 4, 3, \,\, 5, 4, \,\, 2, \,\, 5, 4, 3, 1, \,\, 5, 4, 3, \,\, 5, 4, \,\, 5) & \text{for type $\text{E}_6$ with $r=1$}, \\
    (3, 4, 5, 6, \,\, 2, 4, 5, \,\, 3, 4, \,\, 2, \,\, 3, 4, 5, 6, \,\, 3, 4, 5, \,\, 3, 4, \,\, 3) & \text{for type $\text{E}_6$ with $r=6$}, \\
    {\bf i}_0^{\text{E}_6} & \text{for type $\text{E}_7$ with $r = 7$},
  \end{cases}
 }
\end{equation}
where ${\bf i}_0^{\text{E}_6}$ is equal to the reduced expression \eqref{eq:i_0} of type $\text{E}_6$ with $r = 1$ or $6$. 
Note that the expression ${\bf i}_{\text{J}^{\star}}$ for type $\text{E}_6$ is obtained from \eqref{eq:i_0} of type $\text{D}_5$ associated to $\J$ (cf. \cite[Section 3.1]{JK19}).

The following property of ${\bf i}_0$ is a consequence of \cite{SST, Ste}. For completeness, let us give an elementary proof of this property for type $\text{E}$.
\begin{prop} \label{prop:simply braided}
{\em 
For $i \in \I \setminus \{ r \}$, there exists a reduced expression ${\bf i} \in R(w_0)$, which is equal to ${\bf i}_0$ up to $2$-term braid moves, such that ${\bf i}$ is simply braided for $i$.
}
\end{prop}
\begin{proof}
We state the proof for type $\text{E}_n$, where the similar proof for other cases is given in \cite{JK19}.
\vskip 2mm

{\it Case 1.} Type $\text{E}_6$. Let us consider the case of $r = 6$. Note that the proof for the case of $r = 1$ is almost identical. Since ${\bf i}_{\J^{\star}}$ is the reduced expression of the longest element for type $\text{D}_5$ associated to $\{ \, 2, 3, 4, 5, 6 \, \}$ (cf. \eqref{eq:dynkin of l}) by ignoring ${\bf i}^{\J}$, we know that it is simply braided for $k \neq 1$ \cite[Proposition 3.2]{JK19}, that is, there exists ${\bf i}_{\J^{\star}}'$ such that $k \cdot {\bf i}_{\J^{\star}}'$ is obtained from ${\bf i}_{\J^{\star}}$ by $2$-term and $3$-term braid moves as in Definition \ref{df:simply braided}.
Then it is enough to check that ${\bf i}^{\J} \cdot k$ is simply braided for $k^{\star}$ up to $2$-term braid moves, where $k^{\star}$ is given by
\begin{center}
\begin{tabular}{|c||c|c|c|c|c|}
  \hline
  $k$ & $2$ & $3$ & $4$ & $5$ & $6$ \\
  \hline
  $k^*$ & $5$ & $2$ & $4$ & $3$ & $1$ \\
  \hline
\end{tabular}
\end{center}

Let us consider
\begin{equation} \label{eq:ov iJ for E6}
	\overline{\bf i}^{\J, k} = 
	\begin{cases}
		(6, 5, 4, 3, 1, \,\, 2, 4, 3, \,\, 5, {\bf 6}, {\bf 4}, \,\, {\bf 5}, {\bf 2}, 4, 3, 1)  & \text{if $k = 2$,} \\
		(6, 5, 4, {\bf 2}, {\bf 3}, \,\, {\bf 4}, {\bf 1}, 3, \,\, 5, 4, 2, \,\, 6, 5, 4, 3, 1)  & \text{if $k = 3$,} \\
		(6, 5, 4, 3, 1, \,\, 2, 4, {\bf 5}, \,\, {\bf 3}, 4, 2, \,\, 6, 5, 4, 3, 1)  & \text{if $k = 4$,} \\
		(6, 5, 4, 3, 1, \,\, 2, 4, 3, \,\, 5, 4, 2, \,\, 6, 5, 4, 3, 1)  & \text{if $k = 5, 6$.} \\
	\end{cases}
\end{equation}

One can check that $\overline{\bf i}^{\J, k}$ is equal to ${\bf i}^{\J}$ up to $2$-term braid moves which occur on bold numbers. Furthermore, it is straightforward to check that $\overline{\bf i}^{\J, k} \cdot k$ is simply braided for $k^{\star}$. Hence ${\bf i} := \overline{\bf i}^{\J, k} \cdot {\bf i}_{\J}$ is simply braided for $k^{\star}$ and equal to ${\bf i}_0$ up to $2$-term braid moves.
\vskip 1mm

{\it Case 2.} Type $\text{E}_7$.
In this case, ${\bf i}_{\J^{\star}}$ is the reduced expression of the longest element for type $\text{E}_6$ associated to $\{ \, 1, 2, 3, 4, 5, 6 \, \}$ (cf. \eqref{eq:dynkin of l}) by ignoring ${\bf i}^{\J}$, and it is simply braided for $k \neq 7$ since we take ${\bf i}_{\J^{\star}}$ is equal to ${\bf i}_0$ for type $\text{E}_6$ \eqref{eq:i_J}. Then, we apply the similar argument as in Case 1 by using
\begin{equation} \label{eq:ov iJ for E7}
\scalebox{0.85}{$
	\overline{\bf i}^{\J, k} = 
	\begin{cases}
		(7, 6, 5, 4, 3, 1, \,\, 2, 4, 3, \,\, 5, 4, 2, \,\, {\bf 6}, {\bf 7}, {\bf 5}, {\bf 6}, {\bf 4}, \,\, {\bf 5}, {\bf 3}, {\bf 4}, {\bf 1}, {\bf 3}, \,\, 2, 4, 5, 6, 7) & \text{if $k = 1$,} \\
		(7, 6, 5, 4, {\bf 2}, {\bf 3}, \,\, {\bf 4}, {\bf 1}, 3, \,\, 5, 4, 2, \,\, 6, 5, 4, 3, 1, \,\, 7, 6, 5, 4, 3, \,\, 2, 4, 5, 6, 7) & \text{if $k = 2$,} \\
		(7, 6, 5, 4, 3, 1, \,\, 2, 4, 3, \,\, 5, {\bf 6}, {\bf 4}, \,\, {\bf 5}, {\bf 2}, 4, 3, 1, \,\, 7, 6, 5, 4, {\bf 2}, \,\, {\bf 3}, 4, 5, 6, 7) & \text{if $k = 3$,} \\
		(7, 6, 5, 4, 3, 1, \,\, 2, 4, {\bf 5}, \,\, {\bf 3}, 4, 2, \,\, 6, 5, 4, 3, 1, \,\, 7, 6, 5, 4, 3, \,\, 2, 4, 5, 6, 7) & \text{if $k = 4$,} \\
		(7, 6, 5, 4, 3, 1, \,\, 2, 4, 3, \,\, 5, 4, 2, \,\, 6, 5, 4, 3, 1, \,\, 7, 6, 5, 4, 3, \,\, 2, 4, 5, 6, 7) & \text{if $k = 5, 6$.} 
	\end{cases}
	$}
\end{equation}
Here $k^{\star}$ is given by
\begin{center}
\begin{tabular}{|c||c|c|c|c|c|c|}
  \hline
  $k$ & $1$ & $2$ & $3$ & $4$ & $5$ & $6$ \\
  \hline
  $k^*$ & $6$ & $2$ & $5$ & $4$ & $3$ & $1$ \\
  \hline
\end{tabular}
\end{center}
Then ${\bf i} := \overline{\bf i}^{\J, k} \cdot {\bf i}_{\J^{\star}}$ is simply braided for $k^{\star}$ and equal to ${\bf i}_0$ up to $2$-term braid moves. 
\end{proof}

\begin{rem}
{\em 
\mbox{} \

\begin{itemize}
	\item[(1)] The reduced expression $\biJ$ is unique up to $2$-term braid moves in the sense of \cite[Remark 5.2]{JK20} (see also \cite{Li98, Ste}).
	\vskip 1mm
	
	\item[(2)] One can check that the reduced expression ${\bf i}_0$ of type $\text{A}_n$ is adapted to Dynkin quiver of type $\text{A}_n$ with a single sink (cf. \cite{Kw18}), while the one of the other types does not satisfy the adaptedness.
\end{itemize}
}
\end{rem}

\subsection{Positive roots and convex order} \label{subsec:convex orders}
From now on, we assume that ${\mf g}$ is of type $\text{E}_6$ or $\text{E}_7$, where the Dynkin diagram is given in Figure \ref{fig:dynkins}.
In this subsection, we consider the positive roots and the convex order determined by ${\bf i}_0$ in the sense of \eqref{eq:convex}.  

\subsubsection{{\bf Type} $\text{\bf E}_{\bf 6}$} \label{subsubsec:E6}
In this case, we only consider the case of $r = 6$ since the case of $r = 1$ is almost identical. 
For simplicity, we denote by $\substack{b \\ acdef}$ the positive root $a\alpha_1 + b\alpha_2 + c\alpha_3 + d\alpha_4 + e\alpha_5 + f\alpha_6 \in \Phi^+$.
Then the convex order associated to ${\bf i}_0$ is given as follows.
\begin{equation} \label{eq:positive roots of E6}
\begin{split}
  %
  %
  & \quad\,\, \substack{0 \\ 00001} \prec \substack{0 \\ 00011} \prec \substack{0 \\ 00111} \prec \substack{0 \\ 01111} \prec \substack{0 \\ 11111} \prec \substack{1 \\ 00111} \prec \substack{1 \\ 01111} \prec \substack{1 \\ 11111}  \\
  & \prec \substack{1 \\ 01211} \prec \substack{1 \\ 11211} \prec \substack{1 \\ 12211} \prec \substack{1 \\ 01221} \prec \substack{1 \\ 11221} \prec \substack{1 \\ 12221} \prec \substack{1 \\ 12321} \prec \substack{2 \\ 12321} \\
  %
  %
  & \prec \substack{1 \\ 00000} \prec \substack{1 \\ 00100} \prec \substack{1 \\ 01100} \prec \substack{1 \\ 11100} \prec \substack{1 \\ 00110} \prec \substack{1 \\ 01110} \prec \substack{1 \\ 11110} \prec \substack{1 \\ 01210} \prec \substack{1 \\ 11210} \prec \substack{1 \\ 12210} \\
 & \prec \substack{0 \\ 10000} \prec \substack{0 \\ 11000} \prec \substack{0 \\ 11100} \prec \substack{0 \\ 11110} \prec \substack{0 \\ 01000} \prec \substack{0 \\ 01100} \prec \substack{0 \\ 01110} \prec \substack{0 \\ 00100} \prec \substack{0 \\ 00110} \prec \substack{0 \\ 00010}
\end{split}
\end{equation} \\
where $\substack{2 \\ 12321}$ is the maximal root of $\g$.
Then let us introduce an arrangement of dots in the plane to present the positive roots following the above convex order. 
More precisely, we associate the $k$-th dot in the arrangement in \eqref{eq:arrangement of E6} with the $k$-th positive root in \eqref{eq:positive roots of E6} along the convex order.
Then we often identify the arrangement in \eqref{eq:arrangement of E6} with the one of the above positive roots.

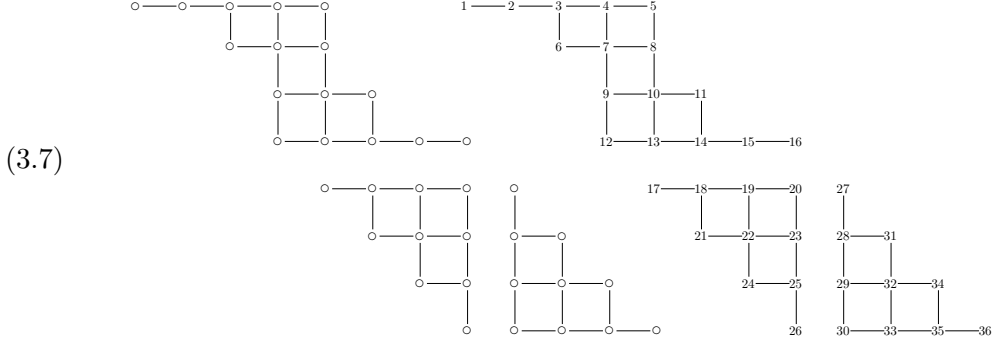
\begin{equation} \label{eq:arrangement of E6}
\scalebox{0.9}{
\begin{tikzpicture}[baseline=(current  bounding  box.center)]
  \draw[-] (0,0) -- (0.5,0); 
  \draw[-] (0.7,0) -- (1.2,0); 
  \draw[-] (1.4,0) -- (1.9,0); 
  \draw[-] (2.1,0) -- (2.6,0); 

  \draw[-] (1.3,-0.1) -- (1.3,-0.5); 
  \draw[-] (2,-0.1) -- (2,-0.5); 
  \draw[-] (2.7,-0.1) -- (2.7,-0.5); 

  \draw[-] (1.4,-0.6) -- (1.9,-0.6); 
  \draw[-] (2.1,-0.6) -- (2.6,-0.6); 

  \draw[-] (2, -0.7) -- (2, -1.2); 
  \draw[-] (2.7, -0.7) -- (2.7, -1.2); 

  \draw[-] (2.1, -1.3) -- (2.6, -1.3); 
  \draw[-] (2.8, -1.3) -- (3.3, -1.3); 

  \draw[-] (2, -1.4)-- (2, -1.9);  
  \draw[-] (2.7, -1.4)-- (2.7, -1.9);  
  \draw[-] (3.4, -1.4)-- (3.4, -1.9);  

  \draw[-] (2.1, -2) -- (2.6, -2); 
  \draw[-] (2.8, -2) -- (3.3, -2); 
  \draw[-] (3.5, -2) -- (4, -2); 
  \draw[-] (4.2, -2) -- (4.7, -2); 
  
  \draw[-] (3.4, -2.8) -- (3.4, -3.3);
  \draw[-] (4.1, -2.8) -- (4.1, -3.3);
  \draw[-] (4.8, -2.8) -- (4.8, -3.3);
  \draw[-] (2.8, -2.7) -- (3.3, -2.7);
  \draw[-] (3.5, -2.7) -- (4, -2.7);
  \draw[-] (4.2, -2.7) -- (4.7, -2.7);
  
  \draw[-] (4.1, -3.5) -- (4.1, -4);
  \draw[-] (4.8, -3.5) -- (4.8, -4);
  \draw[-] (3.5, -3.4) -- (4, -3.4);
  \draw[-] (4.2, -3.4) -- (4.7, -3.4);

  \draw[-] (4.8, -4.2) -- (4.8, -4.7);
  \draw[-] (4.2, -4.1) -- (4.7, -4.1);

  \draw[-] (5.5, -2.8) -- (5.5, -3.3);
  \draw[-] (5.6, -3.4) -- (6.1, -3.4); 
  
  \draw[-] (5.5, -3.5) -- (5.5, -4);
  \draw[-] (6.2, -3.5) -- (6.2, -4);
  \draw[-] (5.6, -4.1) -- (6.1, -4.1); 
  \draw[-] (6.3, -4.1) -- (6.8, -4.1); 
  
  \draw[-] (5.5, -4.2) -- (5.5, -4.7);
  \draw[-] (6.2, -4.2) -- (6.2, -4.7); 
  \draw[-] (6.9, -4.2) -- (6.9, -4.7); 
  \draw[-] (5.6, -4.8) -- (6.1, -4.8); 
  \draw[-] (6.3, -4.8) -- (6.8, -4.8); 
  \draw[-] (7, -4.8) -- (7.5, -4.8); 
    
  \node at (-0.1,0) {${}_{\circ}$}; 
  \node at (0.6,0) {${}_{\circ}$}; 
  \node at (1.3,0) {${}_{\circ}$}; 
  \node at (2,0) {${}_{\circ}$}; 
  \node at (2.7,0) {${}_{\circ}$}; 

  \node at (1.3, -0.6) {${}_{\circ}$}; 
  \node at (2, -0.6) {${}_{\circ}$}; 
  \node at (2.7, -0.6) {${}_{\circ}$}; 

  \node at (2, -1.3) {${}_{\circ}$}; 
  \node at (2.7, -1.3) {${}_{\circ}$}; 
  \node at (3.4, -1.3) {${}_{\circ}$}; 

  \node at (2, -2) {${}_{\circ}$}; 
  \node at (2.7, -2) {${}_{\circ}$}; 
  \node at (3.4, -2) {${}_{\circ}$}; 
  \node at (4.1, -2) {${}_{\circ}$}; 
  \node at (4.8, -2) {${}_{\circ}$}; 
  
  \node at (2.7, -2.7) {${}_{\circ}$}; 
  \node at (3.4, -2.7) {${}_{\circ}$}; 
  \node at (4.1, -2.7) {${}_{\circ}$}; 
  \node at (4.8, -2.7) {${}_{\circ}$}; 
  
  \node at (3.4, -3.4) {${}_{\circ}$}; 
  \node at (4.1, -3.4) {${}_{\circ}$}; 
  \node at (4.8, -3.4) {${}_{\circ}$}; 
  
  \node at (4.1, -4.1) {${}_{\circ}$}; 
  \node at (4.8, -4.1) {${}_{\circ}$}; 
  
  \node at (4.8, -4.8) {${}_{\circ}$}; 
  \node at (5.5, -2.7) {${}_{\circ}$}; 
  
  \node at (5.5, -3.4) {${}_{\circ}$}; 
  \node at (6.2, -3.4) {${}_{\circ}$}; 
  
  \node at (5.5, -4.1) {${}_{\circ}$}; 
  \node at (6.2, -4.1) {${}_{\circ}$}; 
  \node at (6.9, -4.1) {${}_{\circ}$}; 

  \node at (5.5, -4.8) {${}_{\circ}$}; 
  \node at (6.2, -4.8) {${}_{\circ}$}; 
  \node at (6.9, -4.8) {${}_{\circ}$}; 
  \node at (7.6, -4.8) {${}_{\circ}$}; 
\end{tikzpicture}
\hspace{-3.5cm}
\begin{tikzpicture}[baseline=(current  bounding  box.center)]
  \draw[-] (0,0) -- (0.5,0); 
  \draw[-] (0.7,0) -- (1.2,0); 
  \draw[-] (1.4,0) -- (1.9,0); 
  \draw[-] (2.1,0) -- (2.6,0); 
  
  \draw[-] (1.3,-0.1) -- (1.3,-0.5); 
  \draw[-] (2,-0.1) -- (2,-0.5); 
  \draw[-] (2.7,-0.1) -- (2.7,-0.5); 
  
  \draw[-] (1.4,-0.6) -- (1.9,-0.6); 
  \draw[-] (2.1,-0.6) -- (2.6,-0.6); 
  
  \draw[-] (2, -0.7) -- (2, -1.2); 
  \draw[-] (2.7, -0.7) -- (2.7, -1.2); 
  
  \draw[-] (2.1, -1.3) -- (2.6, -1.3); 
  \draw[-] (2.8, -1.3) -- (3.3, -1.3); 
  
  \draw[-] (2, -1.4)-- (2, -1.9);  
  \draw[-] (2.7, -1.4)-- (2.7, -1.9);  
  \draw[-] (3.4, -1.4)-- (3.4, -1.9);  
  
  \draw[-] (2.1, -2) -- (2.6, -2); 
  \draw[-] (2.8, -2) -- (3.3, -2); 
  \draw[-] (3.5, -2) -- (4, -2); 
  \draw[-] (4.2, -2) -- (4.7, -2); 
  
  \draw[-] (3.4, -2.8) -- (3.4, -3.3);
  \draw[-] (4.1, -2.8) -- (4.1, -3.3);
  \draw[-] (4.8, -2.8) -- (4.8, -3.3);
  \draw[-] (2.8, -2.7) -- (3.3, -2.7);
  \draw[-] (3.5, -2.7) -- (4, -2.7);
  \draw[-] (4.2, -2.7) -- (4.7, -2.7);

  \draw[-] (4.1, -3.5) -- (4.1, -4);
  \draw[-] (4.8, -3.5) -- (4.8, -4);
  \draw[-] (3.5, -3.4) -- (4, -3.4);
  \draw[-] (4.2, -3.4) -- (4.7, -3.4);

  \draw[-] (4.8, -4.2) -- (4.8, -4.7);
  \draw[-] (4.2, -4.1) -- (4.7, -4.1);

  \draw[-] (5.5, -2.8) -- (5.5, -3.3);
  \draw[-] (5.6, -3.4) -- (6.1, -3.4); 
  
  \draw[-] (5.5, -3.5) -- (5.5, -4);
  \draw[-] (6.2, -3.5) -- (6.2, -4);
  \draw[-] (5.6, -4.1) -- (6.1, -4.1); 
  \draw[-] (6.3, -4.1) -- (6.8, -4.1); 
  
  \draw[-] (5.5, -4.2) -- (5.5, -4.7);
  \draw[-] (6.2, -4.2) -- (6.2, -4.7); 
  \draw[-] (6.9, -4.2) -- (6.9, -4.7); 
  \draw[-] (5.6, -4.8) -- (6.1, -4.8); 
  \draw[-] (6.3, -4.8) -- (6.8, -4.8); 
  \draw[-] (7, -4.8) -- (7.5, -4.8); 
  
  \node at (-0.1,0) {${}_{\scalebox{0.5}{1}}$}; 
  \node at (0.6,0) {${}_{\scalebox{0.5}{2}}$}; 
  \node at (1.3,0) {${}_{\scalebox{0.5}{3}}$}; 
  \node at (2,0) {${}_{\scalebox{0.5}{4}}$}; 
  \node at (2.7,0) {${}_{\scalebox{0.5}{5}}$}; 
  
  \node at (1.3, -0.6) {${}_{\scalebox{0.5}{6}}$}; 
  \node at (2, -0.6) {${}_{\scalebox{0.5}{7}}$}; 
  \node at (2.7, -0.6) {${}_{\scalebox{0.5}{8}}$}; 
  
  \node at (2, -1.3) {${}_{\scalebox{0.5}{9}}$}; 
  \node at (2.7, -1.3) {${}_{\scalebox{0.5}{10}}$}; 
  \node at (3.4, -1.3) {${}_{\scalebox{0.5}{11}}$}; 
  
  \node at (2, -2) {${}_{\scalebox{0.5}{12}}$}; 
  \node at (2.7, -2) {${}_{\scalebox{0.5}{13}}$}; 
  \node at (3.4, -2) {${}_{\scalebox{0.5}{14}}$}; 
  \node at (4.1, -2) {${}_{\scalebox{0.5}{15}}$}; 
  \node at (4.8, -2) {${}_{\scalebox{0.5}{16}}$}; 
  
  \node at (2.7, -2.7) {${}_{\scalebox{0.5}{17}}$}; 
  \node at (3.4, -2.7) {${}_{\scalebox{0.5}{18}}$}; 
  \node at (4.1, -2.7) {${}_{\scalebox{0.5}{19}}$}; 
  \node at (4.8, -2.7) {${}_{\scalebox{0.5}{20}}$}; 
  
  \node at (3.4, -3.4) {${}_{\scalebox{0.5}{21}}$}; 
  \node at (4.1, -3.4) {${}_{\scalebox{0.5}{22}}$}; 
  \node at (4.8, -3.4) {${}_{\scalebox{0.5}{23}}$}; 
  
  \node at (4.1, -4.1) {${}_{\scalebox{0.5}{24}}$}; 
  \node at (4.8, -4.1) {${}_{\scalebox{0.5}{25}}$}; 
  
  \node at (4.8, -4.8) {${}_{\scalebox{0.5}{26}}$}; 
  \node at (5.5, -2.7) {${}_{\scalebox{0.5}{27}}$}; 
  
  \node at (5.5, -3.4) {${}_{\scalebox{0.5}{28}}$}; 
  \node at (6.2, -3.4) {${}_{\scalebox{0.5}{31}}$}; 
  
  \node at (5.5, -4.1) {${}_{\scalebox{0.5}{29}}$}; 
  \node at (6.2, -4.1) {${}_{\scalebox{0.5}{32}}$}; 
  \node at (6.9, -4.1) {${}_{\scalebox{0.5}{34}}$}; 

  \node at (5.5, -4.8) {${}_{\scalebox{0.5}{30}}$}; 
  \node at (6.2, -4.8) {${}_{\scalebox{0.5}{33}}$}; 
  \node at (6.9, -4.8) {${}_{\scalebox{0.5}{35}}$}; 
  \node at (7.6, -4.8) {${}_{\scalebox{0.5}{36}}$}; 
\end{tikzpicture}
}
\end{equation}

In particular, we will focus on the part consisting of the positive roots defined as in \eqref{eq:beta_k} with respect to $\biJ$, which is denoted by $\Asixb$ (see \eqref{eq:arrangement of positive roots of nilradical in E6}). 
For the case of $r = 1$, we also denote the arrangement by $\Asixa$ similarly.
Note that $\Asixa$ may be obtained from $\Asixb$ by replacing $6 \leftrightarrow 1$ and $5 \leftrightarrow 3$.

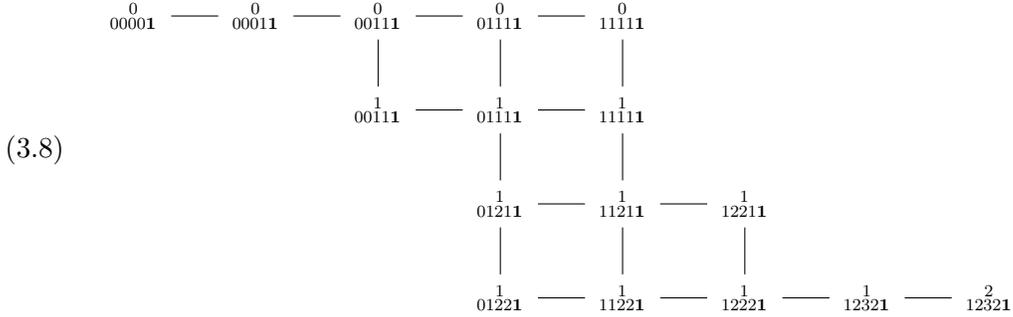
\begin{equation} \label{eq:arrangement of positive roots of nilradical in E6}
\begin{tikzpicture}[scale=1.25, baseline=(current  bounding  box.center)]
  \draw[-] (0.3,0) -- (0.8,0); 
  \draw[-] (1.6,0) -- (2.1,0); 
  \draw[-] (2.9,0) -- (3.4,0); 
  \draw[-] (4.2,0) -- (4.7,0); 
  
  \draw[-] (2.5,-0.25) -- (2.5,-0.75); 
  \draw[-] (3.8,-0.25) -- (3.8,-0.75); 
  \draw[-] (5.1,-0.25) -- (5.1,-0.75); 
  
  \draw[-] (2.9,-1) -- (3.4,-1); 
  \draw[-] (4.2,-1) -- (4.7,-1); 
  
  \draw[-] (3.8, -1.25) -- (3.8, -1.75); 
  \draw[-] (5.1, -1.25) -- (5.1, -1.75); 
  
  \draw[-] (4.2, -2) -- (4.7, -2); 
  \draw[-] (5.5, -2) -- (6, -2); 
  
  \draw[-] (3.8, -2.25)-- (3.8, -2.75);  
  \draw[-] (5.1, -2.25)-- (5.1, -2.75);  
  \draw[-] (6.4, -2.25)-- (6.4, -2.75);  
  
  \draw[-] (4.2, -3) -- (4.7, -3); 
  \draw[-] (5.5, -3) -- (6, -3); 
  \draw[-] (6.8, -3) -- (7.3, -3); 
  \draw[-] (8.1, -3) -- (8.6, -3); 
  
  \node at (-0.1,0) {\scalebox{0.8}{${}_{\substack{0 \\ 0000{\bf 1}}}$}}; 
  \node at (1.2,0) {\scalebox{0.8}{${}_{\substack{0 \\ 0001{\bf 1}}}$}}; 
  \node at (2.5,0) {\scalebox{0.8}{${}_{\substack{0 \\ 0011{\bf 1}}}$}}; 
  \node at (3.8,0) {\scalebox{0.8}{${}_{\substack{0 \\ 0111{\bf 1}}}$}}; 
  \node at (5.1,0) {\scalebox{0.8}{${}_{\substack{0 \\ 1111{\bf 1}}}$}}; 
  
  \node at (2.5, -1) {\scalebox{0.8}{${}_{\substack{1 \\ 0011{\bf 1}}}$}}; 
  \node at (3.8, -1) {\scalebox{0.8}{${}_{\substack{1 \\ 0111{\bf 1}}}$}}; 
  \node at (5.1, -1) {\scalebox{0.8}{${}_{\substack{1 \\ 1111{\bf 1}}}$}}; 
  
  \node at (3.8, -2) {\scalebox{0.8}{${}_{\substack{1 \\ 0121{\bf 1}}}$}}; 
  \node at (5.1, -2) {\scalebox{0.8}{${}_{\substack{1 \\ 1121{\bf 1}}}$}}; 
  \node at (6.4, -2) {\scalebox{0.8}{${}_{\substack{1 \\ 1221{\bf 1}}}$}}; 
  
  \node at (3.8, -3) {\scalebox{0.8}{${}_{\substack{1 \\ 0122{\bf 1}}}$}}; 
  \node at (5.1, -3) {\scalebox{0.8}{${}_{\substack{1 \\ 1122{\bf 1}}}$}}; 
  \node at (6.4, -3) {\scalebox{0.8}{${}_{\substack{1 \\ 1222{\bf 1}}}$}}; 
  \node at (7.7, -3) {\scalebox{0.8}{${}_{\substack{1 \\ 1232{\bf 1}}}$}}; 
  \node at (9, -3) {\scalebox{0.8}{${}_{\substack{2 \\ 1232{\bf 1}}}$}}; 
\end{tikzpicture}
\end{equation}

\subsubsection{{\bf Type} $\text{\bf E}_{\bf 7}$} \label{subsubsec:E7}
We denote by $\substack{b\,\,\, \\ acdefg}$ the positive root $a\alpha_1 + b\alpha_2 + c\alpha_3 + d\alpha_4 + e\alpha_5 + f\alpha_6 + g\alpha_7 \in \Phi^+$. Then the convex order associated to ${\bf i}_0$ is given as follows. 

\begin{equation} \label{eq:positive roots of E7}
{\small
\begin{split}
  & \quad\,\, \substack{0\,\,\, \\ 000001} \prec \substack{0\,\,\, \\ 000011} \prec \substack{0\,\,\, \\ 000111} \prec \substack{0\,\,\, \\ 001111} \prec \substack{0\,\,\, \\ 011111} \prec \substack{0\,\,\, \\ 111111} \prec \substack{1\,\,\, \\ 001111} \prec \substack{1\,\,\, \\ 011111} \prec \substack{1\,\,\, \\ 111111} \\  
  & \prec \substack{1\,\,\, \\ 012111} \prec \substack{1\,\,\, \\ 112111} \prec \substack{1\,\,\, \\ 122111} \prec \substack{1\,\,\, \\ 012211} \prec \substack{1\,\,\, \\ 112211} \prec \substack{1\,\,\, \\ 122211} \prec \substack{1\,\,\, \\ 123211} \prec \substack{2\,\,\, \\ 123211} \prec \substack{1\,\,\, \\ 012221} \\
  & \prec \substack{1\,\,\, \\ 112221} \prec \substack{1\,\,\, \\ 122221} \prec \substack{1\,\,\, \\ 123221} \prec \substack{2\,\,\, \\ 123221} \prec \substack{1\,\,\, \\ 123321} \prec \substack{2\,\,\, \\ 123321} \prec \substack{2\,\,\, \\ 124321} \prec \substack{2\,\,\, \\ 134321} \prec \substack{2\,\,\, \\ 234321} \\
  & \prec \cdots
\end{split}
}
\end{equation}
where $\substack{2\,\,\, \\ 234321}$ is the maximal root of $\g$.
Here the above positive roots are associated to $\biJ$ and the remaining part coincides with \eqref{eq:positive roots of E6} under the identification $\substack{b\,\,\, \\ acdef0}$ with $\substack{b\,\,\, \\ acdef}$.
As in type $\text{E}_6$, let us introduce the arrangement of dots in the plane to present the positive roots, and identify the arrangement with the one of the positive roots along the above convex order. 
Also, we denote by $\Asev$ the part of that corresponding to the positive roots defined as in \eqref{eq:beta_k} with respect to $\biJ$ (see \eqref{eq:arrangement of positive roots of nilradical in E7}).

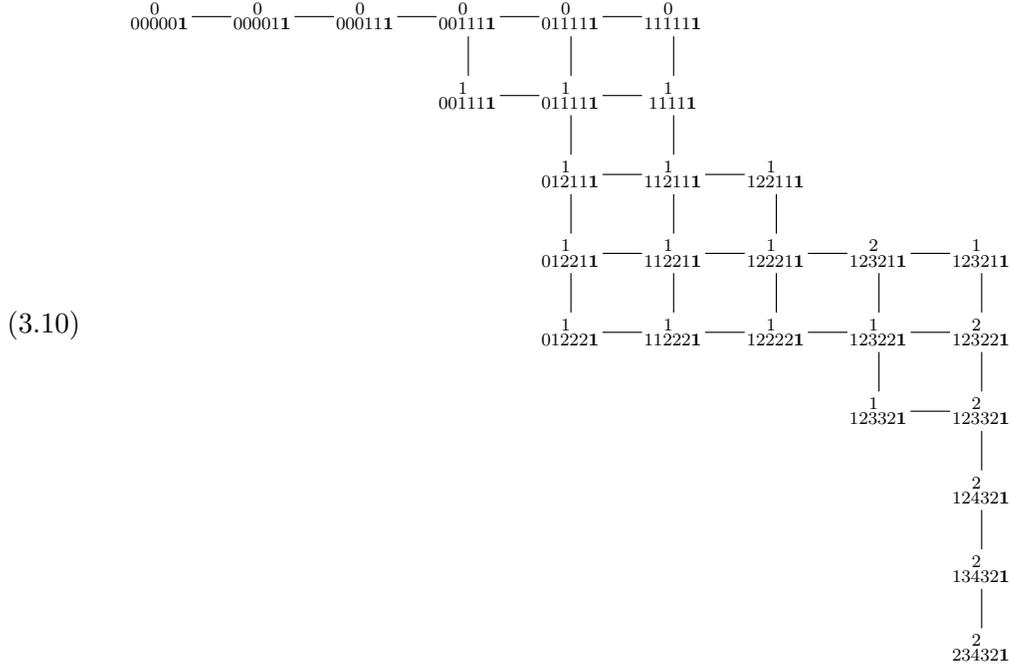
\begin{equation} \label{eq:arrangement of positive roots of nilradical in E7}
\scalebox{1.05}{
\begin{tikzpicture}[baseline=(current  bounding  box.center)]
  \draw[-] (-1, 0) -- (-0.5, 0);
  \draw[-] (0.3,0) -- (0.8,0); 
  \draw[-] (1.6,0) -- (2.1,0); 
  \draw[-] (2.9,0) -- (3.4,0); 
  \draw[-] (4.2,0) -- (4.7,0); 
  
  \draw[-] (2.5,-0.25) -- (2.5,-0.75); 
  \draw[-] (3.8,-0.25) -- (3.8,-0.75); 
  \draw[-] (5.1,-0.25) -- (5.1,-0.75); 
  
  \draw[-] (2.9,-1) -- (3.4,-1); 
  \draw[-] (4.2,-1) -- (4.7,-1); 
  
  \draw[-] (3.8, -1.25) -- (3.8, -1.75); 
  \draw[-] (5.1, -1.25) -- (5.1, -1.75); 
  
  \draw[-] (4.2, -2) -- (4.7, -2); 
  \draw[-] (5.5, -2) -- (6, -2); 
  
  \draw[-] (3.8, -2.25)-- (3.8, -2.75);  
  \draw[-] (5.1, -2.25)-- (5.1, -2.75);  
  \draw[-] (6.4, -2.25)-- (6.4, -2.75);  
  
  \draw[-] (4.2, -3) -- (4.7, -3); 
  \draw[-] (5.5, -3) -- (6, -3); 
  \draw[-] (6.8, -3) -- (7.3, -3); 
  \draw[-] (8.1, -3) -- (8.6, -3); 
  
  \draw[-] (3.8, -3.25)-- (3.8, -3.75);  
  \draw[-] (5.1, -3.25)-- (5.1, -3.75);  
  \draw[-] (6.4, -3.25)-- (6.4, -3.75);  
  \draw[-] (7.7, -3.25) -- (7.7, -3.75); 
  \draw[-] (9, -3.25) -- (9, -3.75); 
  
  \draw[-] (4.2, -4) -- (4.7, -4); 
  \draw[-] (5.5, -4) -- (6, -4); 
  \draw[-] (6.8, -4) -- (7.3, -4); 
  \draw[-] (8.1, -4) -- (8.6, -4); 
  
  \draw[-] (7.7, -4.25) -- (7.7, -4.75);  
  \draw[-] (9, -4.25) -- (9, -4.75); 
  
  \draw[-] (8.1, -5) -- (8.6, -5); 
  
   \draw[-] (9, -5.25) -- (9, -5.75); 
   \draw[-] (9, -6.25) -- (9, -6.75); 
   \draw[-] (9, -7.25) -- (9, -7.75); 
  
  
  \node at (-1.4,0) {\scalebox{0.8}{${}_{\substack{0\,\,\, \\ 00000{\bf 1}}}$}}; 
  \node at (-0.1,0) {\scalebox{0.8}{${}_{\substack{0\,\,\, \\ 00001{\bf 1}}}$}}; 
  \node at (1.2,0) {\scalebox{0.8}{${}_{\substack{0\,\,\, \\ 00011{\bf 1}}}$}}; 
  \node at (2.5,0) {\scalebox{0.8}{${}_{\substack{0\,\,\, \\ 00111{\bf 1}}}$}}; 
  \node at (3.8,0) {\scalebox{0.8}{${}_{\substack{0\,\,\, \\ 01111{\bf 1}}}$}}; 
  \node at (5.1,0) {\scalebox{0.8}{${}_{\substack{0\,\,\, \\ 11111{\bf 1}}}$}}; 

  \node at (2.5, -1) {\scalebox{0.8}{${}_{\substack{1\,\,\, \\ 00111{\bf 1}}}$}}; 
  \node at (3.8, -1) {\scalebox{0.8}{${}_{\substack{1\,\,\, \\ 01111{\bf 1}}}$}}; 
  \node at (5.1, -1) {\scalebox{0.8}{${}_{\substack{1\,\,\, \\ 1111{\bf 1}}}$}}; 
  
  \node at (3.8, -2) {\scalebox{0.8}{${}_{\substack{1\,\,\, \\ 01211{\bf 1}}}$}}; 
  \node at (5.1, -2) {\scalebox{0.8}{${}_{\substack{1\,\,\, \\ 11211{\bf 1}}}$}}; 
  \node at (6.4, -2) {\scalebox{0.8}{${}_{\substack{1\,\,\, \\ 12211{\bf 1}}}$}}; 
  
  \node at (3.8, -3) {\scalebox{0.8}{${}_{\substack{1\,\,\, \\ 01221{\bf 1}}}$}}; 
  \node at (5.1, -3) {\scalebox{0.8}{${}_{\substack{1\,\,\, \\ 11221{\bf 1}}}$}}; 
  \node at (6.4, -3) {\scalebox{0.8}{${}_{\substack{1\,\,\, \\ 12221{\bf 1}}}$}}; 
  \node at (7.7, -3) {\scalebox{0.8}{${}_{\substack{2\,\,\, \\ 12321{\bf 1}}}$}}; 
  \node at (9, -3) {\scalebox{0.8}{${}_{\substack{1\,\,\, \\ 12321{\bf 1}}}$}}; 
  \node at (3.8, -4) {\scalebox{0.8}{${}_{\substack{1\,\,\, \\ 01222{\bf 1}}}$}}; 
  
  \node at (5.1, -4) {\scalebox{0.8}{${}_{\substack{1\,\,\, \\ 11222{\bf 1}}}$}}; 
  \node at (6.4, -4) {\scalebox{0.8}{${}_{\substack{1\,\,\, \\ 12222{\bf 1}}}$}}; 
  \node at (7.7, -4) {\scalebox{0.8}{${}_{\substack{1\,\,\, \\ 12322{\bf 1}}}$}}; 
  \node at (9, -4) {\scalebox{0.8}{${}_{\substack{2\,\,\, \\ 12322{\bf 1}}}$}}; 
  \node at (7.7, -5) {\scalebox{0.8}{${}_{\substack{1\,\,\, \\ 12332{\bf 1}}}$}}; 
  \node at (9, -5) {\scalebox{0.8}{${}_{\substack{2\,\,\, \\ 12332{\bf 1}}}$}}; 
  \node at (9, -6) {\scalebox{0.8}{${}_{\substack{2\,\,\, \\ 12432{\bf 1}}}$}}; 
  \node at (9, -7) {\scalebox{0.8}{${}_{\substack{2\,\,\, \\ 13432{\bf 1}}}$}}; 
  \node at (9, -8) {\scalebox{0.8}{${}_{\substack{2\,\,\, \\ 23432{\bf 1}}}$}}; 
\end{tikzpicture}
}
\end{equation}

\subsubsection{{\bf The arrangement $\Delta_n$}} \label{subsubsec:Delta}
Let us consider the arrangement $\Delta_n$ of dots in the plane to present $(n-1)$-th triangular number.
We often denote by $(i, j)$ the position of a dot under the numbering on rows and columns. For example, $\Delta_{10}$ is given by

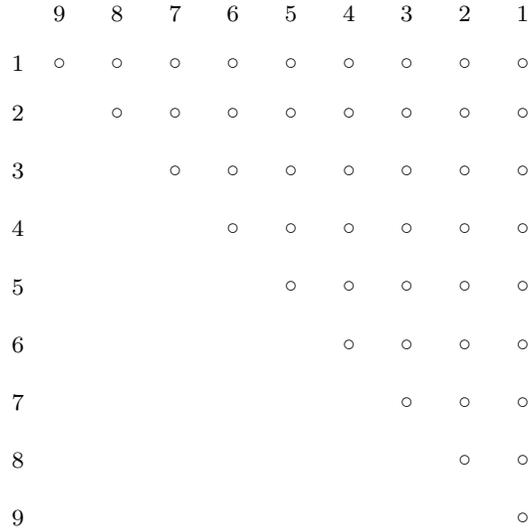
\begin{equation*}
\scalebox{1.1}{
\begin{tikzpicture}
  \node at (-1.3, 0) {${}_{1}$};
  \node at (-1.3, -0.6) {${}_{2}$};
  \node at (-1.3, -1.3) {${}_{3}$};
  \node at (-1.3, -2) {${}_{4}$};
  \node at (-1.3, -2.7) {${}_{5}$};
  \node at (-1.3, -3.4) {${}_{6}$};
  \node at (-1.3, -4.1) {${}_{7}$};
  \node at (-1.3, -4.8) {${}_{8}$};
  \node at (-1.3, -5.5) {${}_{9}$};
  
  \node at (-0.8, 0.6) {${}_{9}$};
  \node at (-0.1, 0.6) {${}_{8}$};
  \node at (0.6, 0.6) {${}_{7}$};
  \node at (1.3, 0.6) {${}_{6}$};
  \node at (2, 0.6) {${}_{5}$};
  \node at (2.7, 0.6) {${}_{4}$};
  \node at (3.4, 0.6) {${}_{3}$};
  \node at (4.1, 0.6) {${}_{2}$};
  \node at (4.8, 0.6) {${}_{1}$};
  
  \node at (-0.8,0) {${}_{\circ}$}; 
  \node at (-0.1,0) {${}_{\circ}$}; 
  \node at (0.6,0) {${}_{\circ}$}; 
  \node at (1.3,0) {${}_{\circ}$}; 
  \node at (2,0) {${}_{\circ}$}; 
  \node at (2.7,0) {${}_{\circ}$}; 
  \node at (3.4,0) {${}_{\circ}$}; 
  \node at (4.1,0) {${}_{\circ}$}; 
  \node at (4.8,0) {${}_{\circ}$}; 

  \node at (-0.1,-0.6) {${}_{\circ}$}; 
  \node at (0.6,-0.6) {${}_{\circ}$}; 
  \node at (1.3, -0.6) {${}_{\circ}$}; 
  \node at (2, -0.6) {${}_{\circ}$}; 
  \node at (2.7, -0.6) {${}_{\circ}$}; 
  \node at (3.4,-0.6) {${}_{\circ}$}; 
  \node at (4.1,-0.6) {${}_{\circ}$}; 
  \node at (4.8,-0.6) {${}_{\circ}$}; 

  \node at (0.6,-1.3) {${}_{\circ}$}; 
  \node at (1.3, -1.3) {${}_{\circ}$}; 
  \node at (2, -1.3) {${}_{\circ}$}; 
  \node at (2.7, -1.3) {${}_{\circ}$}; 
  \node at (3.4, -1.3) {${}_{\circ}$}; 
  \node at (4.1,-1.3) {${}_{\circ}$}; 
  \node at (4.8,-1.3) {${}_{\circ}$}; 

  \node at (1.3, -2) {${}_{\circ}$}; 
  \node at (2, -2) {${}_{\circ}$}; 
  \node at (2.7, -2) {${}_{\circ}$}; 
  \node at (3.4, -2) {${}_{\circ}$}; 
  \node at (4.1, -2) {${}_{\circ}$}; 
  \node at (4.8, -2) {${}_{\circ}$}; 

  \node at (2, -2.7) {${}_{\circ}$}; 
  \node at (2.7, -2.7) {${}_{\circ}$}; 
  \node at (3.4, -2.7) {${}_{\circ}$}; 
  \node at (4.1, -2.7) {${}_{\circ}$}; 
  \node at (4.8, -2.7) {${}_{\circ}$}; 
  
  \node at (2.7, -3.4) {${}_{\circ}$}; 
  \node at (3.4, -3.4) {${}_{\circ}$}; 
  \node at (4.1, -3.4) {${}_{\circ}$}; 
  \node at (4.8, -3.4) {${}_{\circ}$}; 
  
  \node at (3.4, -4.1) {${}_{\circ}$}; 
  \node at (4.1, -4.1) {${}_{\circ}$}; 
  \node at (4.8, -4.1) {${}_{\circ}$}; 
  
  \node at (4.1, -4.8) {${}_{\circ}$}; 
  \node at (4.8, -4.8) {${}_{\circ}$}; 
  
  \node at (4.8, -5.5) {${}_{\circ}$}; 
\end{tikzpicture}
}
\end{equation*}
Then we often regard $\Delta^{\J}_{6, r}$ $(r = 1, 6)$ and $\Asev$ as the sub-arrangements of $\Delta_9$ and $\Delta_{10}$, respectively, that is,
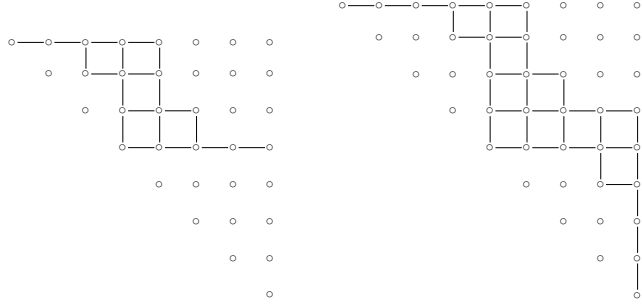
\begin{equation*}
\scalebox{0.7}{
\begin{tikzpicture}
  \draw[-] (0,0) -- (0.5,0); 
  \draw[-] (0.7,0) -- (1.2,0); 
  \draw[-] (1.4,0) -- (1.9,0); 
  \draw[-] (2.1,0) -- (2.6,0); 

  \draw[-] (1.3,-0.1) -- (1.3,-0.5); 
  \draw[-] (2,-0.1) -- (2,-0.5); 
  \draw[-] (2.7,-0.1) -- (2.7,-0.5); 

  \draw[-] (1.4,-0.6) -- (1.9,-0.6); 
  \draw[-] (2.1,-0.6) -- (2.6,-0.6); 

  \draw[-] (2, -0.7) -- (2, -1.2); 
  \draw[-] (2.7, -0.7) -- (2.7, -1.2); 

  \draw[-] (2.1, -1.3) -- (2.6, -1.3); 
  \draw[-] (2.8, -1.3) -- (3.3, -1.3); 

  \draw[-] (2, -1.4)-- (2, -1.9);  
  \draw[-] (2.7, -1.4)-- (2.7, -1.9);  
  \draw[-] (3.4, -1.4)-- (3.4, -1.9);  

  \draw[-] (2.1, -2) -- (2.6, -2); 
  \draw[-] (2.8, -2) -- (3.3, -2); 
  \draw[-] (3.5, -2) -- (4, -2); 
  \draw[-] (4.2, -2) -- (4.7, -2); 

  \node at (-0.1,0) {${}_{\circ}$}; 
  \node at (0.6,0) {${}_{\circ}$}; 
  \node at (1.3,0) {${}_{\circ}$}; 
  \node at (2,0) {${}_{\circ}$}; 
  \node at (2.7,0) {${}_{\circ}$}; 
  \node at (3.4,0) {${}_{\circ}$}; 
  \node at (4.1,0) {${}_{\circ}$}; 
  \node at (4.8,0) {${}_{\circ}$}; 

  \node at (0.6,-0.6) {${}_{\circ}$}; 
  \node at (1.3, -0.6) {${}_{\circ}$}; 
  \node at (2, -0.6) {${}_{\circ}$}; 
  \node at (2.7, -0.6) {${}_{\circ}$}; 
  \node at (3.4,-0.6) {${}_{\circ}$}; 
  \node at (4.1,-0.6) {${}_{\circ}$}; 
  \node at (4.8,-0.6) {${}_{\circ}$}; 

  \node at (1.3, -1.3) {${}_{\circ}$}; 
  \node at (2, -1.3) {${}_{\circ}$}; 
  \node at (2.7, -1.3) {${}_{\circ}$}; 
  \node at (3.4, -1.3) {${}_{\circ}$}; 
  \node at (4.1,-1.3) {${}_{\circ}$}; 
  \node at (4.8,-1.3) {${}_{\circ}$}; 

  \node at (2, -2) {${}_{\circ}$}; 
  \node at (2.7, -2) {${}_{\circ}$}; 
  \node at (3.4, -2) {${}_{\circ}$}; 
  \node at (4.1, -2) {${}_{\circ}$}; 
  \node at (4.8, -2) {${}_{\circ}$}; 
  
  \node at (2.7, -2.7) {${}_{\circ}$}; 
  \node at (3.4, -2.7) {${}_{\circ}$}; 
  \node at (4.1, -2.7) {${}_{\circ}$}; 
  \node at (4.8, -2.7) {${}_{\circ}$}; 
  
  \node at (3.4, -3.4) {${}_{\circ}$}; 
  \node at (4.1, -3.4) {${}_{\circ}$}; 
  \node at (4.8, -3.4) {${}_{\circ}$}; 
  
  \node at (4.1, -4.1) {${}_{\circ}$}; 
  \node at (4.8, -4.1) {${}_{\circ}$}; 
  
  \node at (4.8, -4.8) {${}_{\circ}$}; 
  
\end{tikzpicture}
\hspace{0.7cm}
\begin{tikzpicture}
  \draw[-] (-0.2, 0) -- (-0.7, 0);
  \draw[-] (0,0) -- (0.5,0); 
  \draw[-] (0.7,0) -- (1.2,0); 
  \draw[-] (1.4,0) -- (1.9,0); 
  \draw[-] (2.1,0) -- (2.6,0); 

  \draw[-] (1.3,-0.1) -- (1.3,-0.5); 
  \draw[-] (2,-0.1) -- (2,-0.5); 
  \draw[-] (2.7,-0.1) -- (2.7,-0.5); 

  \draw[-] (1.4,-0.6) -- (1.9,-0.6); 
  \draw[-] (2.1,-0.6) -- (2.6,-0.6); 

  \draw[-] (2, -0.7) -- (2, -1.2); 
  \draw[-] (2.7, -0.7) -- (2.7, -1.2); 

  \draw[-] (2.1, -1.3) -- (2.6, -1.3); 
  \draw[-] (2.8, -1.3) -- (3.3, -1.3); 

  \draw[-] (2, -1.4)-- (2, -1.9);  
  \draw[-] (2.7, -1.4)-- (2.7, -1.9);  
  \draw[-] (3.4, -1.4)-- (3.4, -1.9);  

  \draw[-] (2.1, -2) -- (2.6, -2); 
  \draw[-] (2.8, -2) -- (3.3, -2); 
  \draw[-] (3.5, -2) -- (4, -2); 
  \draw[-] (4.2, -2) -- (4.7, -2); 
  
  \draw[-] (2, -2.1)-- (2, -2.6);  
  \draw[-] (2.7, -2.1)-- (2.7, -2.6);  
  \draw[-] (3.4, -2.1)-- (3.4, -2.6);  
  \draw[-] (4.1, -2.1)-- (4.1, -2.6);  
  \draw[-] (4.8, -2.1)-- (4.8, -2.6);  
  
  \draw[-] (4.1, -2.8) -- (4.1, -3.3);
  \draw[-] (4.8, -2.8) -- (4.8, -3.3);
  \draw[-] (2.1, -2.7) -- (2.6, -2.7);
  \draw[-] (2.8, -2.7) -- (3.3, -2.7);
  \draw[-] (3.5, -2.7) -- (4, -2.7);
  \draw[-] (4.2, -2.7) -- (4.7, -2.7);
  
  \draw[-] (4.8, -3.5) -- (4.8, -4);
  %
  \draw[-] (4.2, -3.4) -- (4.7, -3.4);

  \draw[-] (4.8, -4.2) -- (4.8, -4.7);
  %
  
  \draw[-] (4.8, -4.9) -- (4.8, -5.4);

  \node at (-0.8,0) {${}_{\circ}$}; 
  \node at (-0.1,0) {${}_{\circ}$}; 
  \node at (0.6,0) {${}_{\circ}$}; 
  \node at (1.3,0) {${}_{\circ}$}; 
  \node at (2,0) {${}_{\circ}$}; 
  \node at (2.7,0) {${}_{\circ}$}; 
  \node at (3.4,0) {${}_{\circ}$}; 
  \node at (4.1,0) {${}_{\circ}$}; 
  \node at (4.8,0) {${}_{\circ}$}; 

  \node at (-0.1,-0.6) {${}_{\circ}$}; 
  \node at (0.6,-0.6) {${}_{\circ}$}; 
  \node at (1.3, -0.6) {${}_{\circ}$}; 
  \node at (2, -0.6) {${}_{\circ}$}; 
  \node at (2.7, -0.6) {${}_{\circ}$}; 
  \node at (3.4,-0.6) {${}_{\circ}$}; 
  \node at (4.1,-0.6) {${}_{\circ}$}; 
  \node at (4.8,-0.6) {${}_{\circ}$}; 

  \node at (0.6,-1.3) {${}_{\circ}$}; 
  \node at (1.3, -1.3) {${}_{\circ}$}; 
  \node at (2, -1.3) {${}_{\circ}$}; 
  \node at (2.7, -1.3) {${}_{\circ}$}; 
  \node at (3.4, -1.3) {${}_{\circ}$}; 
  \node at (4.1,-1.3) {${}_{\circ}$}; 
  \node at (4.8,-1.3) {${}_{\circ}$}; 

  \node at (1.3, -2) {${}_{\circ}$}; 
  \node at (2, -2) {${}_{\circ}$}; 
  \node at (2.7, -2) {${}_{\circ}$}; 
  \node at (3.4, -2) {${}_{\circ}$}; 
  \node at (4.1, -2) {${}_{\circ}$}; 
  \node at (4.8, -2) {${}_{\circ}$}; 
  
  \node at (2, -2.7) {${}_{\circ}$}; 
  \node at (2.7, -2.7) {${}_{\circ}$}; 
  \node at (3.4, -2.7) {${}_{\circ}$}; 
  \node at (4.1, -2.7) {${}_{\circ}$}; 
  \node at (4.8, -2.7) {${}_{\circ}$}; 
  
  \node at (2.7, -3.4) {${}_{\circ}$}; 
  \node at (3.4, -3.4) {${}_{\circ}$}; 
  \node at (4.1, -3.4) {${}_{\circ}$}; 
  \node at (4.8, -3.4) {${}_{\circ}$}; 
  
  \node at (3.4, -4.1) {${}_{\circ}$}; 
  \node at (4.1, -4.1) {${}_{\circ}$}; 
  \node at (4.8, -4.1) {${}_{\circ}$}; 
  
  \node at (4.1, -4.8) {${}_{\circ}$}; 
  \node at (4.8, -4.8) {${}_{\circ}$}; 
  
  \node at (4.8, -5.5) {${}_{\circ}$}; 
\end{tikzpicture}
}
\end{equation*}
This convention will be used in Section \ref{subsec:epsilon} to define a certain statistic on $\Delta^{\J}_{6, r}$ and $\Asev$ to present a combinatorial formula of $\varepsilon_r^*$.

\begin{rem}
{\em 
If there is no confusion on ${\mf g}$ and $r$, then we write the arrangements $\Asix$ and $\Asev$ by $\Delta^{\J}$ simply.
}
\end{rem}

\section{Crystals of the quantum nilpotent subalgebras of type E} \label{sec:PBW crystals}

\subsection{Description of $\tf_i$} \label{subsec:tf_i}
Let us briefly review the result in \cite{SST}.
For simplicity, we assume that $\g$ is simply laced.
Let ${\bf i} \in R(w_0)$ be given. For $\beta \in \Phi^+$, we denote by ${\bf 1}_{\beta}$ the element in $\B_{\bf i}$ where $c_\beta=1$ and $c_\gamma=0$ for $\gamma\in \Phi^+\setminus \{\beta\}$.

Let $\sigma=(\sigma_{1},\sigma_2,\ldots,\sigma_s)$ be a sequence with $\sigma_{u}\in \{\,+\,,\,-\, , \ \cdot\ \}$. We replace a pair $(\sigma_{u},\sigma_{u'})=(+,-)$, where $u<u'$ and $\sigma_{u''}=\,\cdot\,$
for $u<u''<u'$, with $(\,\cdot\,,\,\cdot\,)$, and repeat this process as far as possible until we get a sequence with no $-$ placed to the right of $+$. We denote the resulting sequence by ${\sigma}^{\rm red}$. For another sequence $\tau=(\tau_1,\ldots,\tau_{t})$, we denote by $\sigma\cdot\tau$ the concatenation of $\sigma$ and $\tau$.

Given $i \in \I$, suppose that $\bf i$ is simply braided for $i \in \I$ (recall Definition \ref{df:simply braided}).
Then one can obtain ${\bf i}'=(i'_1,\ldots ,i'_N)\in R(w_0)$ with $i'_1=i$ by applying a sequence of braid moves consisting of either a 2-term move or 3-term braid move  
$(\gamma,\gamma',\gamma'')\rightarrow (\gamma'',\gamma',\gamma)$ with $\gamma''=\alpha_i$.
Let 
\begin{equation*}
\Pi_s=\{\gamma_s,\gamma'_s,\gamma''_s\}
\end{equation*}
be the triple of positive roots of type $A_2$ with 
$\gamma'_s=\gamma_s+\gamma''_s$ and $\gamma''_s=\alpha_i$ corresponding to the $s$-th 3-term braid move for $1\leq s\leq t$.
For ${\bf c}\in \B_{\bf i}$, let
\begin{equation}\label{eq:sigma_i}
\sigma_i({\bf c})= 
(\LaTeXunderbrace{-\cdots -}_{c_{\gamma'_1}}\,\LaTeXunderbrace{+\cdots +}_{c_{\gamma_1}}
\ \cdots \
\LaTeXunderbrace{-\cdots -}_{c_{\gamma'_t}}\,\LaTeXunderbrace{+\cdots +}_{c_{\gamma_t}}
).
\end{equation}

\begin{thm}\cite[Theorem 4.6]{SST} \label{thm:signature rule} 
{\em 
Let ${\bf i} \in R(w_0)$ and $i\in \I$. Suppose that ${\bf i}$ is simply braided for $i$. Let ${\bf c}\in \B_{\bf i}$ be given.
\begin{itemize}
\item[(1)] If there exists $+$ in $\sigma_i({\bf c})^{\rm red}$ and the leftmost $+$ appears in $c_{\gamma_s}$, then 
\begin{equation*}
  \tf_i{\bf c} = {\bf c} - {\bf 1}_{\gamma_s} + {\bf 1}_{\gamma'_s}.
\end{equation*}
\item[(2)] If there exists no $+$ in $\sigma_i({\bf c})^{\rm red}$, then 
\begin{equation*}
  \hspace{5.6cm} \tf_i{\bf c} = {\bf c} + {\bf 1}_{\alpha_i}. \hspace{5.6cm} \qed
\end{equation*}
\end{itemize}
}
\end{thm}

\subsection{Crystal $\B^{\J}$ of the quantum nilpotent subalgebra} \label{subsec:B^J}
In this subsection, we consider the crystal $\B^{\J}$ of the quantum nilpotent subalgebra $U_q(\wJ)$, which is obtained as a subcrystal of $\B_{{\bf i}_0}$, where ${\bf i}_0$ is the reduced expression of $w_0$ given in \eqref{eq:i_0}.
Throughout this paper, we set
\begin{equation*}
    \B = \B_{{\bf i}_0}.
\end{equation*}

Recall $\J = \I \setminus \{ r \}$. 
Let ${\mf l}$ be the maximal Levi subalgebra of $\g$ associated to $\{ \, \alpha_i \, | \, i \in \J \}$ of types $\text{D}_5$ and $\text{E}_6$, respectively, as shown on the following Dynkin diagrams:

\begin{figure*}[h]
\begin{equation} \label{eq:dynkin of l}
\scalebox{0.82}{$
\begin{split}
	\raisebox{-0.7pc}{$\text{D}_5$\, : \,}\,\,\, & \xymatrixcolsep{1pc}\xymatrixrowsep{0.5pc}\xymatrix{
	& & \overset{\alpha_2}{\circ} \ar@{-}[d] & & \\
	\color{lightgray}{\underset{\alpha_1}{\circ}} \ar@{-}@[lightgray][r] & \underset{\alpha_3}{\circ} \ar@{-}[r] & \underset{\alpha_4}{\circ} \ar@{-}[r] & \underset{\alpha_5}{\circ} \ar@{-}[r] & \underset{\alpha_6}{\circ}
	}\,\raisebox{-0.9pc}{,}
	\qquad \qquad
	\xymatrixcolsep{1pc}\xymatrixrowsep{0.5pc}\xymatrix{
	& & \overset{\alpha_2}{\circ} \ar@{-}[d] & & \\
	\underset{\alpha_1}{\circ} \ar@{-}[r] & \underset{\alpha_3}{\circ} \ar@{-}[r] & \underset{\alpha_4}{\circ} \ar@{-}[r] & \underset{\alpha_5}{\circ} \ar@{-}@[lightgray][r] & \color{lightgray}{\underset{\alpha_6}{\circ}}
	}
	\\
	\raisebox{-0.7pc}{$\text{E}_6$\, : \,}\,\,\, &
	\xymatrixcolsep{1pc}\xymatrixrowsep{0.5pc}\xymatrix{
	& & \overset{\alpha_2}{\circ} \ar@{-}[d] & & \\
	\underset{\alpha_1}{\circ} \ar@{-}[r] & \underset{\alpha_3}{\circ} \ar@{-}[r] & \underset{\alpha_4}{\circ} \ar@{-}[r] & \underset{\alpha_5}{\circ} \ar@{-}[r] & \underset{\alpha_6}{\circ} \ar@{-}@[lightgray][r] & \color{lightgray}{\underset{\alpha_7}{\circ}}
	}
\end{split}
$}
\end{equation}
\caption{For type $\text{E}_6$ (resp. $\text{E}_7$), the maximal Levi subalgebra ${\mf l}$ is of type $\text{D}_5$ (resp. $\text{E}_6$).}
\end{figure*}

Let $\Phi^+_{\J}$ be the set of positive roots of ${\mf l}$ and let $\Phi^+(\J)$ be the set of positive roots of the nilradical ${\mf u}$ of the parabolic subalgebra of $\g$ associated to ${\mf l}$. Note that
\begin{equation*}
	\Phi^+ = \Phi^+(\J) \,\, \scalebox{0.7}{$\bigsqcup$} \,\, \Phi^+_{\J},
\end{equation*}
and the positive roots in $\Phi^+(\J)$ coincide with the ones in \eqref{eq:arrangement of positive roots of nilradical in E6} and \eqref{eq:arrangement of positive roots of nilradical in E7}.

\begin{rem} \label{rem:tf_i}
{\em 
We should remark that $\tf_i$ on $\B$ can be described by using Theorem \ref{thm:signature rule} due to Proposition \ref{prop:simply braided} in the sense of \cite[Remark 3.3]{JK19} (cf. \cite[Remark 4.6]{SST}).
}
\end{rem}

For simplicity, we modify the notation \eqref{eq:sigma_i}  for type E. 
For example, if $\g$ is of type $\text{E}_6$ with $\gamma_s = \substack{0 \\ 11100}$ and $\gamma_s' = \substack{1 \\ 11100}$\,, 
then we write
\begin{equation*}
	\big( \cdots\, \LaTeXunderbrace{-\,\cdots\,-}_{c_{\gamma_s'}} \LaTeXunderbrace{+\,\cdots\,+}_{c_{\gamma_s}} \,\cdots  \big) 
	= 
	\big( \cdots\, \LaTeXunderbrace{-\,\cdots\,-}_{\,\scalebox{0.8}{$\substack{1 \\ 11100}$}} \LaTeXunderbrace{+\,\cdots\,+}_{\,\scalebox{0.8}{$\substack{0 \\ 11100}$}} \,\cdots  \big).
\end{equation*}
 
\begin{prop} \label{prop:signatures}
{\em
Let $i \in \I$ and ${\bf c} \in \B_{\bf i}$ be given, where $i \neq r$ and ${\bf i}$ is in Theorem \ref{prop:simply braided}. Then the sequence \eqref{eq:sigma_i} is given as follows.
\begin{enumerate}
	\item in type $\text{E}_6$ with $r=1$, we have
	\begin{equation*}
		\sigma_i({\bf c}) = \sigma_{i, 1}({\bf c}) \cdot \sigma_{i, 2}({\bf c}),
	\end{equation*}
	where $\sigma_{i, 1}({\bf c})$ is given by 
	\begin{equation*}
	\scalebox{0.9}{$
	{\allowdisplaybreaks
	\begin{cases}
			( 
			\LaTeXunderbrace{-\cdots-}_{\,\scalebox{0.8}{$\substack{1 \\ 11100}$}} \LaTeXunderbrace{+\cdots+}_{\,\scalebox{0.8}{$\substack{0 \\ 11100}$}}
			\LaTeXunderbrace{-\cdots-}_{\,\scalebox{0.8}{$\substack{1 \\ 11110}$}} \LaTeXunderbrace{+\cdots+}_{\,\scalebox{0.8}{$\substack{0 \\ 11110}$}}
			\LaTeXunderbrace{-\cdots-}_{\,\scalebox{0.8}{$\substack{1 \\ 11111}$}} \LaTeXunderbrace{+\cdots+}_{\,\scalebox{0.8}{$\substack{0 \\ 11111}$}}
			\LaTeXunderbrace{-\cdots-}_{\,\scalebox{0.8}{$\substack{2 \\ 12321}$}} \LaTeXunderbrace{+\cdots+}_{\,\scalebox{0.8}{$\substack{1 \\ 12321}$}}
			) & \text{if $i = 2$}, \\
			( 
			\LaTeXunderbrace{-\cdots-}_{\,\scalebox{0.8}{$\substack{0 \\ 11000}$}} \LaTeXunderbrace{+\cdots+}_{\,\scalebox{0.8}{$\substack{0 \\ 10000}$}}
			\LaTeXunderbrace{-\cdots-}_{\,\scalebox{0.8}{$\substack{1 \\ 12210}$}} \LaTeXunderbrace{+\cdots+}_{\,\scalebox{0.8}{$\substack{1 \\ 11210}$}}
			\LaTeXunderbrace{-\cdots-}_{\,\scalebox{0.8}{$\substack{1 \\ 12211}$}} \LaTeXunderbrace{+\cdots+}_{\,\scalebox{0.8}{$\substack{1 \\ 11211}$}}
			\LaTeXunderbrace{-\cdots-}_{\,\scalebox{0.8}{$\substack{1 \\ 12221}$}} \LaTeXunderbrace{+\cdots+}_{\,\scalebox{0.8}{$\substack{1 \\ 11221}$}}
			) & \text{if $i = 3$}, \\
			( 
			\LaTeXunderbrace{-\cdots-}_{\,\scalebox{0.8}{$\substack{0 \\ 11100}$}} \LaTeXunderbrace{+\cdots+}_{\,\scalebox{0.8}{$\substack{0 \\ 11000}$}}
			\LaTeXunderbrace{-\cdots-}_{\,\scalebox{0.8}{$\substack{1 \\ 11210}$}} \LaTeXunderbrace{+\cdots+}_{\,\scalebox{0.8}{$\substack{1 \\ 11110}$}}
			\LaTeXunderbrace{-\cdots-}_{\,\scalebox{0.8}{$\substack{1 \\ 11211}$}} \LaTeXunderbrace{+\cdots+}_{\,\scalebox{0.8}{$\substack{1 \\ 11111}$}}
			\LaTeXunderbrace{-\cdots-}_{\,\scalebox{0.8}{$\substack{1 \\ 12321}$}} \LaTeXunderbrace{+\cdots+}_{\,\scalebox{0.8}{$\substack{1 \\ 12221}$}}
			) & \text{if $i = 4$}, \\
			( 
			\LaTeXunderbrace{-\cdots-}_{\,\scalebox{0.8}{$\substack{0 \\ 11110}$}} \LaTeXunderbrace{+\cdots+}_{\,\scalebox{0.8}{$\substack{0 \\ 11100}$}}
			\LaTeXunderbrace{-\cdots-}_{\,\scalebox{0.8}{$\substack{1 \\ 11110}$}} \LaTeXunderbrace{+\cdots+}_{\,\scalebox{0.8}{$\substack{1 \\ 11100}$}}
			\LaTeXunderbrace{-\cdots-}_{\,\scalebox{0.8}{$\substack{1 \\ 11221}$}} \LaTeXunderbrace{+\cdots+}_{\,\scalebox{0.8}{$\substack{1 \\ 11211}$}}
			\LaTeXunderbrace{-\cdots-}_{\,\scalebox{0.8}{$\substack{1 \\ 12221}$}} \LaTeXunderbrace{+\cdots+}_{\,\scalebox{0.8}{$\substack{1 \\ 12211}$}}
			) & \text{if $i = 5$}, \\
			( 
			\LaTeXunderbrace{-\cdots-}_{\,\scalebox{0.8}{$\substack{0 \\ 11111}$}} \LaTeXunderbrace{+\cdots+}_{\,\scalebox{0.8}{$\substack{0 \\ 11110}$}}
			\LaTeXunderbrace{-\cdots-}_{\,\scalebox{0.8}{$\substack{1 \\ 11111}$}} \LaTeXunderbrace{+\cdots+}_{\,\scalebox{0.8}{$\substack{1 \\ 11110}$}}
			\LaTeXunderbrace{-\cdots-}_{\,\scalebox{0.8}{$\substack{1 \\ 11211}$}} \LaTeXunderbrace{+\cdots+}_{\,\scalebox{0.8}{$\substack{1 \\ 11210}$}}
			\LaTeXunderbrace{-\cdots-}_{\,\scalebox{0.8}{$\substack{1 \\ 12211}$}} \LaTeXunderbrace{+\cdots+}_{\,\scalebox{0.8}{$\substack{1 \\ 12210}$}}
			) & \text{if $i = 6$}, 
		\end{cases}
		}$
		}
	\end{equation*}\!\!
	and $\sigma_{i, 2}({\bf c})$ is equal to the sequence for $k$ such that $k^{\star} = i$ obtained from \cite[Proposition 3.2]{JK19} to the pair $({\mf l},\, {\bf i}_{\text{J}})$ by regarding $\substack{b \\ 0cdef}$ as the root of type $\text{D}_5$.
	\vskip 2mm
	
	\item in type $\text{E}_6$ with $r=6$, we have
	\begin{equation*}
		\sigma_i({\bf c}) = \sigma_{i, 1}({\bf c}) \cdot \sigma_{i, 2}({\bf c}),
	\end{equation*}
	where $\sigma_{i, 1}({\bf c})$ is given by 
	\begin{equation*}
	\scalebox{0.9}{$
	{\allowdisplaybreaks
		\begin{cases}
			( 
			\LaTeXunderbrace{-\cdots-}_{\,\scalebox{0.8}{$\substack{0 \\ 11111}$}} \LaTeXunderbrace{+\cdots+}_{\,\scalebox{0.8}{$\substack{0 \\ 01111}$}}
			\LaTeXunderbrace{-\cdots-}_{\,\scalebox{0.8}{$\substack{1 \\ 11111}$}} \LaTeXunderbrace{+\cdots+}_{\,\scalebox{0.8}{$\substack{1 \\ 01111}$}}
			\LaTeXunderbrace{-\cdots-}_{\,\scalebox{0.8}{$\substack{1 \\ 11211}$}} \LaTeXunderbrace{+\cdots+}_{\,\scalebox{0.8}{$\substack{1 \\ 01211}$}}
			\LaTeXunderbrace{-\cdots-}_{\,\scalebox{0.8}{$\substack{1 \\ 11221}$}} \LaTeXunderbrace{+\cdots+}_{\,\scalebox{0.8}{$\substack{1 \\ 01221}$}}
			) & \text{if $i = 1$}, \\
			( 
			\LaTeXunderbrace{-\cdots-}_{\,\scalebox{0.8}{$\substack{1 \\ 00111}$}} \LaTeXunderbrace{+\cdots+}_{\,\scalebox{0.8}{$\substack{0 \\ 00111}$}}
			\LaTeXunderbrace{-\cdots-}_{\,\scalebox{0.8}{$\substack{1 \\ 01111}$}} \LaTeXunderbrace{+\cdots+}_{\,\scalebox{0.8}{$\substack{0 \\ 01111}$}}
			\LaTeXunderbrace{-\cdots-}_{\,\scalebox{0.8}{$\substack{1 \\ 11111}$}} \LaTeXunderbrace{+\cdots+}_{\,\scalebox{0.8}{$\substack{0 \\ 11111}$}}
			\LaTeXunderbrace{-\cdots-}_{\,\scalebox{0.8}{$\substack{2 \\ 12321}$}} \LaTeXunderbrace{+\cdots+}_{\,\scalebox{0.8}{$\substack{1 \\ 12321}$}}
			) & \text{if $i = 2$}, \\
			( 
			\LaTeXunderbrace{-\cdots-}_{\,\scalebox{0.8}{$\substack{0 \\ 01111}$}} \LaTeXunderbrace{+\cdots+}_{\,\scalebox{0.8}{$\substack{0 \\ 00111}$}}
			\LaTeXunderbrace{-\cdots-}_{\,\scalebox{0.8}{$\substack{1 \\ 01111}$}} \LaTeXunderbrace{+\cdots+}_{\,\scalebox{0.8}{$\substack{1 \\ 00111}$}}
			\LaTeXunderbrace{-\cdots-}_{\,\scalebox{0.8}{$\substack{1 \\ 12211}$}} \LaTeXunderbrace{+\cdots+}_{\,\scalebox{0.8}{$\substack{1 \\ 11211}$}}
			\LaTeXunderbrace{-\cdots-}_{\,\scalebox{0.8}{$\substack{1 \\ 12221}$}} \LaTeXunderbrace{+\cdots+}_{\,\scalebox{0.8}{$\substack{1 \\ 11221}$}}
			) & \text{if $i = 3$}, \\
			( 
			\LaTeXunderbrace{-\cdots-}_{\,\scalebox{0.8}{$\substack{0 \\ 00111}$}} \LaTeXunderbrace{+\cdots+}_{\,\scalebox{0.8}{$\substack{0 \\ 00011}$}}
			\LaTeXunderbrace{-\cdots-}_{\,\scalebox{0.8}{$\substack{1 \\ 01211}$}} \LaTeXunderbrace{+\cdots+}_{\,\scalebox{0.8}{$\substack{1 \\ 01111}$}}
			\LaTeXunderbrace{-\cdots-}_{\,\scalebox{0.8}{$\substack{1 \\ 11211}$}} \LaTeXunderbrace{+\cdots+}_{\,\scalebox{0.8}{$\substack{1 \\ 11111}$}}
			\LaTeXunderbrace{-\cdots-}_{\,\scalebox{0.8}{$\substack{1 \\ 12321}$}} \LaTeXunderbrace{+\cdots+}_{\,\scalebox{0.8}{$\substack{1 \\ 12221}$}}
			) & \text{if $i = 4$}, \\
			( 
			\LaTeXunderbrace{-\cdots-}_{\,\scalebox{0.8}{$\substack{0 \\ 00011}$}} \LaTeXunderbrace{+\cdots+}_{\,\scalebox{0.8}{$\substack{0 \\ 00001}$}}
			\LaTeXunderbrace{-\cdots-}_{\,\scalebox{0.8}{$\substack{1 \\ 01221}$}} \LaTeXunderbrace{+\cdots+}_{\,\scalebox{0.8}{$\substack{1 \\ 01211}$}}
			\LaTeXunderbrace{-\cdots-}_{\,\scalebox{0.8}{$\substack{1 \\ 11221}$}} \LaTeXunderbrace{+\cdots+}_{\,\scalebox{0.8}{$\substack{1 \\ 11211}$}}
			\LaTeXunderbrace{-\cdots-}_{\,\scalebox{0.8}{$\substack{1 \\ 12221}$}} \LaTeXunderbrace{+\cdots+}_{\,\scalebox{0.8}{$\substack{1 \\ 12211}$}}
			) & \text{if $i = 5$},
		\end{cases}
		}$
		}
	\end{equation*}\!\!
	and $\sigma_{i, 2}({\bf c})$ is equal to the sequence for $k$ such that $k^{\star} = i$ obtained from \cite[Proposition 3.2]{JK19} to the pair $({\mf l},\, {\bf i}_{\text{J}})$ by regarding $\substack{b \\ acde0}$ as the root of type $\text{D}_5$.
	\vskip 2mm
	
	\item in type $\text{E}_7$ with $r=7$, we have
	\begin{equation*}
		\sigma_i({\bf c}) = \sigma_{i, 1}({\bf c}) \cdot \sigma_{i, 2}({\bf c}),
	\end{equation*}
	where $\sigma_{i, 1}({\bf c})$ is given by 
	\begin{equation*}
	{\allowdisplaybreaks
		\,\,\,\,\,\begin{cases}
			 \scalebox{0.75}{$\big(\,
			\LaTeXunderbrace{-\cdots-}_{\,\scalebox{0.8}{$\substack{0\,\,\, \\ 111111}$}} \LaTeXunderbrace{+\cdots+}_{\,\scalebox{0.8}{$\substack{0\,\,\, \\ 011111}$}}
			\LaTeXunderbrace{-\cdots-}_{\,\scalebox{0.8}{$\substack{1\,\,\, \\ 111111}$}} \LaTeXunderbrace{+\cdots+}_{\,\scalebox{0.8}{$\substack{1\,\,\, \\ 011111}$}}
			\LaTeXunderbrace{-\cdots-}_{\,\scalebox{0.8}{$\substack{1\,\,\, \\ 112111}$}} \LaTeXunderbrace{+\cdots+}_{\,\scalebox{0.8}{$\substack{1\,\,\, \\ 012111}$}}
			\LaTeXunderbrace{-\cdots-}_{\,\scalebox{0.8}{$\substack{1\,\,\, \\ 112211}$}} \LaTeXunderbrace{+\cdots+}_{\,\scalebox{0.8}{$\substack{1\,\,\, \\ 012211}$}}
			\LaTeXunderbrace{-\cdots-}_{\,\scalebox{0.8}{$\substack{1\,\,\, \\ 112221}$}} \LaTeXunderbrace{+\cdots+}_{\,\scalebox{0.8}{$\substack{1\,\,\, \\ 012221}$}}
			\LaTeXunderbrace{-\cdots-}_{\,\scalebox{0.8}{$\substack{2\,\,\, \\ 234321}$}} \LaTeXunderbrace{+\cdots+}_{\,\scalebox{0.8}{$\substack{2\,\,\, \\ 134321}$}}
			\,\big)$}
			& \text{if $i = 1$}, \\
			\scalebox{0.75}{$\big(\,
			\LaTeXunderbrace{-\cdots-}_{\,\scalebox{0.8}{$\substack{1\,\,\, \\ 001111}$}} \LaTeXunderbrace{+\cdots+}_{\,\scalebox{0.8}{$\substack{0\,\,\, \\ 011111}$}}
			\LaTeXunderbrace{-\cdots-}_{\,\scalebox{0.8}{$\substack{1\,\,\, \\ 011111}$}} \LaTeXunderbrace{+\cdots+}_{\,\scalebox{0.8}{$\substack{0\,\,\, \\ 011111}$}}
			\LaTeXunderbrace{-\cdots-}_{\,\scalebox{0.8}{$\substack{1\,\,\, \\ 111111}$}} \LaTeXunderbrace{+\cdots+}_{\,\scalebox{0.8}{$\substack{0\,\,\, \\ 111111}$}}
			\LaTeXunderbrace{-\cdots-}_{\,\scalebox{0.8}{$\substack{2\,\,\, \\ 123211}$}} \LaTeXunderbrace{+\cdots+}_{\,\scalebox{0.8}{$\substack{1\,\,\, \\ 123211}$}}
			\LaTeXunderbrace{-\cdots-}_{\,\scalebox{0.8}{$\substack{2\,\,\, \\ 123221}$}} \LaTeXunderbrace{+\cdots+}_{\,\scalebox{0.8}{$\substack{1\,\,\, \\ 123221}$}}
			\LaTeXunderbrace{-\cdots-}_{\,\scalebox{0.8}{$\substack{2\,\,\, \\ 123321}$}} \LaTeXunderbrace{+\cdots+}_{\,\scalebox{0.8}{$\substack{1\,\,\, \\ 123321}$}}
			\,\big)$} & \text{if $i = 2$}, \\
			\scalebox{0.75}{$ \big(\,
			\LaTeXunderbrace{-\cdots-}_{\,\scalebox{0.8}{$\substack{0\,\,\, \\ 011111}$}} \LaTeXunderbrace{+\cdots+}_{\,\scalebox{0.8}{$\substack{0\,\,\, \\ 001111}$}}
			\LaTeXunderbrace{-\cdots-}_{\,\scalebox{0.8}{$\substack{1\,\,\, \\ 011111}$}} \LaTeXunderbrace{+\cdots+}_{\,\scalebox{0.8}{$\substack{1\,\,\, \\ 001111}$}}
			\LaTeXunderbrace{-\cdots-}_{\,\scalebox{0.8}{$\substack{1\,\,\, \\ 122111}$}} \LaTeXunderbrace{+\cdots+}_{\,\scalebox{0.8}{$\substack{1\,\,\, \\ 112111}$}}
			\LaTeXunderbrace{-\cdots-}_{\,\scalebox{0.8}{$\substack{1\,\,\, \\ 122211}$}} \LaTeXunderbrace{+\cdots+}_{\,\scalebox{0.8}{$\substack{1\,\,\, \\ 112211}$}}
			\LaTeXunderbrace{-\cdots-}_{\,\scalebox{0.8}{$\substack{1\,\,\, \\ 122221}$}} \LaTeXunderbrace{+\cdots+}_{\,\scalebox{0.8}{$\substack{1\,\,\, \\ 112221}$}}
			\LaTeXunderbrace{-\cdots-}_{\,\scalebox{0.8}{$\substack{2\,\,\, \\ 134321}$}} \LaTeXunderbrace{+\cdots+}_{\,\scalebox{0.8}{$\substack{2\,\,\, \\ 124321}$}}
			\,\big) $} & \text{if $i = 3$}, \\
			\scalebox{0.75}{$\big(\,
			\LaTeXunderbrace{-\cdots-}_{\,\scalebox{0.8}{$\substack{0\,\,\, \\ 001111}$}} \LaTeXunderbrace{+\cdots+}_{\,\scalebox{0.8}{$\substack{0\,\,\, \\ 000111}$}}
			\LaTeXunderbrace{-\cdots-}_{\,\scalebox{0.8}{$\substack{1\,\,\, \\ 012111}$}} \LaTeXunderbrace{+\cdots+}_{\,\scalebox{0.8}{$\substack{1\,\,\, \\ 011111}$}}
			\LaTeXunderbrace{-\cdots-}_{\,\scalebox{0.8}{$\substack{1\,\,\, \\ 112111}$}} \LaTeXunderbrace{+\cdots+}_{\,\scalebox{0.8}{$\substack{1\,\,\, \\ 111111}$}}
			\LaTeXunderbrace{-\cdots-}_{\,\scalebox{0.8}{$\substack{1\,\,\, \\ 123211}$}} \LaTeXunderbrace{+\cdots+}_{\,\scalebox{0.8}{$\substack{1\,\,\, \\ 122211}$}}
			\LaTeXunderbrace{-\cdots-}_{\,\scalebox{0.8}{$\substack{1\,\,\, \\ 123221}$}} \LaTeXunderbrace{+\cdots+}_{\,\scalebox{0.8}{$\substack{1\,\,\, \\ 122221}$}}
			\LaTeXunderbrace{-\cdots-}_{\,\scalebox{0.8}{$\substack{2\,\,\, \\ 124321}$}} \LaTeXunderbrace{+\cdots+}_{\,\scalebox{0.8}{$\substack{2\,\,\, \\ 123321}$}}
			\,\big) $} & \text{if $i = 4$}, \\
			\scalebox{0.75}{$ \big(\,
			\LaTeXunderbrace{-\cdots-}_{\,\scalebox{0.8}{$\substack{0\,\,\, \\ 000111}$}} \LaTeXunderbrace{+\cdots+}_{\,\scalebox{0.8}{$\substack{0\,\,\, \\ 000011}$}}
			\LaTeXunderbrace{-\cdots-}_{\,\scalebox{0.8}{$\substack{1\,\,\, \\ 012211}$}} \LaTeXunderbrace{+\cdots+}_{\,\scalebox{0.8}{$\substack{1\,\,\, \\ 012111}$}}
			\LaTeXunderbrace{-\cdots-}_{\,\scalebox{0.8}{$\substack{1\,\,\, \\ 112211}$}} \LaTeXunderbrace{+\cdots+}_{\,\scalebox{0.8}{$\substack{1\,\,\, \\ 112111}$}}
			\LaTeXunderbrace{-\cdots-}_{\,\scalebox{0.8}{$\substack{1\,\,\, \\ 122211}$}} \LaTeXunderbrace{+\cdots+}_{\,\scalebox{0.8}{$\substack{1\,\,\, \\ 122111}$}}
			\LaTeXunderbrace{-\cdots-}_{\,\scalebox{0.8}{$\substack{1\,\,\, \\ 123321}$}} \LaTeXunderbrace{+\cdots+}_{\,\scalebox{0.8}{$\substack{1\,\,\, \\ 123221}$}}
			\LaTeXunderbrace{-\cdots-}_{\,\scalebox{0.8}{$\substack{2\,\,\, \\ 123321}$}} \LaTeXunderbrace{+\cdots+}_{\,\scalebox{0.8}{$\substack{2\,\,\, \\ 123221}$}}
			\,\big) $} & \text{if $i = 5$}, \\
			\scalebox{0.75}{$ \big(\,
			\LaTeXunderbrace{-\cdots-}_{\,\scalebox{0.8}{$\substack{0\,\,\, \\ 000111}$}} \LaTeXunderbrace{+\cdots+}_{\,\scalebox{0.8}{$\substack{0\,\,\, \\ 000011}$}}
			\LaTeXunderbrace{-\cdots-}_{\,\scalebox{0.8}{$\substack{1\,\,\, \\ 012211}$}} \LaTeXunderbrace{+\cdots+}_{\,\scalebox{0.8}{$\substack{1\,\,\, \\ 012111}$}}
			\LaTeXunderbrace{-\cdots-}_{\,\scalebox{0.8}{$\substack{1\,\,\, \\ 112211}$}} \LaTeXunderbrace{+\cdots+}_{\,\scalebox{0.8}{$\substack{1\,\,\, \\ 112111}$}}
			\LaTeXunderbrace{-\cdots-}_{\,\scalebox{0.8}{$\substack{1\,\,\, \\ 122211}$}} \LaTeXunderbrace{+\cdots+}_{\,\scalebox{0.8}{$\substack{1\,\,\, \\ 122111}$}}
			\LaTeXunderbrace{-\cdots-}_{\,\scalebox{0.8}{$\substack{1\,\,\, \\ 123321}$}} \LaTeXunderbrace{+\cdots+}_{\,\scalebox{0.8}{$\substack{1\,\,\, \\ 123221}$}}
			\LaTeXunderbrace{-\cdots-}_{\,\scalebox{0.8}{$\substack{2\,\,\, \\ 123321}$}} \LaTeXunderbrace{+\cdots+}_{\,\scalebox{0.8}{$\substack{2\,\,\, \\ 123221}$}}
			\,\big) $} & \text{if $i = 6$},
		\end{cases}
		}
	\end{equation*}\!\!
	and $\sigma_{i, 2}({\bf c})$ is equal to the sequence $\sigma_k({\bf c})$ for $k$ such that $k^{\star} = i$ in (2) by regarding $\substack{b\,\,\, \\ acdef0}$ as the root of type $\text{E}_6$.
\end{enumerate}
}
\end{prop}
\begin{proof}
By Theorem \ref{thm:signature rule} and Remark \ref{rem:tf_i}, the sequence $\sigma_i ({\bf c})$ is obtained from $\overline{\bf i}^{\J}_k \cdot {\bf i}_{\J}$ such that $k^{\star} = i$ (cf. \eqref{eq:ov iJ for E6} and \eqref{eq:ov iJ for E7}). 
\end{proof}

Set
\begin{equation} \label{eq:B^J and B_J}
\begin{split}
	&\B^{\J} = \left\{ \, {\bf c} = (c_{\beta}) \in \B \, \big| \,\,  c_{\beta} = 0 \,\, \text{unless} \,\, \beta \in \Phi^+(\J) \, \right\},
	\\
	&\B_{\J} = \left\{ \, {\bf c} = (c_{\beta}) \in \B \, \big| \,\, c_{\beta} = 0 \,\, \text{unless} \,\,  \beta \in \Phi^+_{\J} \, \right\},
\end{split}
\end{equation}
which we regard them as subcrystals of $\B$, where we assume that $\te_r {\bf c} = \tf_r {\bf c}  = {\bf 0}$ with $\varepsilon_r({\bf c}) = \varphi_r({\bf c}) = -\infty$ for ${\bf c} \in \B_{\J}$.
Then the crystal $\B^{\J}$ can be viewed as {\em the crystal of the quantum nilpotent subalgebra} $U_q(\wJ)$.

\begin{rem}
{\em
The crystal operators $\te_i$ and $\tf_i$ $(i \in \J)$ on $\B^{\J}$ and $\B_{\J}$ are described by the rule in Theorem \ref{thm:signature rule} with respect to the sequences $\sigma_{i, 1}$ and $\sigma_{i, 2}$ in Proposition \ref{prop:signatures}, respectively.
Note that it is well-known in \cite{Lu10} (cf. \cite{Kas91, Sai}) that for ${\bf c} \in \B$, 
\begin{equation*}
	\te_r {\bf c} = 
	\begin{cases}
		{\bf c} - {\bf 1}_{\alpha_r} & \text{if $c_{\alpha_r} \neq 0$}, \\
		{\bf 0} & \text{otherwise},
	\end{cases}
	\qquad \,\,\,
	\tf_r {\bf c} = {\bf c} + {\bf 1}_{\alpha_r},
\end{equation*}
since $\beta_1$ in the convex order of ${\bf i}_0$ is equal to $\alpha_r$.
}
\end{rem}

\begin{rem} \label{rem:signature}
{\em
For ${\bf c} = (c_{\beta_k}) \in \B^{\J}$, we often identify it with an array of $c_{\beta_k}$'s located at the position of $\beta_k$ on $\Delta^{\J}$ (recall \eqref{eq:arrangement of positive roots of nilradical in E6} and \eqref{eq:arrangement of positive roots of nilradical in E7}). Then the sequence $\sigma_i({\bf c})$ for $\tf_i$ $(i \neq r)$ is recorded easily under this identification. Let us explain it in more details.

For type $\text{E}_6$ with $r=6$, the array is given by
\begin{equation*}
\scalebox{1}{
\begin{tikzpicture}[->,>=stealth',auto,node distance=3cm,main node/.style={circle,draw,font=\sffamily\Large\bfseries}]

  \node (1) at (-0.2,0) {${}_{\scalebox{0.7}{$c_{\beta_1}$}}$}; 
  \node (2) at (1.2,0) {${}_{\scalebox{0.7}{$c_{\beta_2}$}}$}; 
  \node (3) at (2.6,0) {${}_{\scalebox{0.7}{$c_{\beta_3}$}}$}; 
  \node (4) at (4,0) {${}_{\scalebox{0.7}{$c_{\beta_4}$}}$}; 
  \node (5) at (5.4,0) {${}_{\scalebox{0.7}{$c_{\beta_5}$}}$}; 
  
  \node (6) at (2.6, -1.2) {${}_{\scalebox{0.7}{$c_{\beta_6}$}}$}; 
  \node (7) at (4, -1.2) {${}_{\scalebox{0.7}{$c_{\beta_7}$}}$}; 
  \node (8) at (5.4, -1.2) {${}_{\scalebox{0.7}{$c_{\beta_8}$}}$}; 
  
  \node (9) at (4, -2.6) {${}_{\scalebox{0.7}{$c_{\beta_9}$}}$}; 
  \node (10) at (5.4, -2.6) {${}_{\scalebox{0.7}{$c_{\beta_{10}}$}}$}; 
  \node (11) at (6.8, -2.6) {${}_{\scalebox{0.7}{$c_{\beta_{11}}$}}$}; 
  
  \node (12) at (4, -4) {${}_{\scalebox{0.7}{$c_{\beta_{12}}$}}$}; 
  \node (13) at (5.4, -4) {${}_{\scalebox{0.7}{$c_{\beta_{13}}$}}$}; 
  \node (15) at (6.8, -4) {${}_{\scalebox{0.7}{$c_{\beta_{14}}$}}$}; 
  \node (14) at (8.2, -4) {${}_{\scalebox{0.7}{$c_{\beta_{15}}$}}$}; 
  \node (16) at (9.6, -4) {${}_{\scalebox{0.7}{$c_{\beta_{16}}$}}$}; 
  
    \path[every node/.style={font=\sffamily\small}]
	(1) edge[] node [above] {${}_{5}$} (2)
	(9) edge[] node [right] {${}_{5}$} (12)
	(10) edge[] node [right] {${}_{5}$} (13)
	(11) edge[] node [right] {${}_{5}$} (15)
	
    	(2) edge[] node [above] {${}_{4}$} (3)
	(7) edge[] node [right] {${}_{4}$} (9)
	(8) edge[] node [right] {${}_{4}$} (10)
	(15) edge[] node [below] {${}_{4}$} (14)
	
	(3) edge[] node [above] {${}_{3}$} (4)
	(6) edge[] node [above] {${}_{3}$} (7)
	(10)  edge[] node [below] {${}_{3}$} (11)
	(13) edge[] node [below] {${}_{3}$} (15)
	
	(3) edge[] node [right] {${}_{2}$} (6)
	(4) edge[] node [right] {${}_{2}$} (7)
	(5)  edge[] node [right] {${}_{2}$} (8)
	(14) edge[] node [below] {${}_{2}$} (16)
	
	(4) edge[] node [above] {${}_{1}$} (5)
	(7) edge[] node [above] {${}_{1}$} (8)
	(9)  edge[] node [below] {${}_{1}$} (10)
	(12) edge[] node [below] {${}_{1}$} (13)
	;
\end{tikzpicture}
}
\end{equation*}
where $\beta_k$'s are enumerated by the convex ordering \eqref{eq:positive roots of E6} and the arrow is determined by
$c_{\beta} \, \overset{i}{\longrightarrow} \, c_{\beta'}$ if and only if $\beta' - \beta = \alpha_i$.
The sequence $\sigma_i({\bf c})$ is obtained by reading $-$'s and $+$
's as $c_{\beta_k}$ located in the terminal and initial points of the arrows labelled by $i$ from north-west to south-east.
The sequences $\sigma_i({\bf c})$ for the case of $r = 1$ are also given by the similar rule.

For type $\text{E}_7$, the array for ${\bf c} \in \B^{\J}$ with the arrows labelled by $i$ $(i \in \J)$ is given by
\begin{equation*}
\scalebox{1}{
\begin{tikzpicture}[->,>=stealth',auto,node distance=3cm,main node/.style={circle,draw,font=\sffamily\Large\bfseries}]
  \node (1) at (-1.6,0) {\scalebox{0.7}{$c_{\beta_1}$}}; 
  \node (2) at (-0.2,0) {\scalebox{0.7}{$c_{\beta_2}$}}; 
  \node (3) at (1.2,0) {\scalebox{0.7}{$c_{\beta_3}$}}; 
  \node (4) at (2.6,0) {\scalebox{0.7}{$c_{\beta_4}$}}; 
  \node (5) at (4,0) {\scalebox{0.7}{$c_{\beta_5}$}}; 
  \node (6) at (5.4,0) {\scalebox{0.7}{$c_{\beta_6}$}}; 

  \node (7) at (2.6, -1.2) {\scalebox{0.7}{$c_{\beta_7}$}}; 
  \node (8) at (4, -1.2) {\scalebox{0.7}{$c_{\beta_8}$}}; 
  \node (9) at (5.4, -1.2) {\scalebox{0.7}{$c_{\beta_9}$}}; 

  \node (10) at (4, -2.6) {\scalebox{0.7}{$c_{\beta_{10}}$}}; 
  \node (11) at (5.4, -2.6) {\scalebox{0.7}{$c_{\beta_{11}}$}}; 
  \node (12) at (6.8, -2.6) {\scalebox{0.7}{$c_{\beta_{12}}$}}; 

  \node (13) at (4, -4) {\scalebox{0.7}{$c_{\beta_{13}}$}}; 
  \node (14) at (5.4, -4) {\scalebox{0.7}{$c_{\beta_{14}}$}}; 
  \node (15) at (6.8, -4) {\scalebox{0.7}{$c_{\beta_{15}}$}}; 
  \node (16) at (8.2, -4) {\scalebox{0.7}{$c_{\beta_{16}}$}}; 
  \node (17) at (9.6, -4) {\scalebox{0.7}{$c_{\beta_{17}}$}}; 

  \node (18) at (4, -5.4) {\scalebox{0.7}{$c_{\beta_{18}}$}}; 
  \node (19) at (5.4, -5.4) {\scalebox{0.7}{$c_{\beta_{19}}$}}; 
  \node (20) at (6.8, -5.4) {\scalebox{0.7}{$c_{\beta_{20}}$}}; 
  \node (21) at (8.2, -5.4) {\scalebox{0.7}{$c_{\beta_{21}}$}}; 
  \node (22) at (9.6, -5.4) {\scalebox{0.7}{$c_{\beta_{22}}$}}; 
  
  \node (23) at (8.2, -6.8) {\scalebox{0.7}{$c_{\beta_{23}}$}}; 
  \node (24) at (9.6, -6.8) {\scalebox{0.7}{$c_{\beta_{24}}$}}; 

  \node (25) at (9.6, -8.2) {\scalebox{0.7}{$c_{\beta_{25}}$}}; 

  \node (26) at (9.6, -9.6) {\scalebox{0.7}{$c_{\beta_{26}}$}}; 
  
  \node (27) at (9.6, -11) {\scalebox{0.7}{$c_{\beta_{27}}$}}; 
  
  \path[every node/.style={font=\sffamily\small}]
	(1) edge[] node [above] {${}_{6}$} (2)
	(13) edge[] node [right] {${}_{6}$} (18)
	(14) edge[] node [right] {${}_{6}$} (19)
	(15) edge[] node [right] {${}_{6}$} (20)
	(16) edge[] node [right] {${}_{6}$} (21)
	(17) edge[] node [right] {${}_{6}$} (22)
	
    	(2) edge[] node [above] {${}_{5}$} (3)
	(10) edge[] node [right] {${}_{5}$} (13)
	(11) edge[] node [right] {${}_{5}$} (14)
	(12) edge[] node [right] {${}_{5}$} (15)
	(21) edge[] node [right] {${}_{5}$} (23)
	(22) edge[] node [right] {${}_{5}$} (24)
	
	(3) edge[] node [above] {${}_{4}$} (4)
	(8) edge[] node [right] {${}_{4}$} (10)
	(9)  edge[] node [right] {${}_{4}$} (11)
	(15) edge[] node [below] {${}_{4}$} (16)
	(20) edge[] node [below] {${}_{4}$} (21)
	(24) edge[] node [right] {${}_{4}$} (25)
	
	(4) edge[] node [above] {${}_{3}$} (5)
	(7) edge[] node [above] {${}_{3}$} (8)
	(11)  edge[] node [below] {${}_{3}$} (12)
	(14) edge[] node [below] {${}_{3}$} (15)
	(19) edge[] node [below] {${}_{3}$} (20)
	(25) edge[] node [right] {${}_{3}$} (26)
	
	(5) edge[] node [above] {${}_{1}$} (6)
	(8) edge[] node [above] {${}_{1}$} (9)
	(10)  edge[] node [below] {${}_{1}$} (11)
	(13) edge[] node [below] {${}_{1}$} (14)
	(18) edge[] node [below] {${}_{1}$} (19)
	(26) edge[] node [right] {${}_{1}$} (27)
	
	(4) edge[] node [right] {${}_{2}$} (7)
	(5) edge[] node [right] {${}_{2}$} (8)
	(6) edge[] node [right] {${}_{2}$} (9)
	(16) edge[] node [below] {${}_{2}$} (17)
	(21) edge[] node [below] {${}_{2}$} (22)
	(23) edge[] node [below] {${}_{2}$} (24)
	;
\end{tikzpicture}
}
\end{equation*}
where $\beta_k$'s are enumerated by the convex ordering \eqref{eq:positive roots of E7}.
}
\end{rem}

The following is the type E analog of \cite[Theorem 4.2]{Kw18} (cf. \cite[Corollary 3.5]{JK19} for type D).

\begin{cor}
{\em
\mbox{} \

\begin{itemize}
	\item[(1)] The crystal $\B_{\J}$ is isomorphic to the crystal of $U_q^-({\mf l})$ as an $\mf l$-crystal.
	\vskip 1mm
	
	\item[(2)] The crystal $\B$ is isomorphic to $\B^{\J} \otimes \B_{\J}$ as an $\g$-crystal by the map ${\bf c} \mapsto\! {\bf c}^{\J} \otimes\, {\bf c}_{\J}$.
\end{itemize}
}
\end{cor}
\begin{proof}
(1) It follows from comparing the crystal structure of $U_q^-({\mf l})$ given in \cite[Section 3.1]{JK19}. 

(2) It follows from Theorem \ref{thm:signature rule}, Proposition \ref{prop:signatures} and the tensor product rule \eqref{eq:tensor_product_rule}.
\end{proof}

We have the characterizations of $\B^{\J}$ and $B(s\varpi_r)$ for $s \ge 1$ by using the $\varepsilon_i^*$-statistics as follows.

\begin{prop} \label{prop:description of B^J}
{\em 
\mbox{} \

\begin{itemize}
	\item[(1)]  We have
\begin{equation*}
	\B^{\J} = \left\{ \, {\bf c} \in \B \, | \, \varepsilon_i^*({\bf c}) = 0 \,\, \text{for $i \in \J$} \, \right\}.
\end{equation*}
	\vskip 1mm
	
	\item[(2)] For $s \ge 1$, let
	\begin{equation*}
		\B^{\J, s} := \left\{\, {\bf c} \in \B^{\J} \, | \, \varepsilon_r^*({\bf c}) \le s \,\right\}
	\end{equation*}
	and it is regarded as a subcrystal of $\B^{\J}$. As $\g$-crystals, we have
	\begin{equation*}
		\B^{\J, s} \otimes T_{s\varpi_r} \cong B(s\varpi_r), \qquad
		\B^{\J} = \bigcup_{s \ge 1} \B^{\J, s}.
	\end{equation*}
\end{itemize}
}
\end{prop}
\begin{proof}
(1) follows from \cite[Section 2.1]{Lu90}, and
(2) follows from (1) and \cite[Proposition 8.2]{Kas95}.
\end{proof}

\section{Kirillov-Reshetikhin crystals of type E with minuscule nodes} \label{sec:polytope of KR crystals}

Let $\ag$ be the affine Lie algebra of types $\text{E}_6^{(1)}$ or $\text{E}_7^{(1)}$ with the index set $\aI$ and the Dynkin diagram given by
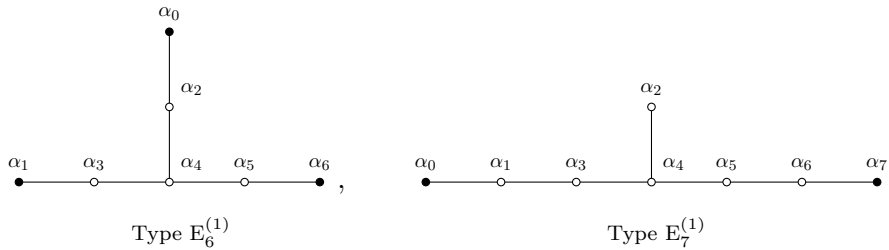
\begin{equation*}
\begin{split}
  &\,\,\, \dynkin[extended, edge length=1cm, o/.append style={ﬁll=black}, */.append style={ﬁll=white}, labels*={\alpha_0, \alpha_1, \alpha_2, \alpha_3, \alpha_4, \alpha_5, \alpha_6}] E{o****o},\qquad
  \dynkin[extended, edge length=1cm, o/.append style={ﬁll=black}, */.append style={ﬁll=white}, labels*={\alpha_0, \alpha_1, \alpha_2, \alpha_3, \alpha_4, \alpha_5, \alpha_6, \alpha_7}] E{******o} \\
  & \,\,\, \qquad \qquad\,\,\, \text{\scriptsize Type $\text{E}_6^{(1)}$ \,\,\,} \qquad \qquad \qquad \qquad \qquad \qquad\,\, \text{\scriptsize Type $\text{E}_7^{(1)}$ \,\,\,}  
\end{split}
\end{equation*}

For $k$ in the orbit of $0$, that is, 
\begin{equation*} 
	k \in
	\begin{cases}
		\{ \, 0, 1, 6 \, \} & \text{for type $\Esix$,} \\
		\{ \, 0, 7 \, \} & \text{for type $\Eseven$,}
	\end{cases}
\end{equation*}
let $\ag_k$ be the subalgebra of $\ag$ corresponding to $\{ \, \alpha_i \, | \, i \in \aI \,\setminus \, \{ k \} \, \}$.
Note that $\ag_0 = \g$ and $\ag_0 \cap \ag_k = {\mf l}$ for $k \neq 0$.
Let $\hat{P} = \bigoplus_{i \in \aI} \mathbb{Z}\Lambda_i \oplus \mathbb{Z}\delta$ be the weight lattice of $\ag$, where $\delta$ is the positive imaginary null root and $\Lambda_i$ is the $i$th fundamental weight of $\ag$ \cite{K}.
Let $\theta$ be the maximal root of $\g$. Then $\alpha_0 = -\theta$ in $\hat{P} / \mathbb{Z}\delta$.

For ${\bf c} \in \B^{\J}$, we define
\begin{equation*} 
\begin{split}
	\te_0 {\bf c} =& {\bf c} + {\bf 1}_{\theta}, \qquad \quad
	\tf_0 {\bf c} =
	\begin{cases}
		{\bf c} - {\bf 1}_{\theta} & \text{if $c_{\theta} > 0$,} \\
		{\bf 0} & \text{otherwise,}
	\end{cases} \\
	\varphi_0 ({\bf c}) &= \max \{ \, k \,\, | \,\, \tf_0^k {\bf c} \neq 0 \, \}, \qquad
	\varepsilon_0 ({\bf c}) = \varphi_0 ({\bf c}) - \langle {\rm wt}({\bf c}),\, h_0 \rangle.
\end{split}
\end{equation*}
The set $\B^{\J}$ is a $\ag$-crystal with respect to ${\rm wt}$, $\varepsilon_i$, $\varphi_i$, $\te_i$, $\tf_i$ for $i \in \aI$, where ${\rm wt}$ is the restriction of ${\rm wt} : \B \rightarrow P$ to $\B^{\J}$, and then we regard the subcrystal $\B^{\J, s}$ of $\B^{\J}$ as a $\ag$-crystal. 

Now we are in position to state the main result in this paper.
\begin{thm} \label{thm:main theorem}
{\em 
For a minuscule $\varpi_r$ and $s \ge 1$, the $\ag$-crystal $\B^{\J}$ is regular and
\begin{equation*}
	\B^{\J, s} \otimes T_{s\varpi_r} \cong B^{r, s},
\end{equation*} 
where $B^{r, s}$ is the KR crystal of types $\Esix$ and $\Eseven$ associated to $s\varpi_r$.
}
\end{thm}
\begin{proof}
By Proposition \ref{prop:description of B^J}(2), $\B^{\J, s} \otimes T_{s\varpi_r}$ is a regular $\ag_0$-crystal isomorphic to $B(s\varpi_r)$.
By Proposition \ref{prop:signatures} (see also Remark \ref{rem:signature}), one can check that the $\ag_r$-crystal $\B^{\J, s} \otimes T_{s\varpi_r}$ is isomorphic to the dual crystal of $\ag_0$-crystal $\B^{\J, s} \otimes T_{s\varpi_r}$ by regarding $\ag_r \cong \ag_0$ under the following correspondence
\begin{equation*}
	\begin{cases}
		\alpha_0 \, \leftrightarrow \, \alpha_1, \quad \alpha_3 \, \leftrightarrow \, \alpha_2 & \text{for type $\Esix$ with $r=1$,} \\
		\alpha_0 \, \leftrightarrow \, \alpha_6, \quad \alpha_5 \, \leftrightarrow \, \alpha_2 & \text{for type $\Esix$ with $r=6$,} \\
	    \alpha_0 \, \leftrightarrow \, \alpha_7, \quad \alpha_1 \, \leftrightarrow \, \alpha_6, \quad \alpha_3 \, \leftrightarrow \, \alpha_5 & \text{for type $\Eseven$ with $r = 7$.}
	\end{cases}
\end{equation*}
This implies that
the $\ag_r$-crystal $\B^{\J, s} \otimes T_{s\varpi_r}$ is isomorphic to $B(s\varpi_{r^*})$ under the above correspondence, where $r^*$ is determined by $w_0(\alpha_r) = -\alpha_{r^*}$ (cf. \cite{Bo}).
It was well-known in \cite{Cha01} that $B^{r, s}$ is classically irreducible, that is, $B^{r, s} \cong B(s\varpi_r)$ as a $\ag_0$-crystal. Furthermore, $B^{r, s} \cong B(s\varpi_{r^*})$ as a $\ag_r$-crystal . 
Hence, it follows from \cite[Theorem 3.15]{JS} (see Remark \ref{rem:uniqueness E7} for type $\Eseven$) that $\B^{\J, s} \otimes T_{s\varpi_r}$ is isomorphic to $B^{r, s}$ as a $\ag$-crystal.
\end{proof}

\begin{rem} \label{rem:uniqueness E7}
{\em 
Recall that ${\mf l}$ is of type $\text{E}_6$ when $\ag$ is of type $\Eseven$.
Since the ${\mf l}$-highest weight vectors $b$ in $B^{7, s}$ are distinguished by their weight with the value $\langle {\rm wt}(b),\, \alpha_7^{\vee} \rangle$ \cite[Section 3.4]{BS}, it allows us to obtain an analog of \cite[Theorem 3.15]{JS} for type $\Eseven$ with $r=7$ by following the proof of \cite[Theorem 3.15]{JS} with $K = \aI \, \setminus \, \{ \, 0, 7 \, \}$ since $\langle {\rm wt}(b), \, \alpha_7^{\vee} \rangle = \langle {\rm wt}(b),\, \alpha_0^{\vee} \rangle$ for a ${\mf l}$-highest weight vector $b$ by the level-zero condition (cf. \cite[Remark 3.4]{JS}).
}
\end{rem}

\begin{rem}
{\em 
By Proposition \ref{prop:description of B^J} and Theorem \ref{thm:main theorem}, the crystal $\B^{\J}$ can viewed as the limit of KR crystals $B^{r, s}$ as $s \rightarrow \infty$.
}
\end{rem}

\begin{rem}
{\em 
After this paper was submitted, Hiroshima in \cite{Hiro} proved the perfectness of KR crystal $B^{r, s}$ for a minuscule $r$ and $s \ge 1$ in type $\text{E}_{6,7}^{(1)}$ based on Theorem \ref{thm:main theorem}.
}
\end{rem}

\section{Combinatorial description of the $\varepsilon_r^*$-statistic} \label{subsec:epsilon}
\subsection{Combinatorial formula for $\varepsilon_r^*$}
In this section, we give a combinatorial formula of the statistic $\varepsilon_r^*$ on $\B^{\J}$ for type $\text{E}_n$ in terms of the triple or quadruple paths on $\Delta_{n+3}$, respectively.

\begin{df} {\em (cf. \cite[Definition 3.10]{JK19})} \label{df:paths}
{\em 
Let $\beta$ be a dot in $\Delta_n$.

\begin{enumerate}	
	\item A {\em (single) path} on $\Delta_n$  is a sequence $p = (\gamma_1, \dots, \gamma_s)$ of dots in $\Delta_n$ for some $s \ge 1$ such that  the position $(i_k, j_k)$ of $\gamma_k$ is equal to $(i_{k-1}+1, j_{k-1})$ or $(i_{k-1}, j_{k-1}+1)$ for all $2 \le k \le s$.
	\vskip 2mm
	
	\item A {\em double path} at $\beta$ is a pair of paths ${\bf p} = (p_1, p_2)$ in $\Delta$ of the same length with $p_1 = (\gamma_1, \dots, \gamma_s)$ and $p_2 = (\delta_1, \dots, \delta_s)$ such that $\gamma_1 = \delta_1 = \beta$ and $\gamma_i$ is located strictly diagonally left of $\delta_i$ for $2 \le i \le s$.
	\vskip 2mm
			
	\item A {\em triple path at $\beta$} is a sequence of paths ${\bf p} = (p_1, p_2, p_3)$ such that each pair $(p_k, p_{k+1})$ is a double path at $\beta_k$ for $k = 1, 2$, where $\beta_1 = \beta$ and $\beta_2$ is located strictly to the left of or below $\beta_1$ by one position.
	\vskip 2mm
	
	\item A {\em quadruple path at $\beta$} is a sequence of paths ${\bf p} = (p_1, p_2, p_3, p_4)$  such that each pair $(p_k, p_{k+1})$ is a double path at $\beta_k$ for $1 \le k \le 3$, where $\beta_1 = \beta$ and $\beta_k$ is located strictly to the left of or below $\beta_{k-1}$ by one position for all $k = 2, 3$.

\end{enumerate}
}
\end{df}

\begin{df} \label{df:paths in Pn}
{\em
Let $\beta$ be the dot in $\Delta_{n+3}$ located at the position $(1, 1)$.
We denote by $p_k$ a path in $\Delta_{n+3}$ with $p_k = \big(\alpha_1^{(k)}, \dots, \alpha_s^{(k)}\big)$ for $1 \le k \le 5$.
We define $\mathcal{P}_n$ as follows.
 
\begin{enumerate}
	\item if $n = 6$, then $\mathcal{P}_6$ is the set of the triple paths ${\bf p} = (p_1,\, p_2,\, p_3)$ at $\beta$ in $\Delta_9$ satisfying the following two conditions:
	\vskip 1mm
\begin{enumerate}
	\item[(i)] $\alpha_s^{(k+1)}$ is located strictly diagonally left of $\alpha_s^{(k)}$ by one position for all $k = 1,\, 2$,
	\vskip 2mm
	
	\item[(ii)] $\alpha_s^{(1)}$, $\alpha_s^{(2)}$ and $\alpha_s^{(3)}$ are placed above the $4$-th row of $\Delta_9$ from top.
\end{enumerate}
	\vskip 2mm

	\item if $n=7$, then $\mathcal{P}_7$ is the set of the quadruple paths ${\bf p} = (p_1, p_2, p_3, p_4)$ at $\beta$ in $\Delta_{10}$ passing through the dots located in $(4, 1)$, $(1, 4)$ satisfying
	\vskip 2mm
	\begin{enumerate}
		\item[(i)] $\alpha_s^{(k+1)}$ is located strictly diagonally left of $\alpha_s^{(k)}$ by one position for all $1 \le k \le 3$,
		\vskip 2mm
		
		\item[(ii)] $\alpha_s^{(1)}, \dots, \alpha_s^{(4)}$ are placed above or below the $5$-th row of $\Delta_{10}$ from top,
	\end{enumerate}
	\vskip 2mm
	and the pairs $({\bf p}_1, {\bf p}_2)$ of a double path ${\bf p}_1$ and a triple path ${\bf p}_2$ satisfying
	\vskip 2mm
\begin{enumerate}		
	\item[(iii)] ${\bf p}_1 = (p_1, p_2)$ is a double path at $\beta'$ such that $\alpha_s^{(1)}$ or $\alpha_s^{(2)}$ is placed at $(5, 5)$, where $\beta'$ is located at $(3, 4)$ or $(4, 3)$ in $\Delta_{10}$, and ${\bf p}_2 = (p_3, p_4, p_5)$ is a triple path at $\beta$ such that ${\bf p}_2$ passes through the dots located at $(4, 1)$, $(1, 4)$, $(2, 6)$ and $(6, 2)$, while it does not pass through the dot located at $(3, 3)$ in $\Delta_{10}$, 
	\vskip 2mm
	
	\item[(iv)] for $({\bf p}_1, {\bf p}_2)$, the paths ${\bf p}_1$ and ${\bf p}_2$ are non-intersecting and $\{ \, \alpha_s^{(k)} \, | \, 1 \le k \le 5 \}$ satisfies the condition as in (1)-(i) under an enumeration such that $\alpha_s^{(1)}$ and $\alpha_s^{(2)}$ are placed between $\alpha_s^{(k)}$ for $3 \le k \le 5$ (cf. Example \ref{ex:paths}).
\end{enumerate}
\end{enumerate}
}
\end{df}

\begin{ex} \label{ex:paths}
{\em 
The followings are examples of a triple path in $\mathcal{P}_6$, a quadruple path and a pair of double path and triple path in $\mathcal{P}_7$, respectively.
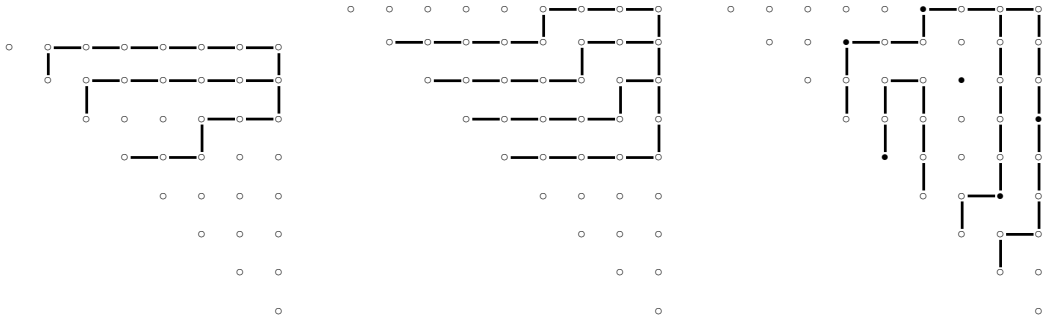
\begin{equation*}
\scalebox{0.73}{
\begin{tikzpicture}
  \node at (-0.1,0) {${}_{\circ}$}; 
  \node at (0.6,0) {${}_{\circ}$}; 
  \node at (1.3,0) {${}_{\circ}$}; 
  \node at (2,0) {${}_{\circ}$}; 
  \node at (2.7,0) {${}_{\circ}$}; 
  \node at (3.4,0) {${}_{\circ}$}; 
  \node at (4.1,0) {${}_{\circ}$}; 
  \node at (4.8,0) {${}_{\circ}$}; 

  \node at (0.6,-0.6) {${}_{\circ}$}; 
  \node at (1.3, -0.6) {${}_{\circ}$}; 
  \node at (2, -0.6) {${}_{\circ}$}; 
  \node at (2.7, -0.6) {${}_{\circ}$}; 
  \node at (3.4,-0.6) {${}_{\circ}$}; 
  \node at (4.1,-0.6) {${}_{\circ}$}; 
  \node at (4.8,-0.6) {${}_{\circ}$}; 

  \node at (1.3, -1.3) {${}_{\circ}$}; 
  \node at (2, -1.3) {${}_{\circ}$}; 
  \node at (2.7, -1.3) {${}_{\circ}$}; 
  \node at (3.4, -1.3) {${}_{\circ}$}; 
  \node at (4.1,-1.3) {${}_{\circ}$}; 
  \node at (4.8,-1.3) {${}_{\circ}$}; 

  \node at (2, -2) {${}_{\circ}$}; 
  \node at (2.7, -2) {${}_{\circ}$}; 
  \node at (3.4, -2) {${}_{\circ}$}; 
  \node at (4.1, -2) {${}_{\circ}$}; 
  \node at (4.8, -2) {${}_{\circ}$}; 

  \node at (2.7, -2.7) {${}_{\circ}$}; 
  \node at (3.4, -2.7) {${}_{\circ}$}; 
  \node at (4.1, -2.7) {${}_{\circ}$}; 
  \node at (4.8, -2.7) {${}_{\circ}$}; 
  
  \node at (3.4, -3.4) {${}_{\circ}$}; 
  \node at (4.1, -3.4) {${}_{\circ}$}; 
  \node at (4.8, -3.4) {${}_{\circ}$}; 
  
  \node at (4.1, -4.1) {${}_{\circ}$}; 
  \node at (4.8, -4.1) {${}_{\circ}$}; 
  
  \node at (4.8, -4.8) {${}_{\circ}$}; 
  
  
  \draw[-, line width=0.5mm] (4.2, 0) -- (4.7, 0);
  \draw[-, line width=0.5mm] (3.5, 0) -- (4, 0);
  \draw[-, line width=0.5mm] (2.8, 0) -- (3.3, 0);
  \draw[-, line width=0.5mm] (2.1, 0) -- (2.6, 0);
  \draw[-, line width=0.5mm] (1.4, 0) -- (1.9, 0);
  \draw[-, line width=0.5mm] (0.7, 0) -- (1.2, 0);
  \draw[-, line width=0.5mm] (0.6, -0.1) -- (0.6, -0.5);
  
  \draw[-, line width=0.5mm] (4.8, -0.1) -- (4.8, -0.5);
  \draw[-, line width=0.5mm] (4.2, -0.6) -- (4.7, -0.6);
  \draw[-, line width=0.5mm] (3.5, -0.6) -- (4, -0.6);
  \draw[-, line width=0.5mm] (2.8, -0.6) -- (3.3, -0.6);
  \draw[-, line width=0.5mm] (2.1, -0.6) -- (2.6, -0.6);
  \draw[-, line width=0.5mm] (1.4, -0.6) -- (1.9, -0.6);
  \draw[-, line width=0.5mm] (1.3, -0.7) -- (1.3, -1.2);
  
  \draw[-, line width=0.5mm] (4.8, -0.7) -- (4.8, -1.2);
  \draw[-, line width=0.5mm] (4.2, -1.3) -- (4.7, -1.3);
  \draw[-, line width=0.5mm] (3.5, -1.3) -- (4, -1.3);
  \draw[-, line width=0.5mm] (3.4, -1.4) -- (3.4, -1.9);
  \draw[-, line width=0.5mm] (2.8, -2) -- (3.3, -2);
  \draw[-, line width=0.5mm] (2.1, -2) -- (2.6, -2);
\end{tikzpicture}
\quad \quad
\begin{tikzpicture}
  \node at (-0.8,0) {${}_{\circ}$}; 
  \node at (-0.1,0) {${}_{\circ}$}; 
  \node at (0.6,0) {${}_{\circ}$}; 
  \node at (1.3,0) {${}_{\circ}$}; 
  \node at (2,0) {${}_{\circ}$}; 
  \node at (2.7,0) {${}_{\circ}$}; 
  \node at (3.4,0) {${}_{\circ}$}; 
  \node at (4.1,0) {${}_{\circ}$}; 
  \node at (4.8,0) {${}_{\circ}$}; 

  \node at (-0.1,-0.6) {${}_{\circ}$}; 
  \node at (0.6,-0.6) {${}_{\circ}$}; 
  \node at (1.3, -0.6) {${}_{\circ}$}; 
  \node at (2, -0.6) {${}_{\circ}$}; 
  \node at (2.7, -0.6) {${}_{\circ}$}; 
  \node at (3.4,-0.6) {${}_{\circ}$}; 
  \node at (4.1,-0.6) {${}_{\circ}$}; 
  \node at (4.8,-0.6) {${}_{\circ}$}; 

  \node at (0.6,-1.3) {${}_{\circ}$}; 
  \node at (1.3, -1.3) {${}_{\circ}$}; 
  \node at (2, -1.3) {${}_{\circ}$}; 
  \node at (2.7, -1.3) {${}_{\circ}$}; 
  \node at (3.4, -1.3) {${}_{\circ}$}; 
  \node at (4.1,-1.3) {${}_{\circ}$}; 
  \node at (4.8,-1.3) {${}_{\circ}$}; 

  \node at (1.3, -2) {${}_{\circ}$}; 
  \node at (2, -2) {${}_{\circ}$}; 
  \node at (2.7, -2) {${}_{\circ}$}; 
  \node at (3.4, -2) {${}_{\circ}$}; 
  \node at (4.1, -2) {${}_{\circ}$}; 
  \node at (4.8, -2) {${}_{\circ}$}; 

  \node at (2, -2.7) {${}_{\circ}$}; 
  \node at (2.7, -2.7) {${}_{\circ}$}; 
  \node at (3.4, -2.7) {${}_{\circ}$}; 
  \node at (4.1, -2.7) {${}_{\circ}$}; 
  \node at (4.8, -2.7) {${}_{\circ}$}; 
  
  \node at (2.7, -3.4) {${}_{\circ}$}; 
  \node at (3.4, -3.4) {${}_{\circ}$}; 
  \node at (4.1, -3.4) {${}_{\circ}$}; 
  \node at (4.8, -3.4) {${}_{\circ}$}; 
  
  \node at (3.4, -4.1) {${}_{\circ}$}; 
  \node at (4.1, -4.1) {${}_{\circ}$}; 
  \node at (4.8, -4.1) {${}_{\circ}$}; 
  
  \node at (4.1, -4.8) {${}_{\circ}$}; 
  \node at (4.8, -4.8) {${}_{\circ}$}; 
  
  \node at (4.8, -5.5) {${}_{\circ}$}; 

  \draw[-, line width=0.5mm] (4.2, 0) -- (4.7, 0);
  \draw[-, line width=0.5mm] (3.5, 0) -- (4, 0);
  \draw[-, line width=0.5mm] (2.8, 0) -- (3.3, 0);
  \draw[-, line width=0.5mm] (2.7, -0.1) -- (2.7, -0.5);
  \draw[-, line width=0.5mm] (2.1, -0.6) -- (2.6, -0.6);
  \draw[-, line width=0.5mm] (1.4, -0.6) -- (1.9, -0.6);
  \draw[-, line width=0.5mm] (0.7, -0.6) -- (1.2, -0.6);
  \draw[-, line width=0.5mm] (0, -0.6) -- (0.5, -0.6);
  
  \draw[-, line width=0.5mm] (4.8, -0.1) -- (4.8, -0.5);
  \draw[-, line width=0.5mm] (4.2, -0.6) -- (4.7, -0.6);
  \draw[-, line width=0.5mm] (3.5, -0.6) -- (4, -0.6);
  \draw[-, line width=0.5mm] (3.4, -0.7) -- (3.4, -1.2);
  \draw[-, line width=0.5mm] (2.8, -1.3) -- (3.3, -1.3);
  \draw[-, line width=0.5mm] (2.1, -1.3) -- (2.6, -1.3);
  \draw[-, line width=0.5mm] (1.4, -1.3) -- (1.9, -1.3);
  \draw[-, line width=0.5mm] (0.7, -1.3) -- (1.2, -1.3);
  
  \draw[-, line width=0.5mm] (4.8, -0.7) -- (4.8, -1.2);
  \draw[-, line width=0.5mm] (4.2, -1.3) -- (4.7, -1.3);
  \draw[-, line width=0.5mm] (4.1, -1.4) -- (4.1, -1.9);
  \draw[-, line width=0.5mm] (3.5, -2) -- (4, -2);
  \draw[-, line width=0.5mm] (2.8, -2) -- (3.3, -2);
  \draw[-, line width=0.5mm] (2.1, -2) -- (2.6, -2);
  \draw[-, line width=0.5mm] (1.4, -2) -- (1.9, -2);
  
   \draw[-, line width=0.5mm] (4.8, -1.4) -- (4.8, -1.9);
   \draw[-, line width=0.5mm] (4.8, -2.1) -- (4.8, -2.6);
   \draw[-, line width=0.5mm] (4.2, -2.7) -- (4.7, -2.7);
   \draw[-, line width=0.5mm] (3.5, -2.7) -- (4, -2.7);
   \draw[-, line width=0.5mm] (2.8, -2.7) -- (3.3, -2.7);
   \draw[-, line width=0.5mm] (2.1, -2.7) -- (2.6, -2.7);
\end{tikzpicture}
\quad \quad
\begin{tikzpicture}
  \node at (-0.8,0) {${}_{\circ}$}; 
  \node at (-0.1,0) {${}_{\circ}$}; 
  \node at (0.6,0) {${}_{\circ}$}; 
  \node at (1.3,0) {${}_{\circ}$}; 
  \node at (2,0) {${}_{\circ}$}; 
  \node at (2.7,0) {${}_{\bullet}$}; 
  \node at (3.4,0) {${}_{\circ}$}; 
  \node at (4.1,0) {${}_{\circ}$}; 
  \node at (4.8,0) {${}_{\circ}$}; 

  \node at (-0.1,-0.6) {${}_{\circ}$}; 
  \node at (0.6,-0.6) {${}_{\circ}$}; 
  \node at (1.3, -0.6) {${}_{\bullet}$}; 
  \node at (2, -0.6) {${}_{\circ}$}; 
  \node at (2.7, -0.6) {${}_{\circ}$}; 
  \node at (3.4,-0.6) {${}_{\circ}$}; 
  \node at (4.1,-0.6) {${}_{\circ}$}; 
  \node at (4.8,-0.6) {${}_{\circ}$}; 

  \node at (0.6,-1.3) {${}_{\circ}$}; 
  \node at (1.3, -1.3) {${}_{\circ}$}; 
  \node at (2, -1.3) {${}_{\circ}$}; 
  \node at (2.7, -1.3) {${}_{\circ}$}; 
  \node at (3.4, -1.3) {${}_{\bullet}$}; 
  \node at (4.1,-1.3) {${}_{\circ}$}; 
  \node at (4.8,-1.3) {${}_{\circ}$}; 

  \node at (1.3, -2) {${}_{\circ}$}; 
  \node at (2, -2) {${}_{\circ}$}; 
  \node at (2.7, -2) {${}_{\circ}$}; 
  \node at (3.4, -2) {${}_{\circ}$}; 
  \node at (4.1, -2) {${}_{\circ}$}; 
  \node at (4.8, -2) {${}_{\bullet}$}; 

  \node at (2, -2.7) {${}_{\bullet}$}; 
  \node at (2.7, -2.7) {${}_{\circ}$}; 
  \node at (3.4, -2.7) {${}_{\circ}$}; 
  \node at (4.1, -2.7) {${}_{\circ}$}; 
  \node at (4.8, -2.7) {${}_{\circ}$}; 
  
  \node at (2.7, -3.4) {${}_{\circ}$}; 
  \node at (3.4, -3.4) {${}_{\circ}$}; 
  \node at (4.1, -3.4) {${}_{\bullet}$}; 
  \node at (4.8, -3.4) {${}_{\circ}$}; 
  
  \node at (3.4, -4.1) {${}_{\circ}$}; 
  \node at (4.1, -4.1) {${}_{\circ}$}; 
  \node at (4.8, -4.1) {${}_{\circ}$}; 
  
  \node at (4.1, -4.8) {${}_{\circ}$}; 
  \node at (4.8, -4.8) {${}_{\circ}$}; 
  
  \node at (4.8, -5.5) {${}_{\circ}$}; 
  
  \draw[-, line width=0.5mm] (4.2, 0) -- (4.7, 0);
  \draw[-, line width=0.5mm] (3.5, 0) -- (4, 0);
  \draw[-, line width=0.5mm] (2.8, 0) -- (3.3, 0);
  \draw[-, line width=0.5mm] (2.7, -0.1) -- (2.7, -0.5);
  \draw[-, line width=0.5mm] (2.1, -0.6) -- (2.6, -0.6);
  \draw[-, line width=0.5mm] (1.4, -0.6) -- (1.9, -0.6);
  \draw[-, line width=0.5mm] (1.3, -0.7) -- (1.3, -1.2);
  \draw[-, line width=0.5mm] (1.3, -1.4) -- (1.3, -1.9);
  
  \draw[-, line width=0.5mm] (4.1, -0.1) -- (4.1, -0.5);
  \draw[-, line width=0.5mm] (4.1, -0.7) -- (4.1, -1.2);
  \draw[-, line width=0.5mm] (4.1, -1.4) -- (4.1, -1.9);
  \draw[-, line width=0.5mm] (4.1, -2.1) -- (4.1, -2.6);
  \draw[-, line width=0.5mm] (4.1, -2.8) -- (4.1, -3.3);
  \draw[-, line width=0.5mm] (3.5, -3.4) -- (4, -3.4);
  \draw[-, line width=0.5mm] (3.4, -3.5) -- (3.4, -4);
  
  \draw[-, line width=0.5mm] (4.8, -0.1) -- (4.8, -0.5);
  \draw[-, line width=0.5mm] (4.8, -0.7) -- (4.8, -1.2);
  \draw[-, line width=0.5mm] (4.8, -1.4) -- (4.8, -1.9);
  \draw[-, line width=0.5mm] (4.8, -2.1) -- (4.8, -2.6);
  \draw[-, line width=0.5mm] (4.8, -2.8) -- (4.8, -3.3);
  \draw[-, line width=0.5mm] (4.8, -3.5) -- (4.8, -4);
  \draw[-, line width=0.5mm] (4.2, -4.1) -- (4.7, -4.1);
  \draw[-, line width=0.5mm] (4.1, -4.2) -- (4.1, -4.7);
  
   \draw[-, line width=0.5mm] (2.1, -1.3) -- (2.6, -1.3);
   \draw[-, line width=0.5mm] (2, -1.4) -- (2, -1.9);
   \draw[-, line width=0.5mm] (2, -2.1) -- (2, -2.6);
   \draw[-, line width=0.5mm] (2.7, -1.4) -- (2.7, -1.9);
   \draw[-, line width=0.5mm] (2.7, -2.1) -- (2.7, -2.6);
   \draw[-, line width=0.5mm] (2.7, -2.8) -- (2.7, -3.3);
\end{tikzpicture}
}
\end{equation*}
where the $\bullet$'s in the rightmost path are the dots located at $(4, 1), (1, 4), (2, 6), (6, 2), (5, 5)$ and $(3,3)$, respectively. 
Note that by the condition (iv) in Definition \ref{df:paths in Pn}(2), the triple path in a pair of double and triple paths in $\mathcal{P}_7$, does not pass through the dots located at $(1, 9)$ and $(9, 1)$. 
}
\end{ex}

Put $\Delta = \Delta_{n+3}$.
For ${\bf c} \in \B^{\J}$ and ${\bf p} \in \mathcal{P}_n$ $(n = 6, 7)$, we define
\begin{equation*}
	||{\bf c}||_{\bf p} = 
	\displaystyle \sum_{\text{$\beta$ lying on ${\bf p} \cap \Delta^{\J}$}} c_{\beta},
\end{equation*}
where ${\bf p} \cap \Delta^{\J}$ is the set of the positive roots in $\Delta^{\J}$ lying on the path ${\bf p}$ under the identification of $\Delta^{\J}$ in $\Delta$.

\begin{thm} \label{thm:formula of epsilon star}
{\em
For ${\bf c} \in \B^{\J}$, we have
\begin{equation*}
	\varepsilon_r^* ({\bf c}) 
	= 
	\max \left\{ \, ||{\bf c}||_{\bf p} \,\, | \,\, {\bf p} \in \mathcal{P}_n \, \right\},
\end{equation*}
where ${\mf g}$ is of type $\text{E}_n$ with $n = 6, 7$ and $r$ is given as in Figure \ref{fig:dynkins}.
}
\end{thm}
\begin{proof}
  We will give the proof in Section \ref{sec:proof of theorem}.
\end{proof}

\begin{rem}\label{rem:remark for analog for type E}
{\em Suppose that $\g$ is of type $\text{A}_n$ and $\mf l$ is of type $A_{r}\times A_{s}$ with $r+s=n-1$. The associated crystal $\B^\J$ can be realized as the set of $(r+1)\times (s+1)$ non-negative integral matrices (see \cite[Section 4.3]{Kw18}). 
For $M \in \B^\J$, let $\lambda=(\lambda_1,\lambda_2,\ldots)$ be the shape of the tableaux corresponding to $M$ under RSK correspondence (cf. \cite{Ful}).
It was a well-known result due to Greene \cite{G} that $\lambda_1$ is a maximal sum of entries in $M$ lying on a lattice path on $(r+1)\times (s+1)$ array of points from northeast to southwest, which coincides with $\varepsilon_r^*(M)$.

Similarly, for type $\text{D}_n$ with $r = n$ (in this case, ${\mf l}$ is of type $A_{n-1}$), the associated crystal $\B^\J$ can be realized as the set of the strictly upper triangular $n \times n$ non-negative integral matrices (see \cite[Section 3.2]{JK19}). For $M \in \B^\J$, let $\lambda=(\lambda_1,\lambda_2,\ldots)$ be the shape of the tableaux corresponding to $M$ under Burge correspondence (see \cite{Bur}, cf.~\cite[Section 4.2]{JK19}), which is an analog of RSK for type $\text{D}$. Then we proved in \cite{JK19} that $\lambda_1$ is a maximal sum of entries in $M$ lying on a non-intersecting double path defined on $(n-1) \times (n-1)$ array of points (see \cite[Definition 3.10]{JK19} for definition of double path), which also coincides with $\varepsilon_n^*(M)$.

In this sense, we expect that there exists an analog of RSK for types $\text{E}_6$ and $\text{E}_7$ to explain the relationship between our model and a tableau model. Then, the formula in Theorem \ref{thm:formula of epsilon star} may be viewed as a non-trivial type $\text{E}$ analog of the previous results for $\varepsilon_r^*$.
}
\end{rem}

\begin{rem} \label{rem:polytope}
{\em 
By Theorem \ref{thm:formula of epsilon star}, we have
\begin{equation*}
	\B^{\J, s} = \bigcap_{{\bf p}} \left\{ \, {\bf c} \in \B^{\J} \, \big| \, ||{\bf c}||_{\bf p} \le s \, \right\},
\end{equation*}
where ${\bf p}$ runs over the triple or quadruple paths in $\Delta$. Combining Theorem \ref{thm:main theorem}, this gives a polytope realization of the KR crystal $B^{r, s}$ of types $\text{E}_{6}^{(1)}$ and $\text{E}_{7}^{(1)}$ with the minuscule nodes $r$ in Figure \ref{fig:dynkins}, which may be viewed as the type $\text{E}$ analog of \cite{JK19}.
}
\end{rem}

\subsection{Proof of Theorem \ref{thm:formula of epsilon star}} \label{sec:proof of theorem}
We prove Theorem \ref{thm:formula of epsilon star} following \cite[Section 5.2]{JK19}.
We refer the reader to \cite[Section 5.1]{JK19} for the brief review of the formula of the transition matrix between Lusztig's parametrization and string parametrization of $B(\infty)$ \cite{BZ}, which is crucial to prove Theorem \ref{thm:formula of epsilon star}.
We keep the notations in \cite[Section 5.1]{JK19}.
\vskip 2mm

\subsubsection{A formula of $\varepsilon_r^*$}
Let ${\bf i}_0$ be the reduced expression given in \eqref{eq:i_0} with ${\bf i}^{\J} = (i_1, \dots, i_M)$ and ${\bf i}_{\J} = (i_{M+1}, \dots, i_N)$, where $M$ and $N$ are given by
\begin{equation*}
	(M, N) = 
	\begin{cases}
		(16, 36) & \text{if ${\mf g}$ is of type $\text{E}_6$,} \\
		(27, 63) & \text{if ${\mf g}$ is of type $\text{E}_7$.}
	\end{cases}
\end{equation*}
If ${\mf g}$ is of type $\text{E}_6$, then $
	1^* = 6,\, 2^* = 2,\, 3^* = 5,\, 4^* = 4,\, 5^* = 3,\, 6^* = 1$.
If ${\mf g}$ is of type $\text{E}_7$, we have $i^* = i$ for all $i \in I$.

Put
\begin{equation} \label{eq:j_0}
	{\bf j}_0 = (j_1, \dots, j_N) := {\bf i}_0^{*{\rm op}} = (i_N^*, \dots, i_1^*).
\end{equation}

\begin{rem} \label{rem:identification}
{\em 
Since the representation $V(\varpi_r)$ of $U_q({\mf g})$ is minuscule, we can identify a ${\bf j}_0$-trail $\pi = (\nu_0, \dots, \nu_N)$ in $V(\varpi_n)$ with a sequence $(b_0, \dots, b_N) \in B(\varpi_r)^{\times (N+1)}$ such that ${\rm wt}(b_k) = \nu_k$ and $\tf_{j_k}^{d_k(\pi)} b_{k-1} = b_k$ with $d_k(\pi) = 0$ or $1$ for $1 \le k \le N$. 
Throughout this section, we frequently use this identification.
}
\end{rem}

\begin{lem} \label{lem:trail from wr to w0wr}
{\em 
\mbox{} \

\begin{enumerate}
	\item There exists a unique ${\bf j}_0$-trail $\pi$ from $\varpi_r$ to $w_0\varpi_r$ such that $d_k(\pi) = 0$ for $1 \le k \le N-M$, and $d_k(\pi) = 1$ otherwise.
	\vskip 1mm
	
	\item There exists a unique $(j_1, \dots, j_{N-M})$-trail from $\varpi_r - \alpha_r$ to $\varpi_r - \theta$. 
We denote it by $(\tilde{\nu}_0, \dots, \tilde{\nu}_{N-M})$
\end{enumerate}
}
\end{lem}
\begin{proof}
The proof is straightforward by considering the structure of the crystal $B(\varpi_r)$ (cf. \cite[Figures 2 and 3]{JS}) and ${\bf j}_0$ \eqref{eq:j_0} under the identification in Remark \ref{rem:identification}.
\end{proof}

Let $\mathcal{T}$ be the set of ${\bf j}_0$-trails from $s_r \varpi_r$ to $w_0 \varpi_r$ in $V(\varpi_r)$.
The following formula plays an important role in the proof of Theorem \ref{thm:formula of epsilon star}.

\begin{lem} \label{lem:formula of epsilon star-1}
{\em 
For ${\bf c} = (c_{\beta_k}) \in \B^{\J}$, we have
\begin{equation*}
	\varepsilon_{r}^* ({\bf c}) = \max \{\, ||{\bf c}||_{\pi} \, | \, \pi \in \mathcal{T} \,\},
\end{equation*}
where $||{\bf c}||_{\pi}$ is defined by 
\begin{equation} \label{eq:statistic of c by pi}
	||{\bf c}||_{\pi}
	=
	\sum_{k=1}^M (1-d_{N-k+1}(\pi))c_{\beta_k}.
\end{equation}
}
\end{lem}
\begin{proof}
By using Lemma \ref{lem:trail from wr to w0wr}(1), the proof follows from the same argument in the one of \cite[Lemma 5.4]{JK19}, where the formula of $\varepsilon_r^*({\bf c})$ is obtained from the one of Berenstein-Zelevinsky \cite{BZ} with respect to the reduced expressions ${\bf i}_0$ \eqref{eq:i_0} and ${\bf j}_0$ \eqref{eq:j_0}.
\end{proof}

Let us characterize the ${\bf j}_0$-trails $\pi \in \mathcal{T}$ such that $||{\bf c}||_{\pi}$ may be maximal for ${\bf c} \in \B^{\J}$.
For $\pi = (\nu_0, \dots, \nu_N) \in \mathcal{T}$, put
$\pi_{\J} = (\nu_0, \dots, \nu_{N-M})$ and $\pi^{\J} = (\nu_{N-M+1}, \dots, \nu_N)$.
We define 
\begin{equation*}
	\mathcal{T}^{\,'}
	=
	\{\, \pi \in \mathcal{T} \, | \, \pi_{\J} = (\tilde{\nu}_0, \dots, \tilde{\nu}_{N-M}) \,\},
\end{equation*}
where $(\tilde{\nu}_0, \dots, \tilde{\nu}_{N-M})$ is the trail as in Lemma \ref{lem:trail from wr to w0wr}(2).

\begin{lem} \label{lem:reduction}
{\em 
For ${\bf c} = (c_{\beta_k}) \in \B^{\J}$, we have
\begin{equation*}
	\varepsilon_{r}^* ({\bf c}) = \max \{\, ||{\bf c}||_{\pi} \, | \, \pi \in \mathcal{T}^{\,'} \,\}.
\end{equation*}
}
\end{lem}
\begin{proof}
Let $\pi = (\nu_0, \dots, \nu_N) \in \mathcal{T} \, \setminus \, \mathcal{T}^{\,'}$ be given. It is enough to show that there exists $\pi' \in \mathcal{T}^{\,'}$ such that $||{\bf c}||_{\pi} \le ||{\bf c}||_{\pi'}$. 
To do this, let us consider the maximal subsequence $(j_{q_1}, \dots, j_{q_{N'}})$ of $(j_{N-M+1}, \dots, j_{N})$ such that $k_{q_k}(\pi) = 1$ for $1 \le k \le N'$.
Since $\pi \notin \mathcal{T}^{\,'}$, by considering the structure of the crystal $B(\varpi_r)$ (cf. \cite[Figures 2 and 3]{JS}), the sequence $(j_{q_1}, \dots, j_{q_{N'}})$ contains a subsequence $(j_{r_1}, \dots, j_{r_{N''}})$ given by
\begin{equation} \label{eq:minimal seq}
	\begin{cases}
		(1, 3, 4, 5, 6) & \text{for type $\text{E}_6$ with $r = 1$,} \\
		(6, 5, 4, 3, 1) & \text{for type $\text{E}_6$ with $r = 6$,} \\
		(7, 6, 5, 4, 3, 2, 4, 5, 6, 7) & \text{for type $\text{E}_7$ with $r = 7$,}
	\end{cases}
\end{equation}
(up to $2$-term braid moves on $2$ and $3$ in type $\text{E}_7$). 
Then we define a ${\bf j}_0$-trail $\pi'$ such that $\pi_{\J}' = (\tilde{\nu}_0, \dots, \tilde{\nu}_{N-M})$ and 
\begin{equation*}
	d_k(\pi') = 
	\begin{cases}
		1 & \text{if $k = r_l$ for some $1 \le l \le N''$,} \\
		0 & \text{otherwise.}
	\end{cases}
\end{equation*}
Then $\pi' \in \mathcal{T}^{\,'}$ and we have $||{\bf c}||_{\pi} \le ||{\bf c}||_{\pi'}$ by \eqref{eq:statistic of c by pi} and the construction of $\pi'$.
\end{proof}

\subsubsection{The ${\bf j}_0$-trails in $\mathcal{T}^{\,'}$}\label{subsec:j0-trails}
This subsection will be devoted to give the complete list of the ${\bf j}_0$-trails in $\mathcal{T}^{\,'}$.
For $\pi \in \mathcal{T}^{\,'}$, since $\pi_{\J}$ is fixed by $(\tilde{\nu}_0, \dots, \tilde{\nu}_{N-M})$, we identify ${\bf \pi}$ with the sequence $(d_k(\pi))_{N-M+1 \le k \le N}$, where $d_k(\pi)$ is $0$ or $1$ (see Remark \ref{rem:identification}).
For convenience, we also identify the sequence $(d_k(\pi))_{N-M+1 \le k \le N}$ with a sequence given as in the following example.
\begin{ex} \label{ex:notation for trail}
{\em 
For type $\text{E}_6$ with $r = 6$, let $\pi \in \mathcal{T}^{\,'}$ be given. In this case, the subsequence $({\bf i}^{\J*})^{\rm op}$ of ${\bf j}_0$ is given by 
\begin{equation*}
	(6,\, 5,\, 4,\, 3,\, 1,\, 2,\, 4,\, 3,\, 5,\, 4,\, 2,\, 6,\, 5,\, 4,\, 3,\, 1)
\end{equation*}
Then we rewrite the $k$-th number from left to right such that $d_k(\pi) = 1$ (resp. $d_k(\pi) = 0$) as the bold (resp. gray) one.
If $(d_k(\pi))_{21 \le k \le 36}$ for $\pi$ is given by $(1, 0, 0, 0, 0, 0, 0, 0, 1, 1, 0, 0, 0, 0, 1, 1)$, then we identify it with
\begin{equation*}
({\bf 6},\, {\color{gray} 5},\, {\color{gray} 4},\, {\color{gray} 3},\, {\color{gray} 1},\, {\color{gray} 2},\, {\color{gray} 4},\, {\color{gray} 3},\, {\bf 5},\, {\bf 4},\, {\color{gray} 2},\, {\color{gray} 6},\, {\color{gray} 5},\, {\color{gray} 4},\, {\bf 3},\, {\bf 1}).
\end{equation*}
}
\end{ex}

\begin{rem} \label{rem:connectedness}
{\em 
Let $\pi_0$ be the ${\bf j}_0$-trail in $\mathcal{T}^{\,'}$ corresponding to
\begin{equation*}
	\begin{cases}
		({\bf 6},\, {\bf 5},\, {\bf 4},\, {\bf 3},\, {\bf 1},\, {\color{gray} 2},\, {\color{gray} 4},\, {\color{gray} 3},\, {\color{gray} 5},\, {\color{gray} 4},\, {\color{gray} 2},\, {\color{gray} 6},\, {\color{gray} 5},\, {\color{gray} 4},\, {\color{gray} 3},\, {\color{gray} 1}) & \text{for type $\text{E}_6$,} \\
		({\bf 7}, {\bf 6}, {\bf 5}, {\bf 4}, {\bf 2}, {\bf 3}, {\bf 4}, {\bf 5}, {\bf 6}, {\bf 7}, {\color{gray}  1}, {\color{gray} 3}, {\color{gray} 4}, {\color{gray} 5}, {\color{gray} 6}, {\color{gray} 2}, {\color{gray} 4}, {\color{gray} 5}, {\color{gray} 3}, {\color{gray} 4}, {\color{gray} 2}, {\color{gray} 1}, {\color{gray} 3}, {\color{gray} 4}, {\color{gray} 5}, {\color{gray} 6}, {\color{gray} 7}) & \text{for type $\text{E}_7$,}
	\end{cases}
\end{equation*}
and let $\pi_0'$ be the ${\bf j}_0$-trail in $\mathcal{T}^{\,'}$ corresponding to 
\begin{equation*}
	\begin{cases}
		({\color{gray} 6},\, {\color{gray} 5},\, {\color{gray} 4},\, {\color{gray} 3},\, {\color{gray} 1},\, {\color{gray} 2},\, {\color{gray} 4},\, {\color{gray} 3},\, {\color{gray} 5},\, {\color{gray} 4},\, {\color{gray} 2},\, {\bf 6},\, {\bf 5},\, {\bf 4},\, {\bf 3},\, {\bf 1}) & \text{for type $\text{E}_6$,} \\
		({\color{gray}  7}, {\color{gray}  6}, {\color{gray}  5}, {\color{gray}  4}, {\color{gray}  2}, {\color{gray}  3}, {\color{gray}  4}, {\color{gray}   5}, {\color{gray}   6}, {\bf  7}, {\color{gray}  1}, {\color{gray}  3}, {\color{gray}   4}, {\color{gray}   5}, {\bf 6}, {\color{gray}  2}, {\color{gray} 4}, {\bf 5}, {\color{gray}  3}, {\bf 4}, {\bf 2}, {\color{gray} 1}, {\bf 3}, {\bf 4}, {\bf 5}, {\bf 6}, {\bf 7}) & \text{for type $\text{E}_7$.}
	\end{cases}
\end{equation*}
Then any ${\bf j}_0$-trail $\pi$ in $\mathcal{T}^{\,'}$ is obtained from $\pi_0$ (resp. $\pi_0'$) by moving the bold numbers to the right (resp. left) successively without a change of the enumeration on them (up to $2$-term braid moves on $2$ and $3$ for type $\text{E}_7$).
}
\end{rem}

Now, we state the complete list of the ${\bf j}_0$-trails in $\mathcal{T}^{\,'}$ as follows.
\vskip 1mm

{\it Case 1.} Type $\text{E}_6$. We consider the case of $r = 6$ only, since the case of $r =1$ is almost identical.
{\allowdisplaybreaks \footnotesize
\begin{gather*}
	({\bf 6},\, {\bf 5},\, {\bf 4},\, {\bf 3},\, {\bf 1},\, {\color{gray} 2},\, {\color{gray} 4},\, {\color{gray} 3},\, {\color{gray} 5},\, {\color{gray} 4},\, {\color{gray} 2},\, {\color{gray} 6},\, {\color{gray} 5},\, {\color{gray} 4},\, {\color{gray} 3},\, {\color{gray} 1}), \\
	({\bf 6},\, {\bf 5},\, {\bf 4},\, {\bf 3},\, {\color{gray} 1},\, {\color{gray} 2},\, {\color{gray} 4},\, {\color{gray} 3},\, {\color{gray} 5},\, {\color{gray} 4},\, {\color{gray} 2},\, {\color{gray} 6},\, {\color{gray} 5},\, {\color{gray} 4},\, {\color{gray} 3},\, {\bf 1}), \\
	({\bf 6},\, {\bf 5},\, {\bf 4},\, {\color{gray} 3},\, {\color{gray} 1},\, {\color{gray} 2},\, {\color{gray} 4},\, {\bf 3},\, {\color{gray} 5},\, {\color{gray} 4},\, {\color{gray} 2},\, {\color{gray} 6},\, {\color{gray} 5},\, {\color{gray} 4},\, {\color{gray} 3},\, {\bf 1}), \\
	({\bf 6},\, {\bf 5},\, {\color{gray} 4},\, {\color{gray} 3},\, {\color{gray} 1},\, {\color{gray} 2},\, {\bf 4},\, {\bf 3},\, {\color{gray} 5},\, {\color{gray} 4},\, {\color{gray} 2},\, {\color{gray} 6},\, {\color{gray} 5},\, {\color{gray} 4},\, {\color{gray} 3},\, {\bf 1}), \\
	({\bf 6},\, {\bf 5},\, {\bf 4},\, {\color{gray} 3},\, {\color{gray} 1},\, {\color{gray} 2},\, {\color{gray} 4},\, {\color{gray} 3},\, {\color{gray} 5},\, {\color{gray} 4},\, {\color{gray} 2},\, {\color{gray} 6},\, {\color{gray} 5},\, {\color{gray} 4},\, {\bf 3},\, {\bf 1}), \\
	({\bf 6},\, {\bf 5},\, {\color{gray} 4},\, {\color{gray} 3},\, {\color{gray} 1},\, {\color{gray} 2},\, {\bf 4},\, {\color{gray} 3},\, {\color{gray} 5},\, {\color{gray} 4},\, {\color{gray} 2},\, {\color{gray} 6},\, {\color{gray} 5},\, {\color{gray} 4},\, {\bf 3},\, {\bf 1}), \\
	({\bf 6},\, {\bf 5},\, {\color{gray} 4},\, {\color{gray} 3},\, {\color{gray} 1},\, {\color{gray} 2},\, {\color{gray} 4},\, {\color{gray} 3},\, {\color{gray} 5},\, {\bf 4},\, {\color{gray} 2},\, {\color{gray} 6},\, {\color{gray} 5},\, {\color{gray} 4},\, {\bf 3},\, {\bf 1}), \\
	({\bf 6},\, {\color{gray} 5},\, {\color{gray} 4},\, {\color{gray} 3},\, {\color{gray} 1},\, {\color{gray} 2},\, {\color{gray} 4},\, {\color{gray} 3},\, {\bf 5},\, {\bf 4},\, {\color{gray} 2},\, {\color{gray} 6},\, {\color{gray} 5},\, {\color{gray} 4},\, {\bf 3},\, {\bf 1}), \\
	({\bf 6},\, {\bf 5},\, {\color{gray} 4},\, {\color{gray} 3},\, {\color{gray} 1},\, {\color{gray} 2},\, {\color{gray} 4},\, {\color{gray} 3},\, {\color{gray} 5},\, {\color{gray} 4},\, {\color{gray} 2},\, {\color{gray} 6},\, {\color{gray} 5},\, {\bf 4},\, {\bf 3},\, {\bf 1}), \\
		({\bf 6},\, {\color{gray} 5},\, {\color{gray} 4},\, {\color{gray} 3},\, {\color{gray} 1},\, {\color{gray} 2},\, {\color{gray} 4},\, {\color{gray} 3},\, {\bf 5},\, {\color{gray} 4},\, {\color{gray} 2},\, {\color{gray} 6},\, {\color{gray} 5},\, {\bf 4},\, {\bf 3},\, {\bf 1}), \\
	({\bf 6},\, {\color{gray} 5},\, {\color{gray} 4},\, {\color{gray} 3},\, {\color{gray} 1},\, {\color{gray} 2},\, {\color{gray} 4},\, {\color{gray} 3},\, {\color{gray} 5},\, {\color{gray} 4},\, {\color{gray} 2},\, {\color{gray} 6},\, {\bf 5},\, {\bf 4},\, {\bf 3},\, {\bf 1}), \\
	({\color{gray} 6},\, {\color{gray} 5},\, {\color{gray} 4},\, {\color{gray} 3},\, {\color{gray} 1},\, {\color{gray} 2},\, {\color{gray} 4},\, {\color{gray} 3},\, {\color{gray} 5},\, {\color{gray} 4},\, {\color{gray} 2},\, {\bf 6},\, {\bf 5},\, {\bf 4},\, {\bf 3},\, {\bf 1}).
\end{gather*}}
\vskip 2mm

{\it Case 2.} Type $\text{E}_7$. In this case, since the subsequence \eqref{eq:minimal seq} involves both $2$ and $3$ which commutes, we separate two cases along the position of them.
\vskip 2mm

{\it Sub-Case 1}. (7, 6, 5, 4, {\bf 2}, {\bf 3}, 4, 5, 6, 7).
{\allowdisplaybreaks \footnotesize 
\begin{gather*}
	({\bf 7}, {\bf 6}, {\bf 5}, {\bf 4}, {\bf 2}, {\bf 3}, {\bf 4}, {\bf 5}, {\bf 6}, {\bf 7}, {\color{gray}  1}, {\color{gray} 3}, {\color{gray} 4}, {\color{gray} 5}, {\color{gray} 6}, {\color{gray} 2}, {\color{gray} 4}, {\color{gray} 5}, {\color{gray} 3}, {\color{gray} 4}, {\color{gray} 2}, {\color{gray} 1}, {\color{gray} 3}, {\color{gray} 4}, {\color{gray} 5}, {\color{gray} 6}, {\color{gray} 7}), \\
	({\bf 7}, {\bf 6}, {\bf 5}, {\bf 4}, {\bf 2}, {\bf 3}, {\bf 4}, {\bf 5}, {\bf 6}, {\color{gray}  7}, {\color{gray}  1}, {\color{gray} 3}, {\color{gray} 4}, {\color{gray} 5}, {\color{gray} 6}, {\color{gray} 2}, {\color{gray} 4}, {\color{gray} 5}, {\color{gray} 3}, {\color{gray} 4}, {\color{gray} 2}, {\color{gray} 1}, {\color{gray} 3}, {\color{gray} 4}, {\color{gray} 5}, {\color{gray} 6}, {\bf 7}), \\
	({\bf 7}, {\bf 6}, {\bf 5}, {\bf 4}, {\bf 2}, {\bf 3}, {\bf 4}, {\bf 5}, {\color{gray}  6}, {\color{gray}  7}, {\color{gray}  1}, {\color{gray} 3}, {\color{gray} 4}, {\color{gray} 5}, {\bf 6}, {\color{gray} 2}, {\color{gray} 4}, {\color{gray} 5}, {\color{gray} 3}, {\color{gray} 4}, {\color{gray} 2}, {\color{gray} 1}, {\color{gray} 3}, {\color{gray} 4}, {\color{gray} 5}, {\color{gray} 6}, {\bf 7}), \\
	({\bf 7}, {\bf 6}, {\bf 5}, {\bf 4}, {\bf 2}, {\bf 3}, {\bf 4}, {\color{gray} 5}, {\color{gray}  6}, {\color{gray}  7}, {\color{gray}  1}, {\color{gray} 3}, {\color{gray} 4}, {\bf 5}, {\bf 6}, {\color{gray} 2}, {\color{gray} 4}, {\color{gray} 5}, {\color{gray} 3}, {\color{gray} 4}, {\color{gray} 2}, {\color{gray} 1}, {\color{gray} 3}, {\color{gray} 4}, {\color{gray} 5}, {\color{gray} 6}, {\bf 7}), \\
	({\bf 7}, {\bf 6}, {\bf 5}, {\bf 4}, {\bf 2}, {\bf 3}, {\color{gray}  4}, {\color{gray}  5}, {\color{gray}  6}, {\color{gray}  7}, {\color{gray}  1}, {\color{gray} 3}, {\bf 4}, {\bf 5}, {\bf 6}, {\color{gray} 2}, {\color{gray} 4}, {\color{gray} 5}, {\color{gray} 3}, {\color{gray} 4}, {\color{gray} 2}, {\color{gray} 1}, {\color{gray} 3}, {\color{gray} 4}, {\color{gray} 5}, {\color{gray} 6}, {\bf 7}), \\
	({\bf 7}, {\bf 6}, {\bf 5}, {\bf 4}, {\bf 2}, {\color{gray}  3}, {\color{gray}  4}, {\color{gray}  5}, {\color{gray}  6}, {\color{gray}  7}, {\color{gray}  1}, {\bf 3}, {\bf 4}, {\bf 5}, {\bf 6}, {\color{gray} 2}, {\color{gray} 4}, {\color{gray} 5}, {\color{gray} 3}, {\color{gray} 4}, {\color{gray} 2}, {\color{gray} 1}, {\color{gray} 3}, {\color{gray} 4}, {\color{gray} 5}, {\color{gray} 6}, {\bf 7}), \\
	({\bf 7}, {\bf 6}, {\bf 5}, {\bf 4}, {\bf 2}, {\bf 3}, {\bf 4}, {\bf 5}, {\color{gray}  6}, {\color{gray}  7}, {\color{gray}  1}, {\color{gray} 3}, {\color{gray} 4}, {\color{gray} 5}, {\color{gray}  6}, {\color{gray} 2}, {\color{gray} 4}, {\color{gray} 5}, {\color{gray} 3}, {\color{gray} 4}, {\color{gray} 2}, {\color{gray} 1}, {\color{gray} 3}, {\color{gray} 4}, {\color{gray} 5}, {\bf 6}, {\bf 7}), \\
	({\bf 7}, {\bf 6}, {\bf 5}, {\bf 4}, {\bf 2}, {\bf 3}, {\bf 4}, {\color{gray}  5}, {\color{gray}  6}, {\color{gray}  7}, {\color{gray}  1}, {\color{gray} 3}, {\color{gray} 4}, {\bf 5}, {\color{gray}  6}, {\color{gray} 2}, {\color{gray} 4}, {\color{gray} 5}, {\color{gray} 3}, {\color{gray} 4}, {\color{gray} 2}, {\color{gray} 1}, {\color{gray} 3}, {\color{gray} 4}, {\color{gray} 5}, {\bf 6}, {\bf 7}), \\
	({\bf 7}, {\bf 6}, {\bf 5}, {\bf 4}, {\bf 2}, {\bf 3}, {\color{gray}  4}, {\color{gray}  5}, {\color{gray}  6}, {\color{gray}  7}, {\color{gray}  1}, {\color{gray} 3}, {\bf 4}, {\bf 5}, {\color{gray}  6}, {\color{gray} 2}, {\color{gray} 4}, {\color{gray} 5}, {\color{gray} 3}, {\color{gray} 4}, {\color{gray} 2}, {\color{gray} 1}, {\color{gray} 3}, {\color{gray} 4}, {\color{gray} 5}, {\bf 6}, {\bf 7}), \\
	({\bf 7}, {\bf 6}, {\bf 5}, {\bf 4}, {\bf 2}, {\color{gray}  3}, {\color{gray}  4}, {\color{gray}  5}, {\color{gray}  6}, {\color{gray}  7}, {\color{gray}  1}, {\bf 3}, {\bf 4}, {\bf 5}, {\color{gray}  6}, {\color{gray} 2}, {\color{gray} 4}, {\color{gray} 5}, {\color{gray} 3}, {\color{gray} 4}, {\color{gray} 2}, {\color{gray} 1}, {\color{gray} 3}, {\color{gray} 4}, {\color{gray} 5}, {\bf 6}, {\bf 7}), \\
	({\bf 7}, {\bf 6}, {\bf 5}, {\bf 4}, {\bf 2}, {\bf 3}, {\bf 4}, {\color{gray}  5}, {\color{gray}  6}, {\color{gray}  7}, {\color{gray}  1}, {\color{gray} 3}, {\color{gray} 4}, {\color{gray}  5}, {\color{gray}  6}, {\color{gray} 2}, {\color{gray} 4}, {\bf 5}, {\color{gray} 3}, {\color{gray} 4}, {\color{gray} 2}, {\color{gray} 1}, {\color{gray} 3}, {\color{gray} 4}, {\color{gray} 5}, {\bf 6}, {\bf 7}), \\
	({\bf 7}, {\bf 6}, {\bf 5}, {\bf 4}, {\bf 2}, {\bf 3}, {\color{gray}  4}, {\color{gray}  5}, {\color{gray}  6}, {\color{gray}  7}, {\color{gray}  1}, {\color{gray} 3}, {\bf 4}, {\color{gray}  5}, {\color{gray}  6}, {\color{gray} 2}, {\color{gray} 4}, {\bf 5}, {\color{gray} 3}, {\color{gray} 4}, {\color{gray} 2}, {\color{gray} 1}, {\color{gray} 3}, {\color{gray} 4}, {\color{gray} 5}, {\bf 6}, {\bf 7}), \\
	({\bf 7}, {\bf 6}, {\bf 5}, {\bf 4}, {\bf 2}, {\color{gray} 3}, {\color{gray}  4}, {\color{gray}  5}, {\color{gray}  6}, {\color{gray}  7}, {\color{gray}  1}, {\bf 3}, {\bf 4}, {\color{gray}  5}, {\color{gray}  6}, {\color{gray} 2}, {\color{gray} 4}, {\bf 5}, {\color{gray} 3}, {\color{gray} 4}, {\color{gray} 2}, {\color{gray} 1}, {\color{gray} 3}, {\color{gray} 4}, {\color{gray} 5}, {\bf 6}, {\bf 7}), \\
	({\bf 7}, {\bf 6}, {\bf 5}, {\bf 4}, {\bf 2}, {\bf 3}, {\color{gray}  4}, {\color{gray}  5}, {\color{gray}  6}, {\color{gray}  7}, {\color{gray}  1}, {\color{gray} 3}, {\color{gray}  4}, {\color{gray}  5}, {\color{gray}  6}, {\color{gray} 2}, {\bf 4}, {\bf 5}, {\color{gray} 3}, {\color{gray} 4}, {\color{gray} 2}, {\color{gray} 1}, {\color{gray} 3}, {\color{gray} 4}, {\color{gray} 5}, {\bf 6}, {\bf 7}), \\
	({\bf 7}, {\bf 6}, {\bf 5}, {\bf 4}, {\bf 2}, {\color{gray}  3}, {\color{gray}  4}, {\color{gray}  5}, {\color{gray}  6}, {\color{gray}  7}, {\color{gray}  1}, {\bf 3}, {\color{gray}  4}, {\color{gray}  5}, {\color{gray}  6}, {\color{gray} 2}, {\bf 4}, {\bf 5}, {\color{gray} 3}, {\color{gray} 4}, {\color{gray} 2}, {\color{gray} 1}, {\color{gray} 3}, {\color{gray} 4}, {\color{gray} 5}, {\bf 6}, {\bf 7}), \\
	({\bf 7}, {\bf 6}, {\bf 5}, {\bf 4}, {\bf 2}, {\bf 3}, {\bf 4}, {\color{gray}  5}, {\color{gray}  6}, {\color{gray}  7}, {\color{gray}  1}, {\color{gray} 3}, {\color{gray} 4}, {\color{gray}  5}, {\color{gray}  6}, {\color{gray} 2}, {\color{gray} 4}, {\color{gray}  5}, {\color{gray} 3}, {\color{gray} 4}, {\color{gray} 2}, {\color{gray} 1}, {\color{gray} 3}, {\color{gray} 4}, {\bf 5}, {\bf 6}, {\bf 7}), \\
	({\bf 7}, {\bf 6}, {\bf 5}, {\bf 4}, {\bf 2}, {\bf 3}, {\color{gray}  4}, {\color{gray}  5}, {\color{gray}  6}, {\color{gray}  7}, {\color{gray}  1}, {\color{gray} 3}, {\bf 4}, {\color{gray}  5}, {\color{gray}  6}, {\color{gray} 2}, {\color{gray} 4}, {\color{gray}  5}, {\color{gray} 3}, {\color{gray} 4}, {\color{gray} 2}, {\color{gray} 1}, {\color{gray} 3}, {\color{gray} 4}, {\bf 5}, {\bf 6}, {\bf 7}), \\
	({\bf 7}, {\bf 6}, {\bf 5}, {\bf 4}, {\bf 2}, {\color{gray}  3}, {\color{gray}  4}, {\color{gray}  5}, {\color{gray}  6}, {\color{gray}  7}, {\color{gray}  1}, {\bf 3}, {\bf 4}, {\color{gray}  5}, {\color{gray}  6}, {\color{gray} 2}, {\color{gray} 4}, {\color{gray}  5}, {\color{gray} 3}, {\color{gray} 4}, {\color{gray} 2}, {\color{gray} 1}, {\color{gray} 3}, {\color{gray} 4}, {\bf 5}, {\bf 6}, {\bf 7}), \\
	({\bf 7}, {\bf 6}, {\bf 5}, {\bf 4}, {\bf 2}, {\bf 3}, {\color{gray}  4}, {\color{gray}  5}, {\color{gray}  6}, {\color{gray}  7}, {\color{gray}  1}, {\color{gray} 3}, {\color{gray}  4}, {\color{gray}  5}, {\color{gray}  6}, {\color{gray} 2}, {\bf 4}, {\color{gray}  5}, {\color{gray} 3}, {\color{gray} 4}, {\color{gray} 2}, {\color{gray} 1}, {\color{gray} 3}, {\color{gray} 4}, {\bf 5}, {\bf 6}, {\bf 7}), \\
	({\bf 7}, {\bf 6}, {\bf 5}, {\bf 4}, {\bf 2}, {\bf 3}, {\color{gray}  4}, {\color{gray}  5}, {\color{gray}  6}, {\color{gray}  7}, {\color{gray}  1}, {\color{gray} 3}, {\color{gray}  4}, {\color{gray}  5}, {\color{gray}  6}, {\color{gray} 2}, {\color{gray} 4}, {\color{gray}  5}, {\color{gray} 3}, {\bf 4}, {\color{gray} 2}, {\color{gray} 1}, {\color{gray} 3}, {\color{gray} 4}, {\bf 5}, {\bf 6}, {\bf 7}) \\
	({\bf 7}, {\bf 6}, {\bf 5}, {\bf 4}, {\bf 2}, {\color{gray} 3}, {\color{gray}  4}, {\color{gray}  5}, {\color{gray}  6}, {\color{gray}  7}, {\color{gray}  1}, {\bf 3}, {\color{gray}  4}, {\color{gray}  5}, {\color{gray}  6}, {\color{gray} 2}, {\color{gray} 4}, {\color{gray}  5}, {\color{gray} 3}, {\bf 4}, {\color{gray} 2}, {\color{gray} 1}, {\color{gray} 3}, {\color{gray} 4}, {\bf 5}, {\bf 6}, {\bf 7}) \\
	({\bf 7}, {\bf 6}, {\bf 5}, {\bf 4}, {\bf 2}, {\color{gray}  3}, {\color{gray}  4}, {\color{gray}  5}, {\color{gray}  6}, {\color{gray}  7}, {\color{gray}  1}, {\bf 3}, {\color{gray}  4}, {\color{gray}  5}, {\color{gray}  6}, {\color{gray} 2}, {\bf 4}, {\color{gray}  5}, {\color{gray} 3}, {\color{gray} 4}, {\color{gray} 2}, {\color{gray} 1}, {\color{gray} 3}, {\color{gray} 4}, {\bf 5}, {\bf 6}, {\bf 7}), \\
	({\bf 7}, {\bf 6}, {\bf 5}, {\bf 4}, {\bf 2}, {\bf 3}, {\color{gray}  4}, {\color{gray}  5}, {\color{gray}  6}, {\color{gray}  7}, {\color{gray}  1}, {\color{gray} 3}, {\color{gray}  4}, {\color{gray}  5}, {\color{gray}  6}, {\color{gray} 2}, {\color{gray}  4}, {\color{gray}  5}, {\color{gray} 3}, {\color{gray} 4}, {\color{gray} 2}, {\color{gray} 1}, {\color{gray} 3}, {\bf 4}, {\bf 5}, {\bf 6}, {\bf 7}), \\
	({\bf 7}, {\bf 6}, {\bf 5}, {\bf 4}, {\bf 2}, {\color{gray}  3}, {\color{gray}  4}, {\color{gray}  5}, {\color{gray}  6}, {\color{gray}  7}, {\color{gray}  1}, {\bf 3}, {\color{gray}  4}, {\color{gray}  5}, {\color{gray}  6}, {\color{gray} 2}, {\color{gray}  4}, {\color{gray}  5}, {\color{gray} 3}, {\color{gray} 4}, {\color{gray} 2}, {\color{gray} 1}, {\color{gray} 3}, {\bf 4}, {\bf 5}, {\bf 6}, {\bf 7}), \\
	({\bf 7}, {\bf 6}, {\bf 5}, {\bf 4}, {\bf 2}, {\color{gray}  3}, {\color{gray}  4}, {\color{gray}  5}, {\color{gray}  6}, {\color{gray}  7}, {\color{gray}  1}, {\color{gray}  3}, {\color{gray}  4}, {\color{gray}  5}, {\color{gray}  6}, {\color{gray} 2}, {\color{gray}  4}, {\color{gray}  5}, {\bf 3}, {\color{gray} 4}, {\color{gray} 2}, {\color{gray} 1}, {\color{gray} 3}, {\bf 4}, {\bf 5}, {\bf 6}, {\bf 7}), \\
	({\bf 7}, {\bf 6}, {\bf 5}, {\bf 4}, {\color{gray}  2}, {\color{gray}  3}, {\color{gray}  4}, {\color{gray}  5}, {\color{gray}  6}, {\color{gray}  7}, {\color{gray}  1}, {\color{gray}  3}, {\color{gray}  4}, {\color{gray}  5}, {\color{gray}  6}, {\bf 2}, {\color{gray}  4}, {\color{gray}  5}, {\bf 3}, {\color{gray} 4}, {\color{gray} 2}, {\color{gray} 1}, {\color{gray} 3}, {\bf 4}, {\bf 5}, {\bf 6}, {\bf 7}), \\
	({\bf 7}, {\bf 6}, {\bf 5}, {\color{gray} 4}, {\color{gray}  2}, {\color{gray}  3}, {\bf  4}, {\color{gray}  5}, {\color{gray}  6}, {\color{gray}  7}, {\color{gray}  1}, {\color{gray}  3}, {\color{gray}  4}, {\color{gray}  5}, {\color{gray}  6}, {\bf 2}, {\color{gray}  4}, {\color{gray}  5}, {\bf 3}, {\color{gray} 4}, {\color{gray} 2}, {\color{gray} 1}, {\color{gray} 3}, {\bf 4}, {\bf 5}, {\bf 6}, {\bf 7}), \\
	({\bf 7}, {\bf 6}, {\bf 5}, {\color{gray} 4}, {\color{gray}  2}, {\color{gray}  3}, {\color{gray} 4}, {\color{gray}  5}, {\color{gray}  6}, {\color{gray}  7}, {\color{gray}  1}, {\color{gray}  3}, {\bf  4}, {\color{gray}  5}, {\color{gray}  6}, {\bf 2}, {\color{gray}  4}, {\color{gray}  5}, {\bf 3}, {\color{gray} 4}, {\color{gray} 2}, {\color{gray} 1}, {\color{gray} 3}, {\bf 4}, {\bf 5}, {\bf 6}, {\bf 7}), \\
	({\bf 7}, {\bf 6}, {\color{gray}  5}, {\color{gray} 4}, {\color{gray}  2}, {\color{gray}  3}, {\color{gray} 4}, {\bf  5}, {\color{gray}  6}, {\color{gray}  7}, {\color{gray}  1}, {\color{gray}  3}, {\bf  4}, {\color{gray}  5}, {\color{gray}  6}, {\bf 2}, {\color{gray}  4}, {\color{gray}  5}, {\bf 3}, {\color{gray} 4}, {\color{gray} 2}, {\color{gray} 1}, {\color{gray} 3}, {\bf 4}, {\bf 5}, {\bf 6}, {\bf 7}), \\
	({\bf 7}, {\bf 6}, {\bf 5}, {\bf 4}, {\color{gray}  2}, {\color{gray}  3}, {\color{gray}  4}, {\color{gray}  5}, {\color{gray}  6}, {\color{gray}  7}, {\color{gray}  1}, {\color{gray}  3}, {\color{gray}  4}, {\color{gray}  5}, {\color{gray}  6}, {\bf 2}, {\color{gray}  4}, {\color{gray}  5}, {\bf 3}, {\bf 4}, {\color{gray} 2}, {\color{gray} 1}, {\color{gray} 3}, {\color{gray} 4}, {\bf 5}, {\bf 6}, {\bf 7}), \\
	({\bf 7}, {\bf 6}, {\bf 5}, {\color{gray} 4}, {\color{gray}  2}, {\color{gray}  3}, {\bf  4}, {\color{gray}  5}, {\color{gray}  6}, {\color{gray}  7}, {\color{gray}  1}, {\color{gray}  3}, {\color{gray}  4}, {\color{gray}  5}, {\color{gray}  6}, {\bf 2}, {\color{gray}  4}, {\color{gray}  5}, {\bf 3}, {\bf 4}, {\color{gray} 2}, {\color{gray} 1}, {\color{gray} 3}, {\color{gray} 4}, {\bf 5}, {\bf 6}, {\bf 7}), \\
	({\bf 7}, {\bf 6}, {\bf 5}, {\color{gray} 4}, {\color{gray}  2}, {\color{gray}  3}, {\color{gray}  4}, {\color{gray}  5}, {\color{gray}  6}, {\color{gray}  7}, {\color{gray}  1}, {\color{gray}  3}, {\bf  4}, {\color{gray}  5}, {\color{gray}  6}, {\bf 2}, {\color{gray}  4}, {\color{gray}  5}, {\bf 3}, {\bf 4}, {\color{gray} 2}, {\color{gray} 1}, {\color{gray} 3}, {\color{gray} 4}, {\bf 5}, {\bf 6}, {\bf 7}), \\
	({\bf 7}, {\bf 6}, {\bf 5}, {\bf 4}, {\bf 2}, {\color{gray}  3}, {\color{gray}  4}, {\color{gray}  5}, {\color{gray}  6}, {\color{gray}  7}, {\color{gray}  1}, {\color{gray}  3}, {\color{gray}  4}, {\color{gray}  5}, {\color{gray}  6}, {\color{gray} 2}, {\color{gray}  4}, {\color{gray}  5}, {\color{gray}  3}, {\color{gray} 4}, {\color{gray} 2}, {\color{gray} 1}, {\bf 3}, {\bf 4}, {\bf 5}, {\bf 6}, {\bf 7}), \\
	({\bf 7}, {\bf 6}, {\bf 5}, {\bf 4}, {\color{gray}  2}, {\color{gray}  3}, {\color{gray}  4}, {\color{gray}  5}, {\color{gray}  6}, {\color{gray}  7}, {\color{gray}  1}, {\color{gray}  3}, {\color{gray}  4}, {\color{gray}  5}, {\color{gray}  6}, {\bf 2}, {\color{gray}  4}, {\color{gray}  5}, {\color{gray}  3}, {\color{gray} 4}, {\color{gray} 2}, {\color{gray} 1}, {\bf 3}, {\bf 4}, {\bf 5}, {\bf 6}, {\bf 7}), \\
	({\bf 7}, {\bf 6}, {\bf 5}, {\color{gray} 4}, {\color{gray}  2}, {\color{gray}  3}, {\bf  4}, {\color{gray}  5}, {\color{gray}  6}, {\color{gray}  7}, {\color{gray}  1}, {\color{gray}  3}, {\color{gray}  4}, {\color{gray}  5}, {\color{gray}  6}, {\bf 2}, {\color{gray}  4}, {\color{gray}  5}, {\color{gray}  3}, {\color{gray} 4}, {\color{gray} 2}, {\color{gray} 1}, {\bf 3}, {\bf 4}, {\bf 5}, {\bf 6}, {\bf 7}), \\
	({\bf 7}, {\bf 6}, {\bf 5}, {\color{gray} 4}, {\color{gray}  2}, {\color{gray}  3}, {\color{gray}   4}, {\color{gray}  5}, {\color{gray}  6}, {\color{gray}  7}, {\color{gray}  1}, {\color{gray}  3}, {\bf  4}, {\color{gray}  5}, {\color{gray}  6}, {\bf 2}, {\color{gray}  4}, {\color{gray}  5}, {\color{gray}  3}, {\color{gray} 4}, {\color{gray} 2}, {\color{gray} 1}, {\bf 3}, {\bf 4}, {\bf 5}, {\bf 6}, {\bf 7}), \\
	({\bf 7}, {\bf 6}, {\color{gray}  5}, {\color{gray} 4}, {\color{gray}  2}, {\color{gray}  3}, {\color{gray}   4}, {\bf  5}, {\color{gray}  6}, {\color{gray}  7}, {\color{gray}  1}, {\color{gray}  3}, {\bf  4}, {\color{gray}  5}, {\color{gray}  6}, {\bf 2}, {\color{gray}  4}, {\color{gray}  5}, {\color{gray}  3}, {\color{gray} 4}, {\color{gray} 2}, {\color{gray} 1}, {\bf 3}, {\bf 4}, {\bf 5}, {\bf 6}, {\bf 7}), \\
	({\bf 7}, {\bf 6}, {\bf 5}, {\bf 4}, {\color{gray}  2}, {\color{gray}  3}, {\color{gray}  4}, {\color{gray}  5}, {\color{gray}  6}, {\color{gray}  7}, {\color{gray}  1}, {\color{gray}  3}, {\color{gray}  4}, {\color{gray}  5}, {\color{gray}  6}, {\color{gray}  2}, {\color{gray}  4}, {\color{gray}  5}, {\color{gray}  3}, {\color{gray} 4}, {\bf 2}, {\color{gray} 1}, {\bf 3}, {\bf 4}, {\bf 5}, {\bf 6}, {\bf 7}), \\
	({\bf 7}, {\bf 6}, {\bf 5}, {\color{gray}  4}, {\color{gray}  2}, {\color{gray}  3}, {\bf 4}, {\color{gray}  5}, {\color{gray}  6}, {\color{gray}  7}, {\color{gray}  1}, {\color{gray}  3}, {\color{gray}  4}, {\color{gray}  5}, {\color{gray}  6}, {\color{gray}  2}, {\color{gray}  4}, {\color{gray}  5}, {\color{gray}  3}, {\color{gray} 4}, {\bf 2}, {\color{gray} 1}, {\bf 3}, {\bf 4}, {\bf 5}, {\bf 6}, {\bf 7}), \\
	({\bf 7}, {\bf 6}, {\bf 5}, {\color{gray}  4}, {\color{gray}  2}, {\color{gray}  3}, {\color{gray}  4}, {\color{gray}  5}, {\color{gray}  6}, {\color{gray}  7}, {\color{gray}  1}, {\color{gray}  3}, {\bf  4}, {\color{gray}  5}, {\color{gray}  6}, {\color{gray}  2}, {\color{gray}  4}, {\color{gray}  5}, {\color{gray}  3}, {\color{gray} 4}, {\bf 2}, {\color{gray} 1}, {\bf 3}, {\bf 4}, {\bf 5}, {\bf 6}, {\bf 7}), \\
	({\bf 7}, {\bf 6}, {\color{gray} 5}, {\color{gray}  4}, {\color{gray}  2}, {\color{gray}  3}, {\color{gray}  4}, {\bf  5}, {\color{gray}  6}, {\color{gray}  7}, {\color{gray}  1}, {\color{gray}  3}, {\bf  4}, {\color{gray}  5}, {\color{gray}  6}, {\color{gray}  2}, {\color{gray}  4}, {\color{gray}  5}, {\color{gray}  3}, {\color{gray} 4}, {\bf 2}, {\color{gray} 1}, {\bf 3}, {\bf 4}, {\bf 5}, {\bf 6}, {\bf 7}), \\
	({\bf 7}, {\bf 6}, {\bf 5}, {\color{gray}  4}, {\color{gray}  2}, {\color{gray}  3}, {\color{gray}  4}, {\color{gray}  5}, {\color{gray}  6}, {\color{gray}  7}, {\color{gray}  1}, {\color{gray}  3}, {\color{gray}   4}, {\color{gray}  5}, {\color{gray}  6}, {\color{gray}  2}, {\bf  4}, {\color{gray}  5}, {\color{gray}  3}, {\color{gray} 4}, {\bf 2}, {\color{gray} 1}, {\bf 3}, {\bf 4}, {\bf 5}, {\bf 6}, {\bf 7}), \\
	({\bf 7}, {\bf 6}, {\color{gray}  5}, {\color{gray}  4}, {\color{gray}  2}, {\color{gray}  3}, {\color{gray}  4}, {\bf  5}, {\color{gray}  6}, {\color{gray}  7}, {\color{gray}  1}, {\color{gray}  3}, {\color{gray}   4}, {\color{gray}  5}, {\color{gray}  6}, {\color{gray}  2}, {\bf  4}, {\color{gray}  5}, {\color{gray}  3}, {\color{gray} 4}, {\bf 2}, {\color{gray} 1}, {\bf 3}, {\bf 4}, {\bf 5}, {\bf 6}, {\bf 7}), \\
	({\bf 7}, {\bf 6}, {\color{gray}  5}, {\color{gray}  4}, {\color{gray}  2}, {\color{gray}  3}, {\color{gray}  4}, {\color{gray}  5}, {\color{gray}  6}, {\color{gray}  7}, {\color{gray}  1}, {\color{gray}  3}, {\color{gray}   4}, {\bf  5}, {\color{gray}  6}, {\color{gray}  2}, {\bf  4}, {\color{gray}  5}, {\color{gray}  3}, {\color{gray} 4}, {\bf 2}, {\color{gray} 1}, {\bf 3}, {\bf 4}, {\bf 5}, {\bf 6}, {\bf 7}), \\
	({\bf 7}, {\color{gray}  6}, {\color{gray}  5}, {\color{gray}  4}, {\color{gray}  2}, {\color{gray}  3}, {\color{gray}  4}, {\color{gray}  5}, {\bf  6}, {\color{gray}  7}, {\color{gray}  1}, {\color{gray}  3}, {\color{gray}   4}, {\bf  5}, {\color{gray}  6}, {\color{gray}  2}, {\bf  4}, {\color{gray}  5}, {\color{gray}  3}, {\color{gray} 4}, {\bf 2}, {\color{gray} 1}, {\bf 3}, {\bf 4}, {\bf 5}, {\bf 6}, {\bf 7}), \\
	({\bf 7}, {\bf 6}, {\bf 5}, {\color{gray}  4}, {\color{gray}  2}, {\color{gray}  3}, {\color{gray}  4}, {\color{gray}  5}, {\color{gray}  6}, {\color{gray}  7}, {\color{gray}  1}, {\color{gray}  3}, {\color{gray}   4}, {\color{gray}  5}, {\color{gray}  6}, {\color{gray}  2}, {\color{gray} 4}, {\color{gray}  5}, {\color{gray}  3}, {\bf 4}, {\bf 2}, {\color{gray} 1}, {\bf 3}, {\bf 4}, {\bf 5}, {\bf 6}, {\bf 7}), \\
	({\bf 7}, {\bf 6}, {\color{gray}  5}, {\color{gray}  4}, {\color{gray}  2}, {\color{gray}  3}, {\color{gray}  4}, {\bf  5}, {\color{gray}  6}, {\color{gray}  7}, {\color{gray}  1}, {\color{gray}  3}, {\color{gray}   4}, {\color{gray}  5}, {\color{gray}  6}, {\color{gray}  2}, {\color{gray} 4}, {\color{gray}  5}, {\color{gray}  3}, {\bf 4}, {\bf 2}, {\color{gray} 1}, {\bf 3}, {\bf 4}, {\bf 5}, {\bf 6}, {\bf 7}), \\
	({\bf 7}, {\bf 6}, {\color{gray}  5}, {\color{gray}  4}, {\color{gray}  2}, {\color{gray}  3}, {\color{gray}  4}, {\color{gray}   5}, {\color{gray}  6}, {\color{gray}  7}, {\color{gray}  1}, {\color{gray}  3}, {\color{gray}   4}, {\bf  5}, {\color{gray}  6}, {\color{gray}  2}, {\color{gray} 4}, {\color{gray}  5}, {\color{gray}  3}, {\bf 4}, {\bf 2}, {\color{gray} 1}, {\bf 3}, {\bf 4}, {\bf 5}, {\bf 6}, {\bf 7}), \\
	({\bf 7}, {\color{gray}  6}, {\color{gray}  5}, {\color{gray}  4}, {\color{gray}  2}, {\color{gray}  3}, {\color{gray}  4}, {\color{gray}   5}, {\bf  6}, {\color{gray}  7}, {\color{gray}  1}, {\color{gray}  3}, {\color{gray}   4}, {\bf  5}, {\color{gray}  6}, {\color{gray}  2}, {\color{gray} 4}, {\color{gray}  5}, {\color{gray}  3}, {\bf 4}, {\bf 2}, {\color{gray} 1}, {\bf 3}, {\bf 4}, {\bf 5}, {\bf 6}, {\bf 7}), \\
	({\bf 7}, {\bf 6}, {\color{gray}  5}, {\color{gray}  4}, {\color{gray}  2}, {\color{gray}  3}, {\color{gray}  4}, {\color{gray}   5}, {\color{gray}  6}, {\color{gray}  7}, {\color{gray}  1}, {\color{gray}  3}, {\color{gray}   4}, {\color{gray}   5}, {\color{gray}  6}, {\color{gray}  2}, {\color{gray} 4}, {\bf 5}, {\color{gray}  3}, {\bf 4}, {\bf 2}, {\color{gray} 1}, {\bf 3}, {\bf 4}, {\bf 5}, {\bf 6}, {\bf 7}), \\
	({\bf 7}, {\color{gray}  6}, {\color{gray}  5}, {\color{gray}  4}, {\color{gray}  2}, {\color{gray}  3}, {\color{gray}  4}, {\color{gray}   5}, {\bf  6}, {\color{gray}  7}, {\color{gray}  1}, {\color{gray}  3}, {\color{gray}   4}, {\color{gray}   5}, {\color{gray}  6}, {\color{gray}  2}, {\color{gray} 4}, {\bf 5}, {\color{gray}  3}, {\bf 4}, {\bf 2}, {\color{gray} 1}, {\bf 3}, {\bf 4}, {\bf 5}, {\bf 6}, {\bf 7}), \\
	({\bf 7}, {\color{gray}  6}, {\color{gray}  5}, {\color{gray}  4}, {\color{gray}  2}, {\color{gray}  3}, {\color{gray}  4}, {\color{gray}   5}, {\color{gray}   6}, {\color{gray}  7}, {\color{gray}  1}, {\color{gray}  3}, {\color{gray}   4}, {\color{gray}   5}, {\bf 6}, {\color{gray}  2}, {\color{gray} 4}, {\bf 5}, {\color{gray}  3}, {\bf 4}, {\bf 2}, {\color{gray} 1}, {\bf 3}, {\bf 4}, {\bf 5}, {\bf 6}, {\bf 7}), \\
	({\color{gray}  7}, {\color{gray}  6}, {\color{gray}  5}, {\color{gray}  4}, {\color{gray}  2}, {\color{gray}  3}, {\color{gray}  4}, {\color{gray}   5}, {\color{gray}   6}, {\bf  7}, {\color{gray}  1}, {\color{gray}  3}, {\color{gray}   4}, {\color{gray}   5}, {\bf 6}, {\color{gray}  2}, {\color{gray} 4}, {\bf 5}, {\color{gray}  3}, {\bf 4}, {\bf 2}, {\color{gray} 1}, {\bf 3}, {\bf 4}, {\bf 5}, {\bf 6}, {\bf 7}).
\end{gather*}}

{\it Sub-Case 2}. (7, 6, 5, 4, {\bf 3}, {\bf 2}, 4, 5, 6, 7).
{\allowdisplaybreaks \footnotesize
\begin{gather*}
	({\bf 7}, {\bf 6}, {\bf 5}, {\bf 4}, {\color{gray}  2}, {\bf  3}, {\color{gray}  4}, {\color{gray}  5}, {\color{gray}  6}, {\color{gray}  7}, {\color{gray}  1}, {\color{gray}  3}, {\color{gray}  4}, {\color{gray}  5}, {\color{gray}  6}, {\bf 2}, {\bf 4}, {\bf 5}, {\color{gray} 3}, {\color{gray} 4}, {\color{gray} 2}, {\color{gray} 1}, {\color{gray} 3}, {\color{gray} 4}, {\color{gray} 5}, {\bf 6}, {\bf 7}), \\
	({\bf 7}, {\bf 6}, {\bf 5}, {\bf 4}, {\color{gray}  2}, {\color{gray}   3}, {\color{gray}  4}, {\color{gray}  5}, {\color{gray}  6}, {\color{gray}  7}, {\color{gray}  1}, {\bf  3}, {\color{gray}  4}, {\color{gray}  5}, {\color{gray}  6}, {\bf 2}, {\bf 4}, {\bf 5}, {\color{gray} 3}, {\color{gray} 4}, {\color{gray} 2}, {\color{gray} 1}, {\color{gray} 3}, {\color{gray} 4}, {\color{gray} 5}, {\bf 6}, {\bf 7}), \\
	({\bf 7}, {\bf 6}, {\bf 5}, {\color{gray}  4}, {\color{gray}  2}, {\color{gray}   3}, {\bf  4}, {\color{gray}  5}, {\color{gray}  6}, {\color{gray}  7}, {\color{gray}  1}, {\bf  3}, {\color{gray}  4}, {\color{gray}  5}, {\color{gray}  6}, {\bf 2}, {\bf 4}, {\bf 5}, {\color{gray} 3}, {\color{gray} 4}, {\color{gray} 2}, {\color{gray} 1}, {\color{gray} 3}, {\color{gray} 4}, {\color{gray} 5}, {\bf 6}, {\bf 7}), \\
	({\bf 7}, {\bf 6}, {\bf 5}, {\bf 4}, {\color{gray}  2}, {\bf  3}, {\color{gray}  4}, {\color{gray}  5}, {\color{gray}  6}, {\color{gray}  7}, {\color{gray}  1}, {\color{gray}  3}, {\color{gray}  4}, {\color{gray}  5}, {\color{gray}  6}, {\bf 2}, {\bf 4}, {\color{gray}  5}, {\color{gray} 3}, {\color{gray} 4}, {\color{gray} 2}, {\color{gray} 1}, {\color{gray} 3}, {\color{gray} 4}, {\bf 5}, {\bf 6}, {\bf 7}), \\
	({\bf 7}, {\bf 6}, {\bf 5}, {\bf 4}, {\color{gray}  2}, {\color{gray}  3}, {\color{gray}  4}, {\color{gray}  5}, {\color{gray}  6}, {\color{gray}  7}, {\color{gray}  1}, {\bf  3}, {\color{gray}  4}, {\color{gray}  5}, {\color{gray}  6}, {\bf 2}, {\bf 4}, {\color{gray}  5}, {\color{gray} 3}, {\color{gray} 4}, {\color{gray} 2}, {\color{gray} 1}, {\color{gray} 3}, {\color{gray} 4}, {\bf 5}, {\bf 6}, {\bf 7}), \\
	({\bf 7}, {\bf 6}, {\bf 5}, {\color{gray} 4}, {\color{gray}  2}, {\color{gray}  3}, {\bf  4}, {\color{gray}  5}, {\color{gray}  6}, {\color{gray}  7}, {\color{gray}  1}, {\bf  3}, {\color{gray}  4}, {\color{gray}  5}, {\color{gray}  6}, {\bf 2}, {\bf 4}, {\color{gray}  5}, {\color{gray} 3}, {\color{gray} 4}, {\color{gray} 2}, {\color{gray} 1}, {\color{gray} 3}, {\color{gray} 4}, {\bf 5}, {\bf 6}, {\bf 7}), \\
	({\bf 7}, {\bf 6}, {\bf 5}, {\bf 4}, {\color{gray}  2}, {\bf  3}, {\color{gray}  4}, {\color{gray}  5}, {\color{gray}  6}, {\color{gray}  7}, {\color{gray}  1}, {\color{gray}  3}, {\color{gray}  4}, {\color{gray}  5}, {\color{gray}  6}, {\bf 2}, {\color{gray}  4}, {\color{gray}  5}, {\color{gray} 3}, {\bf 4}, {\color{gray} 2}, {\color{gray} 1}, {\color{gray} 3}, {\color{gray} 4}, {\bf 5}, {\bf 6}, {\bf 7}), \\
	({\bf 7}, {\bf 6}, {\bf 5}, {\bf 4}, {\color{gray}  2},{\color{gray}   3}, {\color{gray}  4}, {\color{gray}  5}, {\color{gray}  6}, {\color{gray}  7}, {\color{gray}  1}, {\bf  3}, {\color{gray}  4}, {\color{gray}  5}, {\color{gray}  6}, {\bf 2}, {\color{gray}  4}, {\color{gray}  5}, {\color{gray} 3}, {\bf 4}, {\color{gray} 2}, {\color{gray} 1}, {\color{gray} 3}, {\color{gray} 4}, {\bf 5}, {\bf 6}, {\bf 7}), \\
	({\bf 7}, {\bf 6}, {\bf 5}, {\color{gray}  4}, {\color{gray}  2},{\color{gray}   3}, {\bf  4}, {\color{gray}  5}, {\color{gray}  6}, {\color{gray}  7}, {\color{gray}  1}, {\bf  3}, {\color{gray}  4}, {\color{gray}  5}, {\color{gray}  6}, {\bf 2}, {\color{gray}  4}, {\color{gray}  5}, {\color{gray} 3}, {\bf 4}, {\color{gray} 2}, {\color{gray} 1}, {\color{gray} 3}, {\color{gray} 4}, {\bf 5}, {\bf 6}, {\bf 7}), \\
	({\bf 7}, {\bf 6}, {\bf 5}, {\bf 4}, {\color{gray}  2}, {\bf  3}, {\color{gray}  4}, {\color{gray}  5}, {\color{gray}  6}, {\color{gray}  7}, {\color{gray}  1}, {\color{gray}  3}, {\color{gray}  4}, {\color{gray}  5}, {\color{gray}  6}, {\bf 2}, {\color{gray}  4}, {\color{gray}  5}, {\color{gray} 3}, {\color{gray}  4}, {\color{gray} 2}, {\color{gray} 1}, {\color{gray} 3}, {\bf 4}, {\bf 5}, {\bf 6}, {\bf 7}), \\
	({\bf 7}, {\bf 6}, {\bf 5}, {\bf 4}, {\color{gray}  2}, {\bf  3}, {\color{gray}  4}, {\color{gray}  5}, {\color{gray}  6}, {\color{gray}  7}, {\color{gray}  1}, {\color{gray}  3}, {\color{gray}  4}, {\color{gray}  5}, {\color{gray}  6}, {\color{gray}  2}, {\color{gray}  4}, {\color{gray}  5}, {\color{gray} 3}, {\color{gray}  4}, {\bf 2}, {\color{gray} 1}, {\color{gray} 3}, {\bf 4}, {\bf 5}, {\bf 6}, {\bf 7}), \\
	({\bf 7}, {\bf 6}, {\bf 5}, {\bf 4}, {\color{gray}  2}, {\color{gray}  3}, {\color{gray}  4}, {\color{gray}  5}, {\color{gray}  6}, {\color{gray}  7}, {\color{gray}  1}, {\bf  3}, {\color{gray}  4}, {\color{gray}  5}, {\color{gray}  6}, {\color{gray}  2}, {\color{gray}  4}, {\color{gray}  5}, {\color{gray} 3}, {\color{gray}  4}, {\bf 2}, {\color{gray} 1}, {\color{gray} 3}, {\bf 4}, {\bf 5}, {\bf 6}, {\bf 7}), \\
	({\bf 7}, {\bf 6}, {\bf 5}, {\color{gray}  4}, {\color{gray}  2}, {\color{gray}  3}, {\bf  4}, {\color{gray}  5}, {\color{gray}  6}, {\color{gray}  7}, {\color{gray}  1}, {\bf  3}, {\color{gray}  4}, {\color{gray}  5}, {\color{gray}  6}, {\color{gray}  2}, {\color{gray}  4}, {\color{gray}  5}, {\color{gray} 3}, {\color{gray}  4}, {\bf 2}, {\color{gray} 1}, {\color{gray} 3}, {\bf 4}, {\bf 5}, {\bf 6}, {\bf 7}), \\
	({\bf 7}, {\bf 6}, {\bf 5}, {\bf 4}, {\color{gray}  2}, {\color{gray}  3}, {\color{gray}  4}, {\color{gray}  5}, {\color{gray}  6}, {\color{gray}  7}, {\color{gray}  1}, {\color{gray} 3}, {\color{gray}  4}, {\color{gray}  5}, {\color{gray}  6}, {\color{gray}  2}, {\color{gray}  4}, {\color{gray}  5}, {\bf 3}, {\color{gray}  4}, {\bf 2}, {\color{gray} 1}, {\color{gray} 3}, {\bf 4}, {\bf 5}, {\bf 6}, {\bf 7}), \\
	({\bf 7}, {\bf 6}, {\bf 5}, {\color{gray}  4}, {\color{gray}  2}, {\color{gray}  3}, {\bf  4}, {\color{gray}  5}, {\color{gray}  6}, {\color{gray}  7}, {\color{gray}  1}, {\color{gray} 3}, {\color{gray}  4}, {\color{gray}  5}, {\color{gray}  6}, {\color{gray}  2}, {\color{gray}  4}, {\color{gray}  5}, {\bf 3}, {\color{gray}  4}, {\bf 2}, {\color{gray} 1}, {\color{gray} 3}, {\bf 4}, {\bf 5}, {\bf 6}, {\bf 7}), \\
	({\bf 7}, {\bf 6}, {\bf 5}, {\color{gray}  4}, {\color{gray}  2}, {\color{gray}  3}, {\color{gray}   4}, {\color{gray}  5}, {\color{gray}  6}, {\color{gray}  7}, {\color{gray}  1}, {\color{gray} 3}, {\bf  4}, {\color{gray}  5}, {\color{gray}  6}, {\color{gray}  2}, {\color{gray}  4}, {\color{gray}  5}, {\bf 3}, {\color{gray}  4}, {\bf 2}, {\color{gray} 1}, {\color{gray} 3}, {\bf 4}, {\bf 5}, {\bf 6}, {\bf 7}), \\
	({\bf 7}, {\bf 6}, {\color{gray}  5}, {\color{gray}  4}, {\color{gray}  2}, {\color{gray}  3}, {\color{gray}   4}, {\bf 5}, {\color{gray}  6}, {\color{gray}  7}, {\color{gray}  1}, {\color{gray} 3}, {\bf  4}, {\color{gray}  5}, {\color{gray}  6}, {\color{gray}  2}, {\color{gray}  4}, {\color{gray}  5}, {\bf 3}, {\color{gray}  4}, {\bf 2}, {\color{gray} 1}, {\color{gray} 3}, {\bf 4}, {\bf 5}, {\bf 6}, {\bf 7}), \\
	({\bf 7}, {\bf 6}, {\bf 5}, {\color{gray}  4}, {\color{gray}  2}, {\color{gray}  3}, {\color{gray}   4}, {\color{gray}  5}, {\color{gray}  6}, {\color{gray}  7}, {\color{gray}  1}, {\color{gray} 3}, {\color{gray}   4}, {\color{gray}  5}, {\color{gray}  6}, {\color{gray}  2}, {\bf  4}, {\color{gray}  5}, {\bf 3}, {\color{gray}  4}, {\bf 2}, {\color{gray} 1}, {\color{gray} 3}, {\bf 4}, {\bf 5}, {\bf 6}, {\bf 7}), \\
	({\bf 7}, {\bf 6}, {\color{gray}  5}, {\color{gray}  4}, {\color{gray}  2}, {\color{gray}  3}, {\color{gray}   4}, {\bf  5}, {\color{gray}  6}, {\color{gray}  7}, {\color{gray}  1}, {\color{gray} 3}, {\color{gray}   4}, {\color{gray}  5}, {\color{gray}  6}, {\color{gray}  2}, {\bf  4}, {\color{gray}  5}, {\bf 3}, {\color{gray}  4}, {\bf 2}, {\color{gray} 1}, {\color{gray} 3}, {\bf 4}, {\bf 5}, {\bf 6}, {\bf 7}), \\
	({\bf 7}, {\bf 6}, {\color{gray}  5}, {\color{gray}  4}, {\color{gray}  2}, {\color{gray}  3}, {\color{gray}   4}, {\color{gray}   5}, {\color{gray}  6}, {\color{gray}  7}, {\color{gray}  1}, {\color{gray} 3}, {\color{gray}   4}, {\bf  5}, {\color{gray}  6}, {\color{gray}  2}, {\bf  4}, {\color{gray}  5}, {\bf 3}, {\color{gray}  4}, {\bf 2}, {\color{gray} 1}, {\color{gray} 3}, {\bf 4}, {\bf 5}, {\bf 6}, {\bf 7}), \\
	({\bf 7}, {\color{gray}  6}, {\color{gray}  5}, {\color{gray}  4}, {\color{gray}  2}, {\color{gray}  3}, {\color{gray}   4}, {\color{gray}   5}, {\bf  6}, {\color{gray}  7}, {\color{gray}  1}, {\color{gray} 3}, {\color{gray}   4}, {\bf  5}, {\color{gray}  6}, {\color{gray}  2}, {\bf  4}, {\color{gray}  5}, {\bf 3}, {\color{gray}  4}, {\bf 2}, {\color{gray} 1}, {\color{gray} 3}, {\bf 4}, {\bf 5}, {\bf 6}, {\bf 7}), \\
	({\bf 7}, {\bf 6}, {\bf 5}, {\bf 4}, {\color{gray}  2}, {\color{gray}   3}, {\color{gray}  4}, {\color{gray}  5}, {\color{gray}  6}, {\color{gray}  7}, {\color{gray}  1}, {\bf  3}, {\color{gray}  4}, {\color{gray}  5}, {\color{gray}  6}, {\bf 2}, {\color{gray}  4}, {\color{gray}  5}, {\color{gray} 3}, {\color{gray}  4}, {\color{gray} 2}, {\color{gray} 1}, {\color{gray} 3}, {\bf 4}, {\bf 5}, {\bf 6}, {\bf 7}), \\
	({\bf 7}, {\bf 6}, {\bf 5}, {\color{gray}  4}, {\color{gray}  2}, {\color{gray}   3}, {\bf  4}, {\color{gray}  5}, {\color{gray}  6}, {\color{gray}  7}, {\color{gray}  1}, {\bf  3}, {\color{gray}  4}, {\color{gray}  5}, {\color{gray}  6}, {\bf 2}, {\color{gray}  4}, {\color{gray}  5}, {\color{gray} 3}, {\color{gray}  4}, {\color{gray} 2}, {\color{gray} 1}, {\color{gray} 3}, {\bf 4}, {\bf 5}, {\bf 6}, {\bf 7}).
\end{gather*}
}

\begin{rem}
{\em
The configurations of the paths in Definition \ref{df:paths in Pn} are devised following the above description of the ${\bf j}_0$-trails in $\mathcal{T}^{\,'}$ (cf. Examples \ref{ex:array in E7} and \ref{ex:P to D}).
}
\end{rem}

\subsubsection{Description of ${\bf j}_0$-trails in $\mathcal{T}^{\,'}$}
In this subsection, we describe the ${\bf j}_0$-trails in $\mathcal{T}^{\,'}$ in terms of the path on $\Delta$. 
From now on, we assume $r = n$ for simplicity.

Let $\mathcal{D}$ be the set of arrays, where either $0$ or $1$ is placed in each row of $\Delta^{\J}$.
For ${\bf \pi} \in \mathcal{T}^{\,'}$, we can identify a ${\bf j}_0$-trail $\pi$ in  $\mathcal{T}^{\,'}$ with a unique array ${\bf d}(\pi) = (d_k(\pi))_{N-M+1 \le k \le N}$ of $\mathcal{D}$ determined as follows.
\vskip 3mm

{\it Case 1.} Type $\text{E}_6$. In this case, the array ${\bf d}(\pi)$ is given by
\begin{equation*}
\begin{split}
& \scalebox{1.3}{
\begin{tikzpicture}[->,>=stealth',auto,node distance=3cm,main node/.style={circle,draw,font=\sffamily\Large\bfseries}]
  \node (1-1) at (-0.1,0) {${}_{\scalebox{0.6}{1}}$}; 
  \node (3-1) at (0.6,0) {${}_{\scalebox{0.6}{3}}$}; 
  \node (4-1) at (1.3,0) {${}_{\scalebox{0.6}{4}}$}; 
  \node (5-1) at (2,0) {${}_{\scalebox{0.6}{5}}$}; 
  \node (6-1) at (2.7,0) {${}_{\scalebox{0.6}{6}}$}; 
  \node at (1.3, -0.6) {${}_{\scalebox{0.6}{2}}$}; 
  \node (4-2) at (2, -0.6) {${}_{\scalebox{0.6}{4}}$}; 
  \node (5-2) at (2.7, -0.6) {${}_{\scalebox{0.6}{5}}$}; 
  \node (3-2) at (2, -1.3) {${}_{\scalebox{0.6}{3}}$}; 
  \node (4-3) at (2.7, -1.3) {${}_{\scalebox{0.6}{4}}$}; 
  \node at (3.4, -1.3) {${}_{\scalebox{0.6}{2}}$}; 
  \node (1-2) at (2, -2) {${}_{\scalebox{0.6}{1}}$}; 
  \node (3-3) at (2.7, -2) {${}_{\scalebox{0.6}{3}}$}; 
  \node (4-4) at (3.4, -2) {${}_{\scalebox{0.6}{4}}$}; 
  \node (5-3) at (4.1, -2) {${}_{\scalebox{0.6}{5}}$}; 
  \node (6-2) at (4.8, -2) {${}_{\scalebox{0.6}{6}}$}; 
\end{tikzpicture}}
\hspace{-1cm}
\scalebox{1.4}{
\begin{tikzpicture}[->,>=stealth',auto,node distance=3cm,main node/.style={circle,draw,font=\sffamily\Large\bfseries}]
  \node at (-0.1,0) {${}_{\scalebox{0.5}{$d_{36}(\pi)$}}$}; 
  \node at (0.6,0) {${}_{\scalebox{0.5}{$d_{35}(\pi)$}}$}; 
  \node at (1.3,0) {${}_{\scalebox{0.5}{$d_{34}(\pi)$}}$}; 
  \node at (2,0) {${}_{\scalebox{0.5}{$d_{33}(\pi)$}}$}; 
  \node at (2.7,0) {${}_{\scalebox{0.5}{$d_{32}(\pi)$}}$}; 
  \node at (1.3, -0.6) {${}_{\scalebox{0.5}{$d_{31}(\pi)$}}$}; 
  \node at (2, -0.6) {${}_{\scalebox{0.5}{$d_{30}(\pi)$}}$}; 
  \node at (2.7, -0.6) {${}_{\scalebox{0.5}{$d_{29}(\pi)$}}$}; 
  \node at (2, -1.3) {${}_{\scalebox{0.5}{$d_{28}(\pi)$}}$}; 
  \node at (2.7, -1.3) {${}_{\scalebox{0.5}{$d_{27}(\pi)$}}$}; 
  \node at (3.4, -1.3) {${}_{\scalebox{0.5}{$d_{26}(\pi)$}}$}; 
  \node at (2, -2) {${}_{\scalebox{0.5}{$d_{25}(\pi)$}}$}; 
  \node at (2.7, -2) {${}_{\scalebox{0.5}{$d_{24}(\pi)$}}$}; 
  \node at (3.4, -2) {${}_{\scalebox{0.5}{$d_{23}(\pi)$}}$}; 
  \node at (4.1, -2) {${}_{\scalebox{0.5}{$d_{22}(\pi)$}}$}; 
  \node at (4.8, -2) {${}_{\scalebox{0.5}{$d_{21}(\pi)$}}$}; 
\end{tikzpicture}}
\end{split}
\end{equation*}

{\it Case 2.} Type $\text{E}_7$. In this case, the array ${\bf d}(\pi)$ is given by
\begin{equation*}
\scalebox{0.97}{
\begin{tikzpicture}[->,>=stealth',auto,node distance=3cm,main node/.style={circle,draw,font=\sffamily\Large\bfseries}]
  \node (7-1) at (-0.8,0) {${}_{7}$}; 
  \node (6-1) at (-0.1,0) {${}_{6}$}; 
  \node (5-1) at (0.6,0) {${}_{5}$}; 
  \node (4-1) at (1.3,0) {${}_{4}$}; 
  \node (3-1) at (2,0) {${}_{3}$}; 
  \node (1-1) at (2.7,0) {${}_{1}$}; 

  \node (2-1) at (1.3, -0.6) {${}_{2}$}; 
  \node (4-2) at (2, -0.6) {${}_{4}$}; 
  \node (3-2) at (2.7, -0.6) {${}_{3}$}; 

  \node (5-2) at (2, -1.3) {${}_{5}$}; 
  \node (4-3) at (2.7, -1.3) {${}_{4}$}; 
  \node (2-2) at (3.4, -1.3) {${}_{2}$}; 

  \node (6-2) at (2, -2) {${}_{6}$}; 
  \node (5-3) at (2.7, -2) {${}_{5}$}; 
  \node (4-4) at (3.4, -2) {${}_{4}$}; 
  \node (3-3) at (4.1, -2) {${}_{3}$}; 
  \node (1-2) at (4.8, -2) {${}_{1}$}; 

  \node (7-2) at (2, -2.7) {${}_{7}$}; 
  \node (6-3) at (2.7, -2.7) {${}_{6}$}; 
  \node (5-4) at (3.4, -2.7) {${}_{5}$}; 
  \node (4-5) at (4.1, -2.7) {${}_{4}$}; 
  \node (3-4) at (4.8, -2.7) {${}_{3}$}; 
  
  \node (2-3) at (4.1, -3.4) {${}_{2}$}; 
  \node (4-6) at (4.8, -3.4) {${}_{4}$}; 

  \node (5-5) at (4.8, -4.1) {${}_{5}$}; 

  \node (6-4) at (4.8, -4.8) {${}_{6}$}; 
  
  \node (7-3) at (4.8, -5.5) {${}_{7}$}; 
\end{tikzpicture}
}
\scalebox{0.65}{
\begin{tikzpicture}  
  \node at (-1.4,0) {\scalebox{1}{$d_{63}(\pi)$}}; 
  \node at (-0.1,0) {\scalebox{1}{$d_{62}(\pi)$}}; 
  \node at (1.2,0) {\scalebox{1}{$d_{61}(\pi)$}}; 
  \node at (2.5,0) {\scalebox{1}{$d_{60}(\pi)$}}; 
  \node at (3.8,0) {\scalebox{1}{$d_{59}(\pi)$}}; 
  \node at (5.1,0) {\scalebox{1}{$d_{58}(\pi)$}}; 

  \node at (2.5, -1) {\scalebox{1}{$d_{57}(\pi)$}}; 
  \node at (3.8, -1) {\scalebox{1}{$d_{56}(\pi)$}}; 
  \node at (5.1, -1) {\scalebox{1}{$d_{55}(\pi)$}}; 
  
  \node at (3.8, -2) {\scalebox{1}{$d_{54}(\pi)$}}; 
  \node at (5.1, -2) {\scalebox{1}{$d_{53}(\pi)$}}; 
  \node at (6.4, -2) {\scalebox{1}{$d_{52}(\pi)$}}; 
  
  \node at (3.8, -3) {\scalebox{1}{$d_{51}(\pi)$}}; 
  \node at (5.1, -3) {\scalebox{1}{$d_{50}(\pi)$}}; 
  \node at (6.4, -3) {\scalebox{1}{$d_{49}(\pi)$}}; 
  \node at (7.7, -3) {\scalebox{1}{$d_{48}(\pi)$}}; 
  \node at (9, -3) {\scalebox{1}{$d_{47}(\pi)$}}; 
  \node at (3.8, -4) {\scalebox{1}{$d_{46}(\pi)$}}; 
  
  \node at (5.1, -4) {\scalebox{1}{$d_{45}(\pi)$}}; 
  \node at (6.4, -4) {\scalebox{1}{$d_{44}(\pi)$}}; 
  \node at (7.7, -4) {\scalebox{1}{$d_{43}(\pi)$}}; 
  \node at (9, -4) {\scalebox{1}{$d_{42}(\pi)$}}; 
  \node at (7.7, -5) {\scalebox{1}{$d_{41}(\pi)$}}; 
  \node at (9, -5) {\scalebox{1}{$d_{40}(\pi)$}}; 
  \node at (9, -6) {\scalebox{1}{$d_{39}(\pi)$}}; 
  \node at (9, -7) {\scalebox{1}{$d_{38}(\pi)$}}; 
  \node at (9, -8) {\scalebox{1}{$d_{37}(\pi)$}}; 
\end{tikzpicture}
}
\end{equation*}
Then it is obvious that the map
\begin{equation} \label{eq:from T' to D}
\xymatrixcolsep{3pc}\xymatrixrowsep{0pc}
\xymatrix{
 \mathcal{T}^{\,'} \ \ar@{->}[r] & \ \mathcal{D} \\
 \pi \ \ar@{|->}[r] & \ {\bf d}(\pi)
}
\end{equation}
is injective since $\pi_{\J}$ is fixed by $(\tilde{\nu}_0, \dots, \tilde{\nu}_{N-M})$.
Let $\mathcal{D}^{\J}$ be the image of $\mathcal{T}^{\,'}$ under the above map.
Thanks to Section \ref{subsec:j0-trails}, we know the complete description of the image $\mathcal{D}^{\J}$.

\begin{ex} \label{ex:array in E7}
{\em 
Let us consider the case of type $\text{E}_7$. Let $\pi \in \mathcal{T}^{\,'}$ be the ${\bf j}_0$-trail corresponding to
\begin{equation*}
({\bf 7}, {\bf 6}, {\bf 5}, {\color{gray} 4}, {\color{gray}  2}, {\color{gray}  3}, {\bf  4}, {\color{gray}  5}, {\color{gray}  6}, {\color{gray}  7}, {\color{gray}  1}, {\color{gray}  3}, {\color{gray}  4}, {\color{gray}  5}, {\color{gray}  6}, {\bf 2}, {\color{gray}  4}, {\color{gray}  5}, {\bf 3}, {\color{gray} 4}, {\color{gray} 2}, {\color{gray} 1}, {\color{gray} 3}, {\bf 4}, {\bf 5}, {\bf 6}, {\bf 7}).
\end{equation*}
Then ${\bf d}(\pi)$ is given by
\begin{equation*}
\scalebox{1.1}{
\begin{tikzpicture}[->,>=stealth',auto,node distance=3cm,main node/.style={circle,draw,font=\sffamily\Large\bfseries}]
  \node (7-1) at (-0.8,0) {${}_{\bf 7}$}; 
  \node (6-1) at (-0.1,0) {${}_{\bf 6}$}; 
  \node (5-1) at (0.6,0) {${}_{\bf 5}$}; 
  \node (4-1) at (1.3,0) {${}_{\bf 4}$}; 
  \node (3-1) at (2,0) {${}_{\color{gray} 3}$}; 
  \node (1-1) at (2.7,0) {${}_{\color{gray} 1}$}; 

  \node (2-1) at (1.3, -0.6) {${}_{\color{gray} 2}$}; 
  \node (4-2) at (2, -0.6) {${}_{\color{gray} 4}$}; 
  \node (3-2) at (2.7, -0.6) {${}_{\bf 3}$}; 

  \node (5-2) at (2, -1.3) {${}_{\color{gray} 5}$}; 
  \node (4-3) at (2.7, -1.3) {${}_{\color{gray} 4}$}; 
  \node (2-2) at (3.4, -1.3) {${}_{\bf 2}$}; 

  \node (6-2) at (2, -2) {${}_{\color{gray} 6}$}; 
  \node (5-3) at (2.7, -2) {${}_{\color{gray} 5}$}; 
  \node (4-4) at (3.4, -2) {${}_{\color{gray} 4}$}; 
  \node (3-3) at (4.1, -2) {${}_{\color{gray} 3}$}; 
  \node (1-2) at (4.8, -2) {${}_{\color{gray} 1}$}; 

  \node (7-2) at (2, -2.7) {${}_{\color{gray} 7}$}; 
  \node (6-3) at (2.7, -2.7) {${}_{\color{gray} 6}$}; 
  \node (5-4) at (3.4, -2.7) {${}_{\color{gray} 5}$}; 
  \node (4-5) at (4.1, -2.7) {${}_{\bf 4}$}; 
  \node (3-4) at (4.8, -2.7) {${}_{\color{gray} 3}$}; 
  
  \node (2-3) at (4.1, -3.4) {${}_{\color{gray} 2}$}; 
  \node (4-6) at (4.8, -3.4) {${}_{\color{gray} 4}$}; 

  \node (5-5) at (4.8, -4.1) {${}_{\bf 5}$}; 

  \node (6-4) at (4.8, -4.8) {${}_{\bf 6}$}; 
  
  \node (7-3) at (4.8, -5.5) {${}_{\bf 7}$}; 
\end{tikzpicture}
}
\hspace{-1.7cm}
\scalebox{0.8}{
\begin{tikzpicture}  
  \node at (-1.4,0) {\scalebox{0.8}{$\bf 1$}}; 
  \node at (-0.1,0) {\scalebox{0.8}{$\bf 1$}}; 
  \node at (1.2,0) {\scalebox{0.8}{$\bf 1$}}; 
  \node at (2.5,0) {\scalebox{0.8}{$\bf 1$}}; 
  \node at (3.8,0) {\scalebox{0.8}{$\color{gray} 0$}}; 
  \node at (5.1,0) {\scalebox{0.8}{$\color{gray} 0$}}; 

  \node at (2.5, -1) {\scalebox{0.8}{$\color{gray} 0$}}; 
  \node at (3.8, -1) {\scalebox{0.8}{$\color{gray} 0$}}; 
  \node at (5.1, -1) {\scalebox{0.8}{$\bf 1$}}; 
  
  \node at (3.8, -2) {\scalebox{0.8}{$\color{gray} 0$}}; 
  \node at (5.1, -2) {\scalebox{0.8}{$\color{gray} 0$}}; 
  \node at (6.4, -2) {\scalebox{0.8}{$\bf 1$}}; 
  
  \node at (3.8, -3) {\scalebox{0.8}{$\color{gray} 0$}}; 
  \node at (5.1, -3) {\scalebox{0.8}{$\color{gray} 0$}}; 
  \node at (6.4, -3) {\scalebox{0.8}{$\color{gray} 0$}}; 
  \node at (7.7, -3) {\scalebox{0.8}{$\color{gray} 0$}}; 
  \node at (9, -3) {\scalebox{0.8}{$\color{gray} 0$}}; 
  \node at (3.8, -4) {\scalebox{0.8}{$\color{gray} 0$}}; 
  
  \node at (5.1, -4) {\scalebox{0.8}{$\color{gray} 0$}}; 
  \node at (6.4, -4) {\scalebox{0.8}{$\color{gray} 0$}}; 
  \node at (7.7, -4) {\scalebox{0.8}{$\bf 1$}}; 
  \node at (9, -4) {\scalebox{0.8}{$\color{gray} 0$}}; 
  \node at (7.7, -5) {\scalebox{0.8}{$\color{gray} 0$}}; 
  \node at (9, -5) {\scalebox{0.8}{$\color{gray} 0$}}; 
  \node at (9, -6) {\scalebox{0.8}{$\bf 1$}}; 
  \node at (9, -7) {\scalebox{0.8}{$\bf 1$}}; 
  \node at (9, -8) {\scalebox{0.8}{$\bf 1$}}; 
\end{tikzpicture}
}
\end{equation*}
}
\end{ex}
\vskip 1mm

By Remark \ref{rem:connectedness}, for any $\pi \in \mathcal{T}^{\,'}$, there exists a sequence $\pi_1, \dots, \pi_{\ell}$ of $\mathcal{T}^{\,'}$ such that $\pi_{\ell} = \pi$ and $\pi_{k+1}$ is obtained from $\pi_k$ by moving a bold number to the right for $0 \le k \le \ell-1$.
Let $i_k$ be the bold number moved in $\pi_k$. Then
under the identification via the map \eqref{eq:from T' to D}, the array ${\bf d}(\pi_{k+1})$ is obtained from ${\bf d}(\pi_k)$ by moving $1$ on ${\bf d}(\pi_k)$ located at the position of $i_k$ along the dashed arrow between $i_k$'s in \eqref{eq:dashed arrows}.

\begin{equation} \label{eq:dashed arrows}
\scalebox{1.2}{
\begin{tikzpicture}[<-,>=stealth',auto,node distance=3cm,main node/.style={circle,draw,font=\sffamily\Large\bfseries}, baseline=(current  bounding  box.center)]
  \node (1-1) at (-0.1,0) {${}_{\scalebox{0.6}{1}}$}; 
  \node (3-1) at (0.6,0) {${}_{\scalebox{0.6}{3}}$}; 
  \node (4-1) at (1.3,0) {${}_{\scalebox{0.6}{4}}$}; 
  \node (5-1) at (2,0) {${}_{\scalebox{0.6}{5}}$}; 
  \node (6-1) at (2.7,0) {${}_{\scalebox{0.6}{6}}$}; 
  
  \node at (1.3, -0.6) {${}_{\scalebox{0.6}{2}}$}; 
  \node (4-2) at (2, -0.6) {${}_{\scalebox{0.6}{4}}$}; 
  \node (5-2) at (2.7, -0.6) {${}_{\scalebox{0.6}{5}}$}; 
  
  \node (3-2) at (2, -1.3) {${}_{\scalebox{0.6}{3}}$}; 
  \node (4-3) at (2.7, -1.3) {${}_{\scalebox{0.6}{4}}$}; 
  \node at (3.4, -1.3) {${}_{\scalebox{0.6}{2}}$}; 
  
  \node (1-2) at (2, -2) {${}_{\scalebox{0.6}{1}}$}; 
  \node (3-3) at (2.7, -2) {${}_{\scalebox{0.6}{3}}$}; 
  \node (4-4) at (3.4, -2) {${}_{\scalebox{0.6}{4}}$}; 
  \node (5-3) at (4.1, -2) {${}_{\scalebox{0.6}{5}}$}; 
  \node (6-2) at (4.8, -2) {${}_{\scalebox{0.6}{6}}$}; 
  
    \path[every node/.style={font=\sffamily\small}]
	(1-1) edge[bend left=-25, dashed] node [right] {} (1-2) 
	
	(3-1) edge[bend left=-25, dashed] node [right] {} (3-2) 
	(3-2) edge[dashed] node [right] {} (3-3)
	
	(4-1) edge[dashed] node [right] {} (4-2)
	(4-2) edge[dashed] node [right] {} (4-3)
	(4-3) edge[dashed] node [right] {} (4-4)
	
	(5-1) edge[dashed] node [right] {} (5-2)
	(5-2) edge[bend left=25, dashed] node [right] {} (5-3) 
	
	(6-1) edge[bend left=25, dashed] node [right] {} (6-2) 
	;
\end{tikzpicture}
\hspace{-2cm}
\begin{tikzpicture}[<-,>=stealth',auto,node distance=3cm,main node/.style={circle,draw,font=\sffamily\Large\bfseries}, baseline=(current  bounding  box.center)]
  \node (7-1) at (-0.8,0) {${}_{7}$}; 
  \node (6-1) at (-0.1,0) {${}_{6}$}; 
  \node (5-1) at (0.6,0) {${}_{5}$}; 
  \node (4-1) at (1.3,0) {${}_{4}$}; 
  \node (3-1) at (2,0) {${}_{3}$}; 
  \node (1-1) at (2.7,0) {${}_{1}$}; 

  \node (2-1) at (1.3, -0.6) {${}_{2}$}; 
  \node (4-2) at (2, -0.6) {${}_{4}$}; 
  \node (3-2) at (2.7, -0.6) {${}_{3}$}; 

  \node (5-2) at (2, -1.3) {${}_{5}$}; 
  \node (4-3) at (2.7, -1.3) {${}_{4}$}; 
  \node (2-2) at (3.4, -1.3) {${}_{2}$}; 

  \node (6-2) at (2, -2) {${}_{6}$}; 
  \node (5-3) at (2.7, -2) {${}_{5}$}; 
  \node (4-4) at (3.4, -2) {${}_{4}$}; 
  \node (3-3) at (4.1, -2) {${}_{3}$}; 
  \node (1-2) at (4.8, -2) {${}_{1}$}; 

  \node (7-2) at (2, -2.7) {${}_{7}$}; 
  \node (6-3) at (2.7, -2.7) {${}_{6}$}; 
  \node (5-4) at (3.4, -2.7) {${}_{5}$}; 
  \node (4-5) at (4.1, -2.7) {${}_{4}$}; 
  \node (3-4) at (4.8, -2.7) {${}_{3}$}; 
  
  \node (2-3) at (4.1, -3.4) {${}_{2}$}; 
  \node (4-6) at (4.8, -3.4) {${}_{4}$}; 

  \node (5-5) at (4.8, -4.1) {${}_{5}$}; 

  \node (6-4) at (4.8, -4.8) {${}_{6}$}; 
  
  \node (7-3) at (4.8, -5.5) {${}_{7}$}; 

  \path[every node/.style={font=\sffamily\small}]
	(1-1) edge[bend left=30, dashed] node [right] {} (1-2)  
  	(2-1) edge[bend left=3, dotted, thick] node [right] {} (2-2)
	(2-2) edge[bend left=-3, dotted, thick] node [right] {} (2-3)
	(3-1) edge[dashed] node [right] {} (3-2)
	(3-2) edge[bend left=30, dashed] node [right] {} (3-3)
	(3-3) edge[dashed] node [right] {} (3-4)
	(4-1) edge[dashed] node [right] {} (4-2)
	(4-2) edge[dashed] node [right] {} (4-3)
	(4-3) edge[dashed] node [right] {} (4-4)
	(4-4) edge[dashed] node [right] {} (4-5)
	(4-5) edge[dashed] node [right] {} (4-6)
	(5-1) edge[bend right=30, dashed] node [left] {} (5-2)
	(5-2) edge[dashed] node [right] {} (5-3)
	(5-3) edge[dashed] node [right] {} (5-4)
	(5-4) edge[bend right=30, dashed] node [left] {} (5-5)
	(6-1) edge[bend right=30, dashed] node [left] {} (6-2)
	(6-2) edge[dashed] node [right] {} (6-3)
	(6-3) edge[bend right=30, dashed] node [left] {} (6-4)
	(7-1) edge[bend right=30, dashed] node [left] {} (7-2)
	(7-2) edge[bend right=30, dashed] node [left] {} (7-3)
	;
\end{tikzpicture}
}
\end{equation}

\begin{ex}
{\em 
Let $\pi$ be the ${\bf j}_0$-trail in Example \ref{ex:array in E7}
and let $\pi'$ be the ${\bf j}_0$-trail corresponding to the following sequence
\begin{equation*}
({\bf 7}, {\bf 6}, {\bf 5}, {\color{gray} 4}, {\color{gray}  2}, {\color{gray}  3}, {\bf  4}, {\color{gray}  5}, {\color{gray}  6}, {\color{gray}  7}, {\color{gray}  1}, {\color{gray}  3}, {\color{gray}  4}, {\color{gray}  5}, {\color{gray}  6}, {\color{gray}  2}, {\color{gray}  4}, {\color{gray}  5}, {\bf 3}, {\color{gray} 4}, {\bf 2}, {\color{gray} 1}, {\color{gray} 3}, {\bf 4}, {\bf 5}, {\bf 6}, {\bf 7}),
\end{equation*}
which is obtained from $\pi$ by moving ${\bf 2}$ to the right.
Then we have
\vskip 1mm
\begin{equation*}
\scalebox{0.8}{
\begin{tikzpicture}[<-,>=stealth',auto,node distance=3cm,main node/.style={circle,draw,font=\sffamily\Large\bfseries}, baseline=(current  bounding  box.center)]
  \node at (-3, -4) {\scalebox{1.2}{${\bf d}(\pi')$ = }};
  
  \node at (-1.4,0) {\scalebox{0.9}{$\bf 1$}}; 
  \node at (-0.1,0) {\scalebox{0.9}{$\bf 1$}}; 
  \node at (1.2,0) {\scalebox{0.9}{$\bf 1$}}; 
  \node at (2.5,0) {\scalebox{0.9}{$\bf 1$}}; 
  \node at (3.8,0) {\scalebox{0.8}{$\color{gray} 0$}}; 
  \node at (5.1,0) {\scalebox{0.8}{$\color{gray} 0$}}; 

  \node (2-1) at (2.5, -1) {\scalebox{0.9}{$\bf 1$}}; 
  \node at (3.8, -1) {\scalebox{0.8}{$\color{gray} 0$}}; 
  \node at (5.1, -1) {\scalebox{0.8}{$1$}}; 
  
  \node at (3.8, -2) {\scalebox{0.8}{$\color{gray} 0$}}; 
  \node at (5.1, -2) {\scalebox{0.8}{$\color{gray} 0$}}; 
  \node (2-2) at (6.4, -2) {\scalebox{0.8}{$\color{gray} 0$}}; 
  
  \node at (3.8, -3) {\scalebox{0.8}{$\color{gray} 0$}}; 
  \node at (5.1, -3) {\scalebox{0.8}{$\color{gray} 0$}}; 
  \node at (6.4, -3) {\scalebox{0.8}{$\color{gray} 0$}}; 
  \node at (7.7, -3) {\scalebox{0.8}{$\color{gray} 0$}}; 
  \node at (9, -3) {\scalebox{0.8}{$\color{gray} 0$}}; 
  \node at (3.8, -4) {\scalebox{0.8}{$\color{gray} 0$}}; 
  
  \node at (5.1, -4) {\scalebox{0.8}{$\color{gray} 0$}}; 
  \node at (6.4, -4) {\scalebox{0.8}{$\color{gray} 0$}}; 
  \node at (7.7, -4) {\scalebox{0.8}{$\bf 1$}}; 
  \node at (9, -4) {\scalebox{0.8}{$\color{gray} 0$}}; 
  \node at (7.7, -5) {\scalebox{0.8}{$\color{gray}  0$}}; 
  \node at (9, -5) {\scalebox{0.8}{$\color{gray} 0$}}; 
  \node at (9, -6) {\scalebox{0.9}{$\bf 1$}}; 
  \node at (9, -7) {\scalebox{0.9}{$\bf 1$}}; 
  \node at (9, -8) {\scalebox{0.9}{$\bf 1$}}; 
  
  \path[every node/.style={font=\sffamily\small}]
  (2-1) edge[bend left=3, dashed, thick] node [right] {} (2-2);
\end{tikzpicture}
}
\end{equation*}
}
\end{ex}
Under the enumeration of $d_k(\pi)$ on ${\bf d}(\pi)$, for ${\bf p} \in \mathcal{P}_n$, we define ${\bf d}({\bf p}) = (d_k) \in \mathcal{D}$ by 
\begin{equation*}
	d_k = 
	\begin{cases}
		0 & \text{${\bf p} \cap \Delta^{\J}$ passes the position of $d_k$}, \\
		1 & \text{otherwise}.
	\end{cases}
\end{equation*}

\begin{ex} \label{ex:P to D}
{\em 
Let us consider the following path ${\bf p} \in \mathcal{P}_7$ in Example \ref{ex:paths}.
\begin{equation*}
\begin{tikzpicture}[scale=1.1, baseline=(current  bounding  box.center)]
  \node at (-0.8,0) {${}_{\circ}$}; 
  \node at (-0.1,0) {${}_{\circ}$}; 
  \node at (0.6,0) {${}_{\circ}$}; 
  \node at (1.3,0) {${}_{\circ}$}; 
  \node at (2,0) {${}_{\circ}$}; 
  \node at (2.7,0) {${}_{\circ}$}; 
  \node at (3.4,0) {${}_{\circ}$}; 
  \node at (4.1,0) {${}_{\circ}$}; 
  \node at (4.8,0) {${}_{\circ}$}; 

  \node at (-0.1,-0.6) {${}_{\circ}$}; 
  \node at (0.6,-0.6) {${}_{\circ}$}; 
  \node at (1.3, -0.6) {${}_{\circ}$}; 
  \node at (2, -0.6) {${}_{\circ}$}; 
  \node at (2.7, -0.6) {${}_{\circ}$}; 
  \node at (3.4,-0.6) {${}_{\circ}$}; 
  \node at (4.1,-0.6) {${}_{\circ}$}; 
  \node at (4.8,-0.6) {${}_{\circ}$}; 

  \node at (0.6,-1.3) {${}_{\circ}$}; 
  \node at (1.3, -1.3) {${}_{\circ}$}; 
  \node at (2, -1.3) {${}_{\circ}$}; 
  \node at (2.7, -1.3) {${}_{\circ}$}; 
  \node at (3.4, -1.3) {${}_{\circ}$}; 
  \node at (4.1,-1.3) {${}_{\circ}$}; 
  \node at (4.8,-1.3) {${}_{\circ}$}; 

  \node at (1.3, -2) {${}_{\circ}$}; 
  \node at (2, -2) {${}_{\circ}$}; 
  \node at (2.7, -2) {${}_{\circ}$}; 
  \node at (3.4, -2) {${}_{\circ}$}; 
  \node at (4.1, -2) {${}_{\circ}$}; 
  \node at (4.8, -2) {${}_{\circ}$}; 

  \node at (2, -2.7) {${}_{\circ}$}; 
  \node at (2.7, -2.7) {${}_{\circ}$}; 
  \node at (3.4, -2.7) {${}_{\circ}$}; 
  \node at (4.1, -2.7) {${}_{\circ}$}; 
  \node at (4.8, -2.7) {${}_{\circ}$}; 
  
  \node at (2.7, -3.4) {${}_{\circ}$}; 
  \node at (3.4, -3.4) {${}_{\circ}$}; 
  \node at (4.1, -3.4) {${}_{\circ}$}; 
  \node at (4.8, -3.4) {${}_{\circ}$}; 
  
  \node at (3.4, -4.1) {${}_{\circ}$}; 
  \node at (4.1, -4.1) {${}_{\circ}$}; 
  \node at (4.8, -4.1) {${}_{\circ}$}; 
  
  \node at (4.1, -4.8) {${}_{\circ}$}; 
  \node at (4.8, -4.8) {${}_{\circ}$}; 
  
  \node at (4.8, -5.5) {${}_{\circ}$}; 
  
  \draw[-, line width=0.5mm] (4.2, 0) -- (4.7, 0);
  \draw[-, line width=0.5mm] (3.5, 0) -- (4, 0);
  \draw[-, line width=0.5mm] (2.8, 0) -- (3.3, 0);
  \draw[-, line width=0.5mm] (2.7, -0.1) -- (2.7, -0.5);
  \draw[-, line width=0.5mm] (2.1, -0.6) -- (2.6, -0.6);
  \draw[-, line width=0.5mm] (1.4, -0.6) -- (1.9, -0.6);
  \draw[-, line width=0.5mm] (1.3, -0.7) -- (1.3, -1.2);
  \draw[-, line width=0.5mm] (1.3, -1.4) -- (1.3, -1.9);
  
  \draw[-, line width=0.5mm] (4.1, -0.1) -- (4.1, -0.5);
  \draw[-, line width=0.5mm] (4.1, -0.7) -- (4.1, -1.2);
  \draw[-, line width=0.5mm] (4.1, -1.4) -- (4.1, -1.9);
  \draw[-, line width=0.5mm] (4.1, -2.1) -- (4.1, -2.6);
  \draw[-, line width=0.5mm] (4.1, -2.8) -- (4.1, -3.3);
  \draw[-, line width=0.5mm] (3.5, -3.4) -- (4, -3.4);
  \draw[-, line width=0.5mm] (3.4, -3.5) -- (3.4, -4);
  
  \draw[-, line width=0.5mm] (4.8, -0.1) -- (4.8, -0.5);
  \draw[-, line width=0.5mm] (4.8, -0.7) -- (4.8, -1.2);
  \draw[-, line width=0.5mm] (4.8, -1.4) -- (4.8, -1.9);
  \draw[-, line width=0.5mm] (4.8, -2.1) -- (4.8, -2.6);
  \draw[-, line width=0.5mm] (4.8, -2.8) -- (4.8, -3.3);
  \draw[-, line width=0.5mm] (4.8, -3.5) -- (4.8, -4);
  \draw[-, line width=0.5mm] (4.2, -4.1) -- (4.7, -4.1);
  \draw[-, line width=0.5mm] (4.1, -4.2) -- (4.1, -4.7);
  
   \draw[-, line width=0.5mm] (2.1, -1.3) -- (2.6, -1.3);
   \draw[-, line width=0.5mm] (2, -1.4) -- (2, -1.9);
   \draw[-, line width=0.5mm] (2, -2.1) -- (2, -2.6);
   \draw[-, line width=0.5mm] (2.7, -1.4) -- (2.7, -1.9);
   \draw[-, line width=0.5mm] (2.7, -2.1) -- (2.7, -2.6);
   \draw[-, line width=0.5mm] (2.7, -2.8) -- (2.7, -3.3);
\end{tikzpicture}
\end{equation*}
Then ${\bf p} \cap \Delta^{\J}$ and ${\bf d}({\bf p})$ are given by 
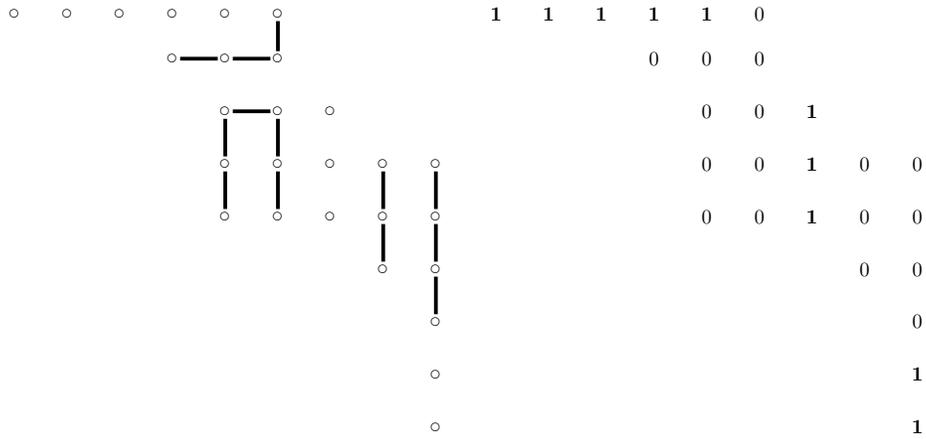
\begin{equation*}
\begin{tikzpicture}[baseline=(current  bounding  box.center)]
  \node at (-0.8,0) {${}_{\circ}$}; 
  \node at (-0.1,0) {${}_{\circ}$}; 
  \node at (0.6,0) {${}_{\circ}$}; 
  \node at (1.3,0) {${}_{\circ}$}; 
  \node at (2,0) {${}_{\circ}$}; 
  \node at (2.7,0) {${}_{\circ}$}; 

  \node at (1.3, -0.6) {${}_{\circ}$}; 
  \node at (2, -0.6) {${}_{\circ}$}; 
  \node at (2.7, -0.6) {${}_{\circ}$}; 

  \node at (2, -1.3) {${}_{\circ}$}; 
  \node at (2.7, -1.3) {${}_{\circ}$}; 
  \node at (3.4, -1.3) {${}_{\circ}$}; 

  \node at (2, -2) {${}_{\circ}$}; 
  \node at (2.7, -2) {${}_{\circ}$}; 
  \node at (3.4, -2) {${}_{\circ}$}; 
  \node at (4.1, -2) {${}_{\circ}$}; 
  \node at (4.8, -2) {${}_{\circ}$}; 

  \node at (2, -2.7) {${}_{\circ}$}; 
  \node at (2.7, -2.7) {${}_{\circ}$}; 
  \node at (3.4, -2.7) {${}_{\circ}$}; 
  \node at (4.1, -2.7) {${}_{\circ}$}; 
  \node at (4.8, -2.7) {${}_{\circ}$}; 
  
  \node at (4.1, -3.4) {${}_{\circ}$}; 
  \node at (4.8, -3.4) {${}_{\circ}$}; 

  \node at (4.8, -4.1) {${}_{\circ}$}; 

  \node at (4.8, -4.8) {${}_{\circ}$}; 
  
  \node at (4.8, -5.5) {${}_{\circ}$}; 
  
  \draw[-, line width=0.5mm] (2.7, -0.1) -- (2.7, -0.5);
  \draw[-, line width=0.5mm] (2.1, -0.6) -- (2.6, -0.6);
  \draw[-, line width=0.5mm] (1.4, -0.6) -- (1.9, -0.6);

  \draw[-, line width=0.5mm] (4.1, -2.1) -- (4.1, -2.6);
  \draw[-, line width=0.5mm] (4.1, -2.8) -- (4.1, -3.3);
  
  \draw[-, line width=0.5mm] (4.8, -2.1) -- (4.8, -2.6);
  \draw[-, line width=0.5mm] (4.8, -2.8) -- (4.8, -3.3);
  \draw[-, line width=0.5mm] (4.8, -3.5) -- (4.8, -4);
  
   \draw[-, line width=0.5mm] (2.1, -1.3) -- (2.6, -1.3);
   \draw[-, line width=0.5mm] (2, -1.4) -- (2, -1.9);
   \draw[-, line width=0.5mm] (2, -2.1) -- (2, -2.6);
   \draw[-, line width=0.5mm] (2.7, -1.4) -- (2.7, -1.9);
   \draw[-, line width=0.5mm] (2.7, -2.1) -- (2.7, -2.6);
\end{tikzpicture}
\quad
\begin{tikzpicture}[baseline=(current  bounding  box.center)]
  \node at (-0.8,0) {\scalebox{0.7}{$\bf 1$}}; 
  \node at (-0.1,0) {\scalebox{0.7}{$\bf 1$}}; 
  \node at (0.6,0) {\scalebox{0.7}{$\bf 1$}}; 
  \node at (1.3,0) {\scalebox{0.7}{$\bf 1$}}; 
  \node at (2,0) {\scalebox{0.7}{$\bf 1$}}; 
  \node at (2.7,0) {\scalebox{0.7}{$0$}}; 

  \node at (1.3, -0.6) {\scalebox{0.7}{$0$}}; 
  \node at (2, -0.6) {\scalebox{0.7}{$0$}}; 
  \node at (2.7, -0.6) {\scalebox{0.7}{$0$}}; 

  \node at (2, -1.3) {\scalebox{0.7}{$0$}}; 
  \node at (2.7, -1.3) {\scalebox{0.7}{$0$}}; 
  \node at (3.4, -1.3) {\scalebox{0.7}{$\bf 1$}}; 

  \node at (2, -2) {\scalebox{0.7}{$0$}}; 
  \node at (2.7, -2) {\scalebox{0.7}{$0$}}; 
  \node at (3.4, -2) {\scalebox{0.7}{$\bf 1$}}; 
  \node at (4.1, -2) {\scalebox{0.7}{$0$}}; 
  \node at (4.8, -2) {\scalebox{0.7}{$0$}}; 

  \node at (2, -2.7) {\scalebox{0.7}{$0$}}; 
  \node at (2.7, -2.7) {\scalebox{0.7}{$0$}}; 
  \node at (3.4, -2.7) {\scalebox{0.7}{$\bf 1$}}; 
  \node at (4.1, -2.7) {\scalebox{0.7}{$0$}}; 
  \node at (4.8, -2.7) {\scalebox{0.7}{$0$}}; 
  
  \node at (4.1, -3.4) {\scalebox{0.7}{$0$}}; 
  \node at (4.8, -3.4) {\scalebox{0.7}{$0$}}; 

  \node at (4.8, -4.1) {\scalebox{0.7}{$0$}}; 

  \node at (4.8, -4.8) {\scalebox{0.7}{$\bf 1$}}; 
  
  \node at (4.8, -5.5) {\scalebox{0.7}{$\bf 1$}}; 
\end{tikzpicture}
\end{equation*} 
Note that the above array consisting of $0$'s and $1$'s corresponds to the sequence
\begin{equation*}
({\bf 7}, {\bf 6}, {\color{gray}  5}, {\color{gray} 4}, {\color{gray}  2}, {\color{gray}  3}, {\color{gray}   4}, {\bf  5}, {\color{gray}  6}, {\color{gray}  7}, {\color{gray}  1}, {\color{gray}  3}, {\bf  4}, {\color{gray}  5}, {\color{gray}  6}, {\bf 2}, {\color{gray}  4}, {\color{gray}  5}, {\color{gray}  3}, {\color{gray} 4}, {\color{gray} 2}, {\color{gray} 1}, {\bf 3}, {\bf 4}, {\bf 5}, {\bf 6}, {\bf 7}).
\end{equation*}
}
\end{ex}
\vskip 2mm

\begin{lem} \label{lem:from Pr to DJ}
{\em 
The map
\begin{equation*}
\xymatrixcolsep{3pc}\xymatrixrowsep{0pc}
\Psi : \xymatrix{
  \mathcal{P}_n \ \ar@{->}[r] & \ \mathcal{D}^{\J} \\
 {\bf p} \ar@{|->}[r] & {\bf d}({\bf p})
}
\end{equation*}
is well-defined and surjective. For ${\bf c} \in \B^{\J}$ and ${\bf p},\, {\bf p}' \in \Psi^{-1}({\bf d})$, we have
\begin{equation*}
	||{\bf c}||_{\bf p} = ||{\bf c}||_{{\bf p}'}.
\end{equation*}
}
\end{lem}
\begin{proof}
We state the proof for type $\text{E}_7$ with $r = 7$. 
The proof for type $\text{E}_6$ is almost identical. 
Note that the key idea of the proof here is same with \cite[Lemma 5.8]{JK19}.

Let ${\bf p} \in \mathcal{P}_7$ be given.
To prove ${\bf d}({\bf p}) \in \mathcal{D}^{\J}$, we introduce an operator on $\mathcal{P}_7$.
To do this, let us consider the array with the dashed arrows in \eqref{eq:dashed arrows}. 
We denote by ${\bf p}^c$ the set of dots in $\Delta_{10}$ outside ${\bf p}$.
For ${\bf p}, {\bf p}' \in \mathcal{P}_7$, we define
\begin{equation*}
	{\bf p} \cong {\bf p}' \, \Longleftrightarrow \, {\bf d}({\bf p}) = {\bf d}({\bf p}'),
\end{equation*}
and let $\overline{\mathcal{P}_7}$ be the set of the equivalent classes of the paths ${\bf p} \in \mathcal{P}_7$ with respect to the relation $\cong$.
Then we define the operator $\texttt{p}_i$ by
\begin{equation*}
	\texttt{p}_i : 
	\xymatrixcolsep{3pc}\xymatrixrowsep{0pc}
\xymatrix{
  \mathcal{P}_7 \ \ar@{->}[r] &  \overline{\mathcal{P}_7} \cup \{ \, {\bf 0} \, \}  \\
 {\bf p} \ar@{|->}[r] & \texttt{p}_i ({\bf p})
}	
\end{equation*}
where $\texttt{p}_i ({\bf p})$ is the equivalent class of the paths in $\mathcal{P}_7$ determined by changing the dot in ${\bf p}^{c} \cap \Delta^{\J}$ located at the position of the leftmost $i$ on the array in \eqref{eq:dashed arrows} along the dashed arrow between $i$'s by one position if it is possible, otherwise we assume $\texttt{p}_i ({\bf p}) = {\bf 0}$.
For ${\bf p}' \in \texttt{p}_i({\bf p})\neq {\bf 0}$, write ${\bf p} \, \overset{\texttt{p}_i}{\longrightarrow} \, {\bf p}'$.

For example, if a path ${\bf p} \in \mathcal{P}_7$ has the black-filled dot outside ${\bf p}$ located at the position of $2$ in the third row of \eqref{eq:dashed arrows}, then $\texttt{p}_2({\bf p}) \neq {\bf 0}$ and it is a path in $\mathcal{P}_7$ determined by interchanging the black-filled dot with the dot located at the position of $2$ in the second row of \eqref{eq:dashed arrows} along the dashed arrow between $2$'s.
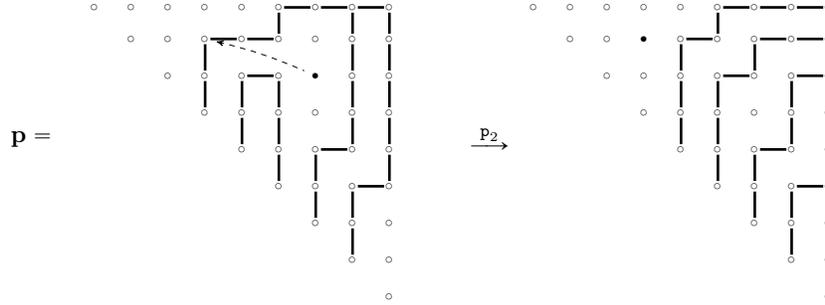
\begin{equation*}
\scalebox{0.7}{
\begin{tikzpicture}[<-,>=stealth',auto,node distance=3cm,main node/.style={circle,draw,font=\sffamily\Large\bfseries}]
  \node at (-0.8,0) {${}_{\circ}$}; 
  \node at (-0.1,0) {${}_{\circ}$}; 
  \node at (0.6,0) {${}_{\circ}$}; 
  \node at (1.3,0) {${}_{\circ}$}; 
  \node at (2,0) {${}_{\circ}$}; 
  \node at (2.7,0) {${}_{\circ}$}; 
  \node at (3.4,0) {${}_{\circ}$}; 
  \node at (4.1,0) {${}_{\circ}$}; 
  \node at (4.8,0) {${}_{\circ}$}; 

  \node at (-0.1,-0.6) {${}_{\circ}$}; 
  \node at (0.6,-0.6) {${}_{\circ}$}; 
  \node (2-1) at (1.3, -0.6) {${}_{\circ}$}; 
  \node at (2, -0.6) {${}_{\circ}$}; 
  \node at (2.7, -0.6) {${}_{\circ}$}; 
  \node at (3.4,-0.6) {${}_{\circ}$}; 
  \node at (4.1,-0.6) {${}_{\circ}$}; 
  \node at (4.8,-0.6) {${}_{\circ}$}; 

  \node at (0.6,-1.3) {${}_{\circ}$}; 
  \node at (1.3, -1.3) {${}_{\circ}$}; 
  \node at (2, -1.3) {${}_{\circ}$}; 
  \node at (2.7, -1.3) {${}_{\circ}$}; 
  \node (2-2) at (3.4, -1.3) {${}_{\bullet}$}; 
  \node at (4.1,-1.3) {${}_{\circ}$}; 
  \node at (4.8,-1.3) {${}_{\circ}$}; 

  \node at (1.3, -2) {${}_{\circ}$}; 
  \node at (2, -2) {${}_{\circ}$}; 
  \node at (2.7, -2) {${}_{\circ}$}; 
  \node at (3.4, -2) {${}_{\circ}$}; 
  \node at (4.1, -2) {${}_{\circ}$}; 
  \node at (4.8, -2) {${}_{\circ}$}; 

  \node at (2, -2.7) {${}_{\circ}$}; 
  \node at (2.7, -2.7) {${}_{\circ}$}; 
  \node at (3.4, -2.7) {${}_{\circ}$}; 
  \node at (4.1, -2.7) {${}_{\circ}$}; 
  \node at (4.8, -2.7) {${}_{\circ}$}; 
  
  \node at (2.7, -3.4) {${}_{\circ}$}; 
  \node at (3.4, -3.4) {${}_{\circ}$}; 
  \node at (4.1, -3.4) {${}_{\circ}$}; 
  \node at (4.8, -3.4) {${}_{\circ}$}; 
  
  \node at (3.4, -4.1) {${}_{\circ}$}; 
  \node at (4.1, -4.1) {${}_{\circ}$}; 
  \node at (4.8, -4.1) {${}_{\circ}$}; 
  
  \node at (4.1, -4.8) {${}_{\circ}$}; 
  \node at (4.8, -4.8) {${}_{\circ}$}; 
  
  \node at (4.8, -5.5) {${}_{\circ}$}; 

    \node at (-2, -2.5) {\scalebox{1.2}{${\bf p}=$}};
      \node at (6.7, -2.5) {\scalebox{1.2}{$\overset{\texttt{p}_2}{\longrightarrow}$}};
  
  \draw[-, line width=0.5mm] (4.2, 0) -- (4.7, 0);
  \draw[-, line width=0.5mm] (3.5, 0) -- (4, 0);
  \draw[-, line width=0.5mm] (2.8, 0) -- (3.3, 0);
  \draw[-, line width=0.5mm] (2.7, -0.1) -- (2.7, -0.5);
  \draw[-, line width=0.5mm] (2.1, -0.6) -- (2.6, -0.6);
  \draw[-, line width=0.5mm] (1.4, -0.6) -- (1.9, -0.6);
  \draw[-, line width=0.5mm] (1.3, -0.7) -- (1.3, -1.2);
  \draw[-, line width=0.5mm] (1.3, -1.4) -- (1.3, -1.9);
  
  \draw[-, line width=0.5mm] (4.1, -0.1) -- (4.1, -0.5);
  \draw[-, line width=0.5mm] (4.1, -0.7) -- (4.1, -1.2);
  \draw[-, line width=0.5mm] (4.1, -1.4) -- (4.1, -1.9);
  \draw[-, line width=0.5mm] (4.1, -2.1) -- (4.1, -2.6);
  \draw[-, line width=0.5mm] (3.5, -2.7) -- (4, -2.7);
  \draw[-, line width=0.5mm] (3.4, -2.8) -- (3.4, -3.3);
  \draw[-, line width=0.5mm] (3.4, -3.5) -- (3.4, -4);
  
  \draw[-, line width=0.5mm] (4.8, -0.1) -- (4.8, -0.5);
  \draw[-, line width=0.5mm] (4.8, -0.7) -- (4.8, -1.2);
  \draw[-, line width=0.5mm] (4.8, -1.4) -- (4.8, -1.9);
  \draw[-, line width=0.5mm] (4.8, -2.1) -- (4.8, -2.6);
  \draw[-, line width=0.5mm] (4.8, -2.8) -- (4.8, -3.3);
  \draw[-, line width=0.5mm] (4.2, -3.4) -- (4.7, -3.4);
  \draw[-, line width=0.5mm] (4.1, -3.5) -- (4.1, -4);
  \draw[-, line width=0.5mm] (4.1, -4.2) -- (4.1, -4.7); 
  
   \draw[-, line width=0.5mm] (2.1, -1.3) -- (2.6, -1.3);
   \draw[-, line width=0.5mm] (2, -1.4) -- (2, -1.9);
   \draw[-, line width=0.5mm] (2, -2.1) -- (2, -2.6);
   \draw[-, line width=0.5mm] (2.7, -1.4) -- (2.7, -1.9);
   \draw[-, line width=0.5mm] (2.7, -2.1) -- (2.7, -2.6);
   \draw[-, line width=0.5mm] (2.7, -2.8) -- (2.7, -3.3);

    \path[every node/.style={font=\sffamily\small}]
	(2-1) edge[bend left=5, dashed] node [right] {} (2-2);
\end{tikzpicture}}
\scalebox{0.7}{
\begin{tikzpicture}
  \node at (-0.8,0) {${}_{\circ}$}; 
  \node at (-0.1,0) {${}_{\circ}$}; 
  \node at (0.6,0) {${}_{\circ}$}; 
  \node at (1.3,0) {${}_{\circ}$}; 
  \node at (2,0) {${}_{\circ}$}; 
  \node at (2.7,0) {${}_{\circ}$}; 
  \node at (3.4,0) {${}_{\circ}$}; 
  \node at (4.1,0) {${}_{\circ}$}; 
  \node at (4.8,0) {${}_{\circ}$}; 

  \node at (-0.1,-0.6) {${}_{\circ}$}; 
  \node at (0.6,-0.6) {${}_{\circ}$}; 
  \node at (1.3, -0.6) {${}_{\bullet}$}; 
  \node at (2, -0.6) {${}_{\circ}$}; 
  \node at (2.7, -0.6) {${}_{\circ}$}; 
  \node at (3.4,-0.6) {${}_{\circ}$}; 
  \node at (4.1,-0.6) {${}_{\circ}$}; 
  \node at (4.8,-0.6) {${}_{\circ}$}; 

  \node at (0.6,-1.3) {${}_{\circ}$}; 
  \node at (1.3, -1.3) {${}_{\circ}$}; 
  \node at (2, -1.3) {${}_{\circ}$}; 
  \node at (2.7, -1.3) {${}_{\circ}$}; 
  \node at (3.4, -1.3) {${}_{\circ}$}; 
  \node at (4.1,-1.3) {${}_{\circ}$}; 
  \node at (4.8,-1.3) {${}_{\circ}$}; 

  \node at (1.3, -2) {${}_{\circ}$}; 
  \node at (2, -2) {${}_{\circ}$}; 
  \node at (2.7, -2) {${}_{\circ}$}; 
  \node at (3.4, -2) {${}_{\circ}$}; 
  \node at (4.1, -2) {${}_{\circ}$}; 
  \node at (4.8, -2) {${}_{\circ}$}; 
  
  \node at (2, -2.7) {${}_{\circ}$}; 
  \node at (2.7, -2.7) {${}_{\circ}$}; 
  \node at (3.4, -2.7) {${}_{\circ}$}; 
  \node at (4.1, -2.7) {${}_{\circ}$}; 
  \node at (4.8, -2.7) {${}_{\circ}$}; 
  
  \node at (2.7, -3.4) {${}_{\circ}$}; 
  \node at (3.4, -3.4) {${}_{\circ}$}; 
  \node at (4.1, -3.4) {${}_{\circ}$}; 
  \node at (4.8, -3.4) {${}_{\circ}$}; 
  
  \node at (3.4, -4.1) {${}_{\circ}$}; 
  \node at (4.1, -4.1) {${}_{\circ}$}; 
  \node at (4.8, -4.1) {${}_{\circ}$}; 
  
  \node at (4.1, -4.8) {${}_{\circ}$}; 
  \node at (4.8, -4.8) {${}_{\circ}$}; 
  
  \node at (4.8, -5.5) {${}_{\circ}$}; 
  
  \draw[-, line width=0.5mm] (4.2, 0) -- (4.7, 0);
  \draw[-, line width=0.5mm] (3.5, 0) -- (4, 0);
  \draw[-, line width=0.5mm] (2.8, 0) -- (3.3, 0);
  \draw[-, line width=0.5mm] (2.7, -0.1) -- (2.7, -0.5);
  \draw[-, line width=0.5mm] (2.1, -0.6) -- (2.6, -0.6);
  \draw[-, line width=0.5mm] (2, -0.7) -- (2, -1.2);
  \draw[-, line width=0.5mm] (2, -1.4) -- (2, -1.9);
  \draw[-, line width=0.5mm] (2, -2.1) -- (2, -2.6);

  \draw[-, line width=0.5mm] (4.8, -0.1) -- (4.8, -0.5);
  \draw[-, line width=0.5mm] (4.2, -0.6) -- (4.7, -0.6);
  \draw[-, line width=0.5mm] (3.5, -0.6) -- (4, -0.6);
  \draw[-, line width=0.5mm] (3.4, -0.7) -- (3.4, -1.2);
  \draw[-, line width=0.5mm] (2.8, -1.3) -- (3.3, -1.3);
  \draw[-, line width=0.5mm] (2.7, -1.4) -- (2.7, -1.9);
  \draw[-, line width=0.5mm] (2.7, -2.1) -- (2.7, -2.6);
  \draw[-, line width=0.5mm] (2.7, -2.8) -- (2.7, -3.3);
  
  \draw[-, line width=0.5mm] (4.8, -0.7) -- (4.8, -1.2);
  \draw[-, line width=0.5mm] (4.2, -1.3) -- (4.7, -1.3);
  \draw[-, line width=0.5mm] (4.1, -1.4) -- (4.1, -1.9);
  \draw[-, line width=0.5mm] (4.1, -2.1) -- (4.1, -2.6);
  \draw[-, line width=0.5mm] (3.5, -2.7) -- (4, -2.7);
  \draw[-, line width=0.5mm] (3.4, -2.8) -- (3.4, -3.3); 
  \draw[-, line width=0.5mm] (3.4, -3.5) -- (3.4, -4); 
  
  \draw[-, line width=0.5mm] (4.8, -1.4) -- (4.8, -1.9);
  \draw[-, line width=0.5mm] (4.8, -2.1) -- (4.8, -2.6);
  \draw[-, line width=0.5mm] (4.8, -2.8) -- (4.8, -3.3);
  \draw[-, line width=0.5mm] (4.2, -3.4) -- (4.7, -3.4);
  \draw[-, line width=0.5mm] (4.1, -3.5) -- (4.1, -4);
  \draw[-, line width=0.5mm] (4.1, -4.2) -- (4.1, -4.7); 
\end{tikzpicture}
}
\end{equation*}
where the path in the right-hand side is a representative in $\texttt{p}_2({\bf p})$.

It is clear that there exists the unique quadruple path ${\bf p}_0'$ in $\mathcal{P}_7$ such that ${\bf d}({\bf p}_0') = {\bf d}(\pi_0')$ (see Remark \ref{rem:connectedness}).
By the configurations of the paths in Definition \ref{df:paths in Pn}(2), one can check that $\texttt{p}_i$ is well-defined and for any ${\bf p} \neq {\bf p}_0'$, there exists $i$ such that $\texttt{p}_i({\bf p}) \neq {\bf 0}$. 
This implies that ${\bf p}_0'$ is obtained from ${\bf p}$ by applying $\texttt{p}_i$'s successively with finitely many steps, that is, there exists sequences $(i_2, \dots, i_m)$ and $({\bf p}_m, \dots, {\bf p}_1)$  such that ${\bf p}_k \in \mathcal{P}_7$ $(1 \le k \le m)$ and 
\begin{equation*}
	{\bf p} = {\bf p}_m \, \overset{\texttt{p}_{i_m}}{\longrightarrow} \, {\bf p}_{m-1} 
	\, \overset{\texttt{p}_{i_{m-1}}}{\longrightarrow} \,	\cdots 
	\, \overset{\texttt{p}_{i_2}}{\longrightarrow} \, 
 	{\bf p}_1 = {\bf p}_0'.
\end{equation*}
Then it allows us to verify that ${\bf d}({\bf p})$ is in $\mathcal{D}^{\J}$ for ${\bf p} \in \mathcal{P}_7$ by an inductive argument on $m$.
More precisely, if ${\bf d}({\bf p}_{m-1}) \in \mathcal{D}^{\J}$, then ${\bf d}({\bf p}_{m-1})$ is one of the ${\bf j}_0$-trails in Section \ref{subsec:j0-trails}. Since ${\bf p}_m \, \overset{\texttt{p}_{i_m}}{\longrightarrow} \, {\bf p}_{m-1}$, the array ${\bf d}({\bf p}_{m-1})$ is obtained from ${\bf d}({\bf p}_m)$ by moving $1$ on ${\bf d}({\bf p}_m)$ located at the position of $i_m$ along the dashed arrow between $i_m$'s in \eqref{eq:dashed arrows}.
Then it is straightforward to check that if ${\bf d}({\bf p}_m) \notin \mathcal{D}^{\J}$, then ${\bf p}_m$ cannot be a path in $\mathcal{P}_7$, which is a contradiction.

By Remark \ref{rem:connectedness}, for any ${\bf d}(\pi)$, the array ${\bf d}(\pi_0') = {\bf d}({\bf p}_0')$ is obtained from ${\bf d}(\pi)$ by moving $1$'s successively following \eqref{eq:dashed arrows} in which each step corresponds to $\texttt{p}_i$ for some $i$. Hence the surjectivity of $\Psi$ is proved.

Finally, for ${\bf p} \in \Psi^{-1}({\bf d})$ and ${\bf c} = (c_{\beta_k}) \in \B^{\J}$, since $||{\bf c}||_{\bf p}$ is equal to the sum of $c_{\beta_k}$'s located at the position of $0$'s in ${\bf d}$, the last statement is proved.
\end{proof}

\begin{rem}
{\em
We would like to give a technical remark related to paths in type $\text{E}_{6,7}$.
Unlike type $\text{E}_6$ (cf. Remark \ref{rem:remark for analog for type E}), the subsequence \eqref{eq:minimal seq} of ${\bf j}_0$-trails $\pi$ at which $||c||_\pi$ is maximal contains a pair of commuting letters $2$ and $3$.  Due to this difference, the combinatorics of trails for type $\text{E}_7$ is more involved (cf. Definition \ref{df:paths in Pn}).
}
\end{rem}
\vskip 2mm

\subsubsection{Proof of Theorem \ref{thm:formula of epsilon star}}
By Lemma \ref{lem:from Pr to DJ}, there exists a surjection from $\mathcal{P}_n$ to $\mathcal{T}^{\,'}$. Then for $\pi \in \mathcal{T}^{\,'}$, we have
\begin{equation*}
	||{\bf c}||_{\pi} = ||{\bf c}||_{\bf p},
\end{equation*}
where ${\bf p}$ is the corresponding path in $\mathcal{P}_n$ under the surjection.
Hence, by Lemma \ref{lem:reduction}, we obtain the desired formula of $\varepsilon_r^*({\bf c})$.
\qed


\begin{thebibliography}{HK}
{\small
\bibitem{BS}
R. Biswal, T. Scrimshaw, {\em Kirillov-Reshetikhin crystals $B^{7, s}$ for type $E_7^{(1)}$}, Comm. Algebra. (2021) 1--16.

\bibitem{BZ}
A. Berenstein, A. Zelevinsky, {\em Tensor product multiplicities, canonical bases and totally positive varieties}, Invent. math. {\bf 143} (2001) 77--128.

\bibitem{Bur}
W.H. Burge, {\em Four correspondences between graphs and generalized Young tableaux}, J. Combin. Theory Ser. A {\bf 17} (1974) 12--30.

\bibitem{Bo}
N. Bourbaki, {\em Lie groups and Lie algebras, Chapter 4--6}, Springer-Verlag Berlin, 2002.

\bibitem{Bo2}
N. Bourbaki, {\em Lie groups and Lie algebras, Chapters 7--9}, Springer-Verlag, Berlin, 2005.

\bibitem{Cha01}
V. Chari, {\em On the fermionic formula and the Kirillov-Reshetikhin conjecture}, Int. Math. Res. Not. (2001) 629--654.

\bibitem{CP95}
V. Chari, A. Pressley, {\em Quantum affine algebras and their representations}, in: Representations of Groups, in: CMS Conf. Proc., vol. {\bf 16}, Amer. Math. Soc., Providence, RI, 1995, pp. 59--78.


\bibitem{CH}
V. Chari, D. Hernandez, {\em Beyond Kirillov-Reshetikhin modules}, Contemp. Math. {\bf 506} (2010) 49--81.

\bibitem{FOS}
G. Fourier, M. Okado, A. Schilling, {\em Kirillov-Reshetikhin crystals for nonexceptional types}, Adv. in Math. {\bf 222} (2009) 1080--1116.

\bibitem{Ful}
W. Fulton, {\em Young tableaux, with Application to Representation theory and Geometry}, Cambridge Univ. Press, 1997.

\bibitem{G}
C. Greene, {\em An extension of Schensted's theorem}, Adv. Math, \textbf{14} (1974) 254--265.


\bibitem{HKOTY99}
G. Hatayama, A. Kuniba, M. Okado, T. Takagi, Y. Yamada, {\em Remarks on fermionic formula}, Contemp. Math. {\bf 248} (1999) 243--291.


\bibitem{Hiro}
T. Hiroshima, {\em Perfectness of Kirillov-Reshetikhin crystals $B^{r,s}$ for types $E_6^{(1)}$ and $E_7^{(1)}$ with a minuscule node $r$}, preprint (2021), arXiv:2107.08614.

\bibitem{HK}
J. Hong, S.-J. Kang, {\em Introduction to Quantum Groups and Crystal Bases}, Graduate Studies in Mathematics {\bf 42},  Amer. Math. Soc., 2002.


\bibitem{Hum}
J. E. Humphreys, {\em Reﬂection Groups and Coxeter Groups}, Cambridge University Press, Cambridge, 1990.

\bibitem{JK19}
I.-S. Jang, J.-H. Kwon, {\em Quantum nilpotent subalgebra of classical quantum groups and affine crystals}, J. Combin. Theory Ser. A {\bf 168} (2019) 219--254.

\bibitem{JK20}
I.-S. Jang, J.-H. Kwon, {\em Lusztig data of Kashiwara-Nakashima tableaux in type $D$}, Algebr. Represent. Theor. {\bf 24} (2021), no. 4, 959--989. 

\bibitem{JS}
B. Jones, A. Schilling, {\em Affine structures and a tableau model for $E_6$ crystals}, J. Algebra {\bf 324} (2010) 2512--2542.

\bibitem{K} V.~Kac, {\em Infinite-dimensional Lie algebras}, Third edition, Cambridge Univ. Press, 1990.

\bibitem{Kas91}
M. Kashiwara, {\em On crystal bases of the q-analogue of universal enveloping algebras}, Duke Math. J. {\bf 63} (1991) 465--516.


\bibitem{Kas95}
M. Kashiwara, {\em On crystal bases}, Representations of groups, CMS Conf. Proc., vol. 16, Amer. Math. Soc., Providence, RI, (1995) 155--197.

\bibitem{KKMMNN92a}
S.-J Kang, M. Kashiwara, K. C. Misra, T. Miwa, T. Nakashima, A. Nakayashiki, {\em Affine crystals and vertex models}, In Infinite analysis, Part A, B (Kyoto, 1991), volume 16 of Adv. Ser. Math. Phys., pages 449–-484. World Sci. Publ., River Edge, NJ, 1992.

\bibitem{KKMMNN92b}
S.-J Kang, M. Kashiwara, K. C. Misra, T. Miwa, T. Nakashima, A. Nakayashiki, {\em  Perfect crystals of quantum affine Lie algebras}, Duke Math. J. {\bf 68} (1992) 499--607.

\bibitem{Ki12}
Y. Kimura, {\em Quantum unipotent subgroup and dual canonical basis}, Kyoto J. Math. {\bf 52} (2012) 277--331.


\bibitem{Knu}
D. E. Knuth, {\em Permutations, matrices and generalized Young tableaux}, Paciﬁc J. Math. {\bf 34} (1970) 709--729.

\bibitem{Kw13}
J.-H. Kwon, {\em RSK correspondence and classically irreducible Kirillov-Reshetikhin crystals}, J. Combin. Theory Ser. A {\bf 120} (2013) 433--452.

\bibitem{Kw18}
J.-H. Kwon, {\em A crystal embedding into Lusztig data of type $A$}, J. Combin. Theory Ser. A {\bf 154} (2018) 422--443.


\bibitem{LNSSS}
C. Lenart, S. Naito, D. Sagaki, A. Schilling, M. Shimozono, {\em A uniform model for Kirillov-Reshetikhin crystals II. Alcove model, path model, and $P=X$}, Int. Math. Res. Not. IMRN 2017, no. 14, 4259--4319.

\bibitem{LL}
C. Lenart, A. Lubovsky,  {\em A generalization of the alcove model and its applications}, J. Algebraic Combin. {\bf 41} (2015), no. 3, 751--783.

\bibitem{LT}
C. Lenart, T. Scrimshaw, {\em On higher level Kirillov-Reshetikhin crystals, Demazure crystals, and related uniform models}, J. Algebra {\bf 539} (2019), 285--304.




\bibitem{Li98}
P. Littelmann, {\em Cones, crystals, and patterns}, Transform. Groups {\bf 3} (1998), no. 2, 145--179.

\bibitem{Lu90}
G. Lusztig, {\em Canonical bases arising from quantized enveloping algebras}. J. Amer. Math. Soc. {\bf 3} (1990) 447--498.

\bibitem{Lu90-2}
G. Lusztig, {\em Canonical bases arising from quantized enveloping algebras II}, Progr. Theoret. Phys. Suppl. {\bf 102} (1990) 175--201.

\bibitem{Lu10}
G. Lusztig, {\em Introduction to quantum groups}, Progr. Math. Vol. {\bf 110}, Birkh\"auser, 2010.


\bibitem{NZ}
T. Nakashima, A. Zelevinsky, {\em Polyhedral realizations of crystal bases for quantized Kac-Moody algebras},  Adv. Math. {\bf 131} (1997), no. 1, 253--278.

\bibitem{Nao18}
K. Naoi, {\em Existence of Kirillov-Reshetikhin crystals of type $G_2^{(1)}$ and $D_4^{(3)}$}, J. Algebra {\bf 512} (2018) 47–-65.


\bibitem{NaoScr21}
K. Naoi, T. Scrimshaw, {\em Existence of Kirillov-Reshetikhin crystals for near adjoint nodes in exceptional types}, J. Pure Appl. Algebra {\bf 225} (2021) no. 5, 106593, 38 pp. 


\bibitem{OS08}
M. Okado, A. Schilling, {\em Existence of Kirillov–Reshetikhin crystals for nonexceptional types}, Represent. Theory {\bf 12} (2008) 186--207.
 


\bibitem{Sai}
Y. Saito, {\em PBW basis of quantized universal enveloping algebras}, Publ. Res. Inst. Math. Sci. {\bf 30} (1994) 209--232.

\bibitem{SST}
B. Salisbury, A. Schultze, P. Tingley, {\em Combinatorial descriptions of the crystal structure on certain PBW bases}, Transform. Groups {\bf 23} (2018) 501--525.


\bibitem{Shi}
M. Shimozono, {\em Affine type $A$ crystal structure on tensor products of rectangles, Demazure characters, and nilpotent varieties}, J. Algebraic Combin. {\bf 15} (2002) 151--187.

\bibitem{Ste}
J. Stembridge, {\em On the fully commutative elements of Coxeter groups}, J. Algebraic Combin. {\bf 5} (1996), no. 4, 353--385. 

\bibitem{Pa}
P. Papi, {\em A characterization of a special ordering in a root system}, Proc. Amer. Math. Soc. {\bf 120} (1994) 661--665.
}
\end{thebibliography}
\end{document}